\documentclass[10pt,reqno]{amsart}
\usepackage{amssymb,mathrsfs,color}
\allowdisplaybreaks
\usepackage{enumerate}
\usepackage{graphicx}
\usepackage{xcolor} 
\usepackage{geometry}
\usepackage{amsfonts,amssymb,amsmath,mathtools}
\usepackage{slashed}
\usepackage{tikz}
\usepackage{ mathrsfs }
\usepackage[titletoc,title]{appendix}
\usepackage{wrapfig}

\usepackage{cancel}

\usepackage{cite}

\definecolor{green}{rgb}{0,0.8,0} 

\renewcommand{\Re}{\mathrm{Re}}

\newcommand{\veps}{\varepsilon}

\newcommand{\bfH}{{\bf H}}

\newcommand{\bbN}{\mathbb N}

\newcommand{\bbR}{\mathbb R}

\newcommand{\calC}{\mathcal C}
\newcommand{\calD}{\mathcal D}

\newcommand{\calF}{\mathcal F}

\newcommand{\calJ}{\mathcal J}
\newcommand{\calK}{\mathcal K}
\newcommand{\calL}{\mathcal L}

\newcommand{\calR}{\mathcal R}
\newcommand{\calS}{\mathcal S}

\definecolor{deepgreen}{cmyk}{1,0,1,0.5}

\newcommand{\C}{\mathbb{C}}

\newcommand{\N}{\mathbb{N}}
\newcommand{\R}{\mathbb{R}}

\makeatletter

\newcommand{\Rmnum}[1]{\expandafter\@slowromancap\romannumeral #1@}
\makeatother

\newcommand{\ha}{\widehat}
\renewcommand{\bar}{\underline}

\newcommand{\Del}[1]{}

\numberwithin{equation}{section}

\newtheorem{theorem}{Theorem}[section]
\newtheorem{corollary}[theorem]{Corollary}
\newtheorem{lemma}[theorem]{Lemma}
\newtheorem{proposition}[theorem]{Proposition}
\newtheorem{remark}[theorem]{Remark}

\renewcommand\Re{\mathrm{Re}\,}

\newcommand{\frakR}{\mathfrak{R}}
\newcommand{\xb}{\overline{x}}

\newcommand{\trho}{\widetilde{\rho}}

\newcommand{\xt}{\tilde{x}}
\newcommand{\tcalL}{\tilde{\calL}}
\newcommand{\tcalJ}{\tilde{\calJ}}

\begin{document}

\title{On stability of blow up solutions for the critical co-rotational Wave Maps problem}
\author{Joachim Krieger, Shuang Miao}


\thanks{Support of the Swiss National Fund is gratefully acknowledged.}
\begin{abstract}
We show that the finite time blow up solutions for the co-rotational Wave Maps problem constructed in \cite{KST2,GaoKr} are stable under suitably small perturbations within the co-rotational class, provided the scaling parameter $\lambda(t) = t^{-1-\nu}$ is sufficiently close to $t^{-1}$, i. e. the constant $\nu$ is sufficiently small and positive. The method of proof is inspired by \cite{CondBlow,BuKr}, but takes advantage of geometric structures of the Wave Maps problem already used in \cite{RaRod,BeKrTa} to simplify the analysis. In particular, we heavily exploit that the resonance at zero satisfies a natural first order differential equation. 
\end{abstract}

\maketitle
\tableofcontents

\section{Introduction}

Let $M\hookrightarrow \R^k$ be a Riemannian submanifold of Euclidean space, and $U: \R^{n+1}\longrightarrow M$ a map. Here $\R^{n+1}$ denotes Minkowski space equipped with the standard metric $m$ with signature $-1,1,\ldots, 1$. Then $U$ is called a Wave Map, provided it is critical with respect to the formal action functional 
\[
\mathcal{L}(U): = \int_{\R^{n+1}}\partial_{\alpha}U\cdot\partial^{\alpha}U\,dxdt,
\]
where $U$ is expressed in terms of the ambient coordinates of $\R^k$ and the raising of indices is in accordance with the Minkowski metric, i. e. 
\[
\partial^{\alpha} = m^{\alpha\beta}\partial_{\beta}.
\]
The Wave Maps problem can then be cast in the form 
\begin{equation}\label{eq:WMgeneral}
\Box U^i + C^i_{jk}(U)\partial_{\alpha}U^j\partial^{\alpha}U^k = 0
\end{equation}
for suitable coefficient functions $C^i_{jk}(U)$. 
Throughout the Einstein summation convention is in place, with summation taking place over repeated indices, and Greek letters denoting indices $0,1,\ldots, n$, while Roman indices refer to spatial indices $1,2,\ldots, n$ only. 
The Wave Maps problem has attracted particular interest in the {\it{critical}} dimension $n = 2$, where recent progress has revealed a rather satisfying picture in terms of regularity versus singularity formation of solutions. Specifically, the {\it{Threshold Theorem}} established in full generality in \cite{StTat} and in more specific cases (but with more detailed description of the solutions) in \cite{Tao1,Tao2,Tao3,Tao4,Tao5,Tao6,Tao7,KS0}, implies that provided $M$ does not admit non-trivial finite energy harmonic maps originating on $\R^2$, singularities cannot form. In particular, Wave Maps $U: \R^{2+1}\longrightarrow {\mathbf{H}}^2$ are globally regular. These results were obtained after obtaining a rather precise theory for small-energy wave maps from $\bbR^{2+1}$ to $\bfH^{2}$ as well as general targets, see \cite{Krieger04,Tao2,Tataru1,Tataru3,Tataru2}, which in turn were preceded by more specialized results in a symmetry reduced setting by Christodoulou, Tahvildar-Zadeh \cite{ChShadi}, and
Struwe \cite{Struwe2}. An important precursor of the small energy theory by Tao and Tataru was the work by Klainerman-Machedon \cite{KlMa1,KlMa2}, introducing the framework of $X^{s,b}$-spaces in the context of wave equations with null-structures. 
\\

By contrast, $M = S^2$, the simplest target admitting non-trivial finite energy harmonic maps originating on $\R^2$, does admit solutions resulting in finite time singularities. More precisely, numerical evidence in \cite{BizWM,IsenLieb} strongly suggested singularity development for equivariant wave maps of co-rotation index one from $\bbR^{2+1}$ to $S^{2}$. Struwe \cite{Struwe1} showed that restricting to co-rotational Wave Maps, such a  singularity can only occur in an \emph{energy concentration scenario}, and more precisely he showed that this has to happen by the bubbling off of at least one harmonic map along a suitable time sequence. Later, this type of singular solution with exactly one bubble was constructed in \cite{KST2,RodSter,RaRod}. Similar `bubbling off' singularities can be also constructed for solutions to the focusing energy-critical nonlinear wave equation, Yang-Mills equation, and the critical Schr\"odinger map equation (see \cite{KST1,KS1,KST,MeRaRo1,MeRaRo2,Perelman}). We also note that the result of Struwe was rendered more precise by Cote \cite{Cote}, detailing a decomposition of singular co-rotational Wave Maps into a number of bubbles plus an error term. 
\\

The numerical result in \cite{BizWM} suggests that if the blowup rate is close to the self-similar rate, then the blowup solution (with one bubble) is stable. In the present work, we give a rigorous proof of this in case of the solutions built in \cite{KST2, GaoKr}.\\
Due to energy conservation of Wave Maps 
\[
\sum_{\alpha = 0}^n\int_{\R^2}|\partial_{\alpha}U|^2\,dx = \text{const},
\]
singularities are necessarily of an energy concentration type. Specifically, the solutions in the above mentioned references are all of the form 
\[
U(t, x) = Q(\lambda(t)x) + \veps(t, x),
\]
where $Q$ is a suitable harmonic map $Q: \R^2\longrightarrow S^2$, and $\lim_{t\rightarrow T}\lambda(t) = +\infty$, with $T$ the finite time blow up. The radiation term $\veps(t, x)$ on the other hand gets evacuated outside of the backward light cone centered at the singularity, in the sense that 
\[
\lim_{t\rightarrow T}\sum_{\alpha = 0}^2\int_{|x|\leq T - t}|\partial_{\alpha}\veps|^2\,dx = 0.
\]
The constructions in \cite{KST2,RodSter,RaRod} all rely on suitable symmetry reductions to reduce the equations to a system of wave equations which no longer involve the complicated derivative structure as in \eqref{eq:WMgeneral}. Precisely, both \cite{KST2, RaRod} consider the so-called {\it{co-rotational symmetry reduction}} of Wave Maps $U: \R^{2+1}\longrightarrow S^2$, where the Wave Map $U$ is stipulated to be of the specific form 
\[
U(t, x) = (u(t, r),\omega),\,r = |x|,
\]
where one uses standard polar coordinates on the sphere $S^2$ and $(r, \omega)$ refer to standard spherical coordinates on $\R^2$. One then infers the equation 
\begin{equation}\label{eq:WMcorot}
-u_{tt} + u_{rr} + \frac1r u_r = \frac{\sin(2u)}{2r^2}
\end{equation}
Observe that energy conservation for this problem is expressed by 
\[
\int_{0}^\infty \left[\frac12(u_r^2 + u_t^2)+ \frac{\sin^2u}{2r^2}\right]\,r\,dr = \text{const}.
\]

This problem admits the {\it{static finite energy solution}} 
\[
Q(r) = 2\arctan r,
\]
and the blow ups in \cite{KST2,RaRod} are constructed by suitable perturbations of this solution, and are of the form 
\[
u(t, r) = Q(\lambda(t)r) + \veps(t, r),\,\lim_{t\rightarrow 0}\lambda(t) = +\infty.
\]

An important difference between the examples constructed in \cite{RaRod} versus those in \cite{KST2} is that the former were shown to be {\it{stable under suitable (co-rotational) perturbations}}, while the solutions in \cite{KST2} came without any stability assertion. Also, while the solutions in \cite{RaRod} have $C^\infty$-smooth data, those in \cite{KST2} are of lesser regularity $H^{1+\nu-}$, $\nu>\frac12$, and also display quite different concentration dynamics (i. e. the functions $\lambda(t)$), namely 
\[
\lambda(t) = t^{-1-\nu},\,\nu>\frac12
\]
in \cite{KST2} versus 
\[
\lambda(t)\sim t^{-1}e^{c\sqrt{|\log t|}}
\]
in \cite{RaRod}. 
\\
The construction of \cite{KST2} was extended to the full polynomial range $\nu>0$ in \cite{GaoKr}, where solutions for \eqref{eq:WMcorot} of the form 
\[
u(t, r) = Q(\lambda(t)r) + \veps(t, r),\,\lambda(t) = t^{-1-\nu},\,0<\nu<\frac12
\]
were constructed. The issue of their {\it{stability}}, however, remained open, just as for the solutions constructed in \cite{KST2}. 
\\

In this paper, we address the issue of stability, and prove that the solutions constructed in \cite{KST2,GaoKr} are stable under suitable (co-rotational) perturbations, provided $\nu>0$ is small enough. This result is a direct analogue of the corresponding one established in \cite{BuKr} for the energy critical focussing nonlinear wave equation 
\[
\Box u = -u^5
\]
on $\R^{3+1}$. 
\\
In order to formulate the main theorem, we associate with a function $\veps(r)$, defined for $r\geq 0$, the map from $\R^2\rightarrow \C\simeq\R^2$ defined by  
\[
(r, \theta)\longrightarrow \veps(r)\cdot e^{i\theta}
\]
Then denote by 
\[
\big\|\veps\big\|_{H^{l}_{\R^2}}
\]
the $H^{l}$-norm of the above map from $\R^2$ to $\R^2$.

\begin{theorem}\label{thm:Main} Let $\nu>0$ sufficiently small. Then there exists a singular solution $u_{\nu}$ for \eqref{eq:WMcorot} on $[t_0, 0)\times \R^2$, $t_0 = t_0(\nu)>0$ small enough,  of the form 
\[
u_{\nu}(t, r) = Q(\lambda(t)r) + \veps_{\nu}(t, r),\,\lambda(t) = t^{-1-\nu}, 
\]
constructed as in \cite{KST2,GaoKr}, with the following property: There is $\delta_0>0$ small enough, such that for any data perturbation $(\veps_0, \veps_1)\in H_{\R^2}^4\times H_{\R^2}^3$ with 
\[
\big\|(\veps_0, \veps_1)\big\|_{H_{\R^2}^4\times H_{\R^2}^3}<\delta_0, 
\]
the data $(u_{\nu}(t_0, \cdot)+\veps_0, \partial_t u_{\nu}(t_0, \cdot) + \veps_1)$ at time $t_0$ lead to a solution for \eqref{eq:WMcorot} on $[t_0, 0)\times \R^2$ of the form 
\begin{equation}\label{eq:ansatz}
u(t, r) = Q(\lambda(t)r) + \veps(t, r)
\end{equation}
with $(\veps(t, \cdot), \veps_t(t, \cdot)\in H^{1+\nu-}_{\R^2}\times H^{\nu-}_{\R^2}$ for all $t\in [t_0, 0)$ and with 
\[
\lim_{t\rightarrow 0}\int_0^{t}\left(\veps_t^2 + \veps_r^2 + \frac{\sin^2\veps}{2r^2}\right)\,rdr = 0
\]
Thus the perturbed solutions display the same dynamics, and blow up at the same point in space time. Their regularity at any time $t\in [t_0,0)$ is of class $H^{1+\nu-}$. 
\end{theorem}
\begin{remark} We note that the conditions on the data $\veps_{0,1}$ imply in particular that they vanish at the origin $r = 0$. 

\end{remark}

\section{Outline of paper}

The preceding theorem will be obtained by implementing a suitable iterative scheme to solve the perturbative problem, which will be formulated in terms of two quantities $\calD\veps$ and $c(\tau)$, representing the `non-resonant' as well as `resonant' parts of $\veps$ and which can be re-assembled to produce $\veps$ via the formula \eqref{eqn:epsdecomp}. The next section will be devoted to deduce the structure of the equation for $\calD\veps$, given in \eqref{eq Dveps temp 2}, as well as the ODE driving $c(\tau)$, given in \eqref{ctau ODE}. 
\\
Section 4 then details the distorted Fourier basis which we use to describe the function $\calD\veps$, largely based on \cite{KST2,BeKrTa}. There we also derive the all-important transference identity given in Proposition~\ref{prop:Kstructure}. At the end of that section, we also introduce the norm $\|\cdot\|_{S_0}$, which will be the key ingredient to control the iterates. 
In section 5, the translation of the problem to the Fourier side is implemented, carefully transforming the left hand side into a manageable transport operator, see \eqref{eq on Fourier side final}. This is largely analogous to the formulation used in \cite{KS1}. 
The following sections are then devoted to controlling the source terms appearing on the right of \eqref{eq on Fourier side final}, which will be possible once the correct fine structure of the Fourier transform $\xb(\tau, \xi)$ of $\calD\veps$ has been identified, see \eqref{eq:admissiblexb}. The necessity of such a decomposition comes from the limited smoothness of the $u_{\nu}$ getting perturbed, as manifested by a `shock' these solutions experience across the light cone. This latter feature is also responsible for the somewhat puzzling fact that the perturbed solutions blow up in the same space-time location: the perturbations are too regular to displace the shock to a shifted light cone. Thus the shock across the light cone appears to impart a certain rigidity to these solutions. However, as far as the main thrust of this paper is concerned, this is only a technical aspect, specific to the solutions being perturbed, while the method in and of itself is presumably of much wider applicability. 
\\
In the sections leading from there up to the critical section 11, bounds for iterates are deduced, sometimes requiring {\it{re-iteration}} to force a smallness gain, see e. g. Lemma~\ref{lem:ctocdelicate}. However, this involves at most a small number of re-iterations. 
\\
There is, however, one type of source term, specifically the {\it{non-local linear operators}} given by $2\frac{\lambda'}{\lambda}\mathcal{K}_0\calD_{\tau}\xb$, see Proposition~\ref{prop:nonlocallinear1}, where no smallness is obtained, even after two- or three-fold re-iteration. This is an issue which arose in analogous fashion in \cite{KST1,DoMHKS,CondBlow}, and we deal with this by a similar method via {\it{manifold re-iteration}}, in section 11.

\section{Separation of the dynamics into resonant and non-resonant parts}

Following \cite{KST2}, the equation governing the evolution of $\veps(t, r)$ in the ansatz \eqref{eq:ansatz} is given by 
\begin{equation}\label{eq:epsequation1}\begin{split}
&\left(-\partial_t^2 + \partial_r^2 + \frac1r\partial_r\right)\veps - \frac{\cos(2Q(\lambda(t)r)}{r^2}\veps = N(\veps),\\
&N(\veps) = \frac{\cos(2u^\nu) - \cos(2Q(\lambda(t)r)}{r^2}\veps + \frac{\sin(2u^\nu)}{2r^2}(\cos(2\veps) - 1) + \frac{\cos(2u^\nu)}{2r^2}(\sin(2\veps) - 2\veps)
\end{split}\end{equation}
Introducing the new coordinates
\[
\tau = \int_t^\infty \lambda(s)\,ds,\,R = \lambda(t)r, 
\]
and using the identities
\begin{align*}
 \partial_{t}\veps= -\lambda\left(\partial_{\tau}\veps+\frac{\lambda'(\tau)}{\lambda}R\partial_{R}\veps\right),\quad \partial_{r}\veps=\lambda\partial_{R}\veps,
\end{align*}
and
\begin{align*}
 \partial^{2}_{t}\veps=\lambda'(\tau)\lambda\left(\partial_{\tau}\veps+\frac{\lambda'(\tau)}{\lambda}R\partial_{R}\veps\right)+\lambda^{2}\left(\partial_{\tau}+\frac{\lambda'(\tau)}{\lambda}R\partial_{R}\right)^{2}\veps,\quad \partial_{r}^{2}\veps=\lambda^{2}\partial_{R}^{2}\veps,
\end{align*}
we infer the following equation in terms of the new coordinate system: 
\begin{align}\label{eq veps tau R}
\begin{split}
 -&\left(\left(\partial_{\tau}+\frac{\lambda'(\tau)}{\lambda(\tau)}R\partial_{R}\right)^{2}+\frac{\lambda'(\tau)}{\lambda(\tau)}\left(\partial_{\tau}+\frac{\lambda'(\tau)}{\lambda(\tau)}R\partial_{R}\right)\right)\veps\\
 &+\left(\partial_{R}^{2}+\frac{1}{R}\partial_{R}-\frac{\cos(2Q(R))}{R^{2}}\right)\veps=\lambda^{-2}N(\veps),\\
 \Rightarrow\quad &-\left(\left(\partial_{\tau}+\frac{\lambda'(\tau)}{\lambda(\tau)}R\partial_{R}\right)^{2}+\frac{\lambda'(\tau)}{\lambda(\tau)}\left(\partial_{\tau}+\frac{\lambda'(\tau)}{\lambda(\tau)}R\partial_{R}\right)\right)\veps-\calL\veps=\lambda^{-2}N(\veps)
 \end{split}
\end{align}
where we use the notation 
\begin{align*}
\calL\veps &= \left(-\partial_{R}^{2}-\frac{1}{R}\partial_{R}+\frac{\cos(2Q(R))}{R^{2}}\right)\veps\\
& =  \left(-\partial_{R}^{2}-\frac{1}{R}\partial_{R} + \frac{1}{R^2}\frac{1-6R^2 + R^4}{(1+R^2)^2}\right)\veps
\end{align*}
Following \cite{RaRod,BeKrTa} we now introduce a key first order operator
\begin{align*}
 \calD:=\partial_{R}+\frac{1}{R}\frac{R^{2}-1}{R^{2}+1}=\partial_{R}+\frac{1}{R}\left(1-\frac{2}{R^{2}+1}\right)=\partial_{R}+\frac{1}{R}-\frac{2}{R(R^{2}+1)}.
\end{align*}
as well as its dual operator 
 \begin{align*}
 -\calD^{*}=\partial_{R}+\frac{1}{R}\left(1+\frac{1-R^{2}}{1+R^{2}}\right).
 \end{align*}
Then the following remarkable identity is checked by direct computation: 
\begin{proposition}(\cite{RaRod,BeKrTa}) The operator $\mathcal{L}$ can be factorized
\[
\mathcal{L} = \calD^{*} \calD
\]
Furthermore, the operator $\mathcal{D}$ annihilates the resonance: 
\[
\mathcal{D}\phi_0(R) = 0,\quad {\phi_{0}(R):=\frac{R}{1+R^{2}}.}
\]
\end{proposition}

The idea now is to pass from the equation for $\veps$ to one for $\mathcal{D}\veps$. This will replace the operator $\mathcal{L}$ by the operator $\tilde{\mathcal{L}} = \mathcal{D}\mathcal{D}^*$, which no longer has a resonance at zero, and correspondingly has a much more regular spectral measure associated to it. Then introducing the right inverse of $\mathcal{D}$ given by
\[
\phi(g) = \phi_0(R)\int_0^R(\phi_0(s))^{-1}g(s)\,ds,
\]
we can set 
\begin{equation}\label{eqn:epsdecomp}
\veps = \phi(\mathcal{D}\veps) + c(\tau)\phi_0(R), 
\end{equation}
and we reduce to controlling the evolution of the pair of functions $(\mathcal{D}\veps, c(\tau))$. In the following, we deduce the system of equations describing the evolution of these functions. 

\subsection{The equation for $\mathcal{D}\veps$.}

Here we apply $\mathcal{D}$ to the equation and encounter a number of commutators. We have 
\begin{align}\label{commutator D RdR}
\begin{split}
[\calD,R\partial_{R}]=&[\partial_{R},R\partial_{R}]+\left[\frac{1}{R},R\partial_{R}\right]-\left[\frac{2}{R(R^{2}+1)},R\partial_{R}\right]\\
=&\partial_{R}+\frac{1}{R}-\frac{6R^{2}+2}{R(R^{2}+1)^{2}}=\calD-\frac{4R}{(R^{2}+1)^{2}}.
\end{split}
\end{align}
and so 
\begin{align*}
\left[\calD,\partial_{\tau}+\frac{\lambda'}{\lambda}R\partial_{R}\right]=\frac{\lambda'}{\lambda}[\calD,R\partial_{R}]=\frac{\lambda'}{\lambda}\calD-\frac{4R\lambda'}{(R^{2}+1)^{2}\lambda},
\end{align*}
and further 
\\
\begin{align*}
\left[\calD,\left(\partial_{\tau}+\frac{\lambda'}{\lambda}R\partial_{R}\right)^{2}\right]\veps
&=\left[\calD,\partial_{\tau}+\frac{\lambda'}{\lambda}R\partial_{R}\right]\left(\partial_{\tau}+\frac{\lambda'}{\lambda}R\partial_{R}\right)\veps+\left(\partial_{\tau}+\frac{\lambda'}{\lambda}R\partial_{R}\right)\left[\calD,\partial_{\tau}+\frac{\lambda'}{\lambda}R\partial_{R}\right]\veps\\
=&\frac{\lambda'}{\lambda}\left(\calD-\frac{4R}{(R^{2}+1)^{2}}\right)\left(\partial_{\tau}+\frac{\lambda'}{\lambda}R\partial_{R}\right)\veps+\left(\partial_{\tau}+\frac{\lambda'}{\lambda}R\partial_{R}\right)\left(\frac{\lambda'}{\lambda}\left(\calD-\frac{4R}{(R^{2}+1)^{2}}\right)\right)\veps.
\end{align*}
Applying $\mathcal{D}$ to \eqref{eq veps tau R}, we then obtain 
\begin{align}\label{eq Dveps temp 1}
\begin{split}
&-\left(\left(\partial_{\tau}+\frac{\lambda'(\tau)}{\lambda(\tau)}R\partial_{R}\right)^{2}+\frac{\lambda'(\tau)}{\lambda(\tau)}\left(\partial_{\tau}+\frac{\lambda'(\tau)}{\lambda(\tau)}R\partial_{R}\right)\right)\calD\veps-\tcalL\calD\veps\\
=&\lambda^{-2}\calD\left(N(\veps)\right)+\frac{\lambda'}{\lambda}\left(\calD-\frac{4R}{(R^{2}+1)^{2}}\right)\left(\partial_{\tau}+\frac{\lambda'}{\lambda}R\partial_{R}\right)\veps\\
&+\left(\partial_{\tau}+\frac{\lambda'}{\lambda}R\partial_{R}\right)\left(\frac{\lambda'}{\lambda}\left(\calD-\frac{4R}{(R^{2}+1)^{2}}\right)\right)\veps\\
&+\left(\frac{\lambda'}{\lambda}\right)^{2}\left(\calD-\frac{4R}{(1+R^{2})^{2}}\right)\veps\\
=:&\lambda^{-2}\calD\left(N(\veps)\right)+I+II+III.
\end{split}
\end{align}
Commuting $\calD$ and $\partial_{\tau}+\frac{\lambda'}{\lambda}R\partial_{R}$ again, we obtain:

\begin{align*}
I=&\left(\frac{\lambda'}{\lambda}\right)^{2}\left(\calD-\frac{4R}{(R^{2}+1)^{2}}\right)\veps+\frac{\lambda'}{\lambda}\left(\partial_{\tau}+\frac{\lambda'}{\lambda}R\partial_{R}\right)\calD\veps-\frac{\lambda'}{\lambda}\frac{4R}{(R^{2}+1)^{2}}\left(\partial_{\tau}+\frac{\lambda'}{\lambda}R\partial_{R}\right)\veps,\\
II=&\left(\frac{\lambda'}{\lambda}\right)'\left(\calD-\frac{4R}{(R^{2}+1)^{2}}\right)\veps+\frac{\lambda'}{\lambda}\left(\partial_{\tau}+\frac{\lambda'}{\lambda}R\partial_{R}\right)\calD\veps-\frac{\lambda'}{\lambda}\left(\partial_{\tau}+\frac{\lambda'}{\lambda}R\partial_{R}\right)\left(\frac{4R}{(R^{2}+1)^{2}}\veps\right)
\end{align*}
Finally, we can reformulate equation \eqref{eq Dveps temp 1} as follows:
\begin{align}\label{eq Dveps temp 2}
\begin{split}
&-\left(\left(\partial_{\tau}+\frac{\lambda'}{\lambda}R\partial_{R}\right)^{2}+3\frac{\lambda'}{\lambda}\left(\partial_{\tau}+\frac{\lambda'}{\lambda}R\partial_{R}\right)\right)\calD\veps-\tilde{\calL}\calD\veps\\
=&\lambda^{-2}\calD\left(N(\veps)\right)-\frac{4R}{(R^{2}+1)^{2}}\left(2\left(\frac{\lambda'}{\lambda}\right)^{2}+\left(\frac{\lambda'}{\lambda}\right)'\right)\veps\\
&-\frac{\lambda'}{\lambda}\frac{4R}{(R^{2}+1)^{2}}\left(\partial_{\tau}+\frac{\lambda'}{\lambda}R\partial_{R}\right)\veps-\frac{\lambda'}{\lambda}\left(\partial_{\tau}+\frac{\lambda'}{\lambda}R\partial_{R}\right)\left(\frac{4R}{(R^{2}+1)^{2}}\veps\right)\\
&+\left(2\left(\frac{\lambda'}{\lambda}\right)^{2}+\left(\frac{\lambda'}{\lambda}\right)'\right)\calD\veps\\
=:&\lambda^{-2}\calD\left(N(\veps)\right)+\calR(\veps,\calD\veps)+\left(2\left(\frac{\lambda'}{\lambda}\right)^{2}+\left(\frac{\lambda'}{\lambda}\right)'\right)\calD\veps.
\end{split}
\end{align}
Here recall that $\tilde{\mathcal{L}} = \mathcal{D}\mathcal{D}^*$. For the most part, we shall be working with the preceding equation to derive bounds on $\calD\veps$. 

\subsection{The equation for $c(\tau)$.} Here we deduce a second order ODE for $c(\tau)$ (as in \eqref{eqn:epsdecomp}). For this, assume that we have at least the order of vanishing 
\[
\left(\partial_{\tau} + \frac{\lambda'}{\lambda}R\partial_R\right)^{\kappa}\mathcal{D}\veps(\tau, R) = O(R^2),\,\kappa = 0,1,2.
\]
as $R\rightarrow 0$. Further, assume that 
\[
\mathcal{D}\veps(\tau, R) = \alpha(\tau)R^2 + O(R^4).
\]
This will be justified later on once we introduce the space to control $\mathcal{D}\veps(\tau, R)$ and introduce suitable regularizations. It is then easily verified that all terms of the form 
\[
\left(\partial_{\tau} + \frac{\lambda'}{\lambda}R\partial_R\right)^{\kappa}\phi(\mathcal{D}\veps(\tau, R))
\]
are of size $O(R^2)$ as $R\rightarrow 0$. Furthermore, we infer (as $R\rightarrow 0$)
\[
\mathcal{L}\big(\phi(\mathcal{D}\veps(\tau, R))\big) = \mathcal{D}^*\mathcal{D}\veps(\tau, R)) = h(\tau)\phi_0(R) + O(R^3)
\]
We further have the following relations 
\begin{align*}
\left(\partial_{\tau}+\frac{\lambda'}{\lambda}R\partial_{R}\right)\left(c(\tau)\phi_{0}(R)\right)=&\phi_{0}(R)\left(\partial_{\tau}+\frac{\lambda'}{\lambda}\right)c(\tau)+O(R^{2}),\\
\left(\partial_{\tau}+\frac{\lambda'}{\lambda}R\partial_{R}\right)^{2}\left(c(\tau)\phi_{0}(R)\right)=&\phi_{0}(R)\left(\partial_{\tau}+\frac{\lambda'}{\lambda}\right)^{2}c(\tau)+O(R^{2}).
\end{align*}
Recalling \eqref{eqn:epsdecomp}, we see that the left hand side of \eqref{eq veps tau R} can be written in the regime of small $R$ in the form 
\begin{align}\label{ctau ODE inves 1}
 \begin{split}
  -\phi_{0}(R)\left(\partial_{\tau}+\frac{\lambda'}{\lambda}\right)^{2}c(\tau)-\phi_{0}(R)\frac{\lambda'}{\lambda}\left(\partial_{\tau}+\frac{\lambda'}{\lambda}\right)c(\tau)
  -h(\tau)\phi_{0}(R)+O(R^{2}).
 \end{split}
\end{align}
On the other hand, the right hand side of \eqref{eq veps tau R} is 

\begin{align}\label{ctau ODE inves 2}
 \lambda^{-2}N(\veps)=\lambda^{-2}\phi(\calD(N(\veps)))+\lambda^{-2}n(\tau)\phi_{0}(R).
\end{align}
Dividing by $R$ and letting $R\rightarrow 0$, we obtain the following ODE for the parameter $c(\tau)$: 
\begin{align}\label{ctau ODE}
 \left(\partial_{\tau}+\frac{\lambda'}{\lambda}\right)^{2}c(\tau)+\frac{\lambda'}{\lambda}\left(\partial_{\tau}+\frac{\lambda'}{\lambda}\right)c(\tau)
  +h(\tau)+\lambda^{-2}n(\tau)=0.
\end{align}
where the functions $h(\tau), n(\tau)$ are determined by the relations 
\[
h(\tau) = \lim_{R\rightarrow 0}R^{-1}\mathcal{D}^*\mathcal{D}\veps(\tau, R),\,n(\tau) = \lim_{R\rightarrow 0}R^{-1} \lambda^{-2}N(\veps)(\tau, R).
\]
\subsection{Summary}

The goal now becomes to solve the coupled system consisting of \eqref{eq Dveps temp 2}, \eqref{ctau ODE}, keeping in mind that $\veps(\tau, R)$ is then given by \eqref{eqn:epsdecomp}. 

\section{Fourier representation of the derivative $\mathcal{D}\veps$.}

\subsection{Distorted Fourier basis at the level of the derivative}

Here, in analogy to \cite{BeKrTa}, we derive the Fourier representation with a much less singular measure compared to the one in \cite{KST2}, upon passing to the derivative of the perturbation. 
Recall from \cite{KST2} that any $L^2_{dR}$ function $f(R)$ on $(0,\infty)$ admits a representation 
\[
f(R) = \int_0^\infty x(\xi)\phi_{KST}(R, \xi)\rho(\xi)\,d\xi,\,\quad x(\xi) = \int_0^\infty f(R)\phi_{KST}(R, \xi)\,dR,\,\quad\mathcal{L}\big(R^{-\frac12}\phi_{KST}(R, \xi)\big) = \xi R^{-\frac12}\phi_{KST}(R, \xi).
\]
In fact, the Fourier basis $\phi(r, z)$ from Section 4 in \cite{KST2} is renamed $\phi_{KST}(R, \xi)$ here. We shall consistently rely on the asymptotics derived in loc. cit. 
In particular, the spectral measure $\rho(\xi)$ has the asdymptotics $\rho(\xi)\sim \xi$ as $\xi\rightarrow\infty$, and $\rho(\xi)\sim \frac{1}{\xi\log^2\xi}$ as $\xi\rightarrow 0$. As we shall be working with radial functions on $\R^2$ and hence replacing $dR$ by $RdR$, we replace the Fourier basis $\phi_{KST}$ by $R^{-\frac12}\phi_{KST}$. Also, we identify radial $g(x)$ with $g(x) = f(R)$, $|x| = R$, and then write 
\[
f(R) = \int_0^\infty x(\xi)R^{-\frac12}\phi_{KST}(R, \xi)\rho(\xi)\,d\xi,\,\quad x(\xi) = \langle R^{-\frac12}\phi_{KST}(R, \xi), f(R)\rangle_{L^2_{RdR}}
\]

Assuming $f$ to be smooth and compactly supported away zero, say, we can then differentiate the preceding relation, obtaining 
\begin{align*}
\mathcal{D}f(R) &= \int_0^\infty x(\xi)\mathcal{D}\big(R^{-\frac12}\phi_{KST}(R, \xi)\big)\rho(\xi)\,d\xi\\
& =  \int_0^\infty x(\xi)\xi^{-1}\mathcal{D}\big(R^{-\frac12}\phi_{KST}(R, \xi)\big)\tilde{\rho}(\xi)\,d\xi
\end{align*}
where $\tilde{\rho}(\xi) = \xi\rho(\xi)$. Following \cite{BeKrTa}, let us set henceforth
\begin{equation}\label{eq:phidef}
\phi(R, \xi): = \xi^{-1}\mathcal{D}\big(R^{-\frac12}\phi_{KST}(R, \xi)\big).
\end{equation}
Observe that these are now generalized eigenfunctions associated to the operator $\tilde{\mathcal{L}}$: 
\[
\tilde{\calL}\phi(R, \xi) = \xi\phi(R, \xi).
\]

Then for $f$ as before we obtain the representation formula
\[
\mathcal{D}f(R) = \int_0^\infty x(\xi)\phi(R, \xi)\tilde{\rho}(\xi)\,d\xi, 
\]
where we have the relation 
\[
x(\xi) = \langle \mathcal{D}f, \phi(R, \xi)\rangle_{L^2_{RdR}}.
\]
Indeed, we have 
\begin{align*}
\langle \mathcal{D}f, \phi(R, \xi)\rangle_{L^2_{RdR}} &= \xi^{-1} \langle \mathcal{D}f, \mathcal{D}\big(R^{-\frac12}\phi_{KST}(R, \xi)\big)\rangle_{L^2_{RdR}} =  \xi^{-1} \langle f, \mathcal{L}\big(R^{-\frac12}\phi_{KST}(R, \xi)\big)\rangle_{L^2_{RdR}}\\& =  \langle f, \big(R^{-\frac12}\phi_{KST}(R, \xi)\big)\rangle_{L^2_{RdR}} = x(\xi).
 \end{align*}
 Moreover, the map $\mathcal{D}f\longrightarrow x(\xi)$ is an isometry from $L^2_{RdR}$ to $L^2_{\tilde{\rho}}$: 
 \begin{align*}
 \big\|\mathcal{D}f\big\|_{L^2_{RdR}}^2 = \langle \mathcal{L}f, f\rangle_{L^2_{RdR}} &= \int_0^\infty \langle R^{\frac12} \mathcal{L}f, \phi_{KST}\rangle_{L^2_{dR}} \overline{\langle R^{\frac12}f, \phi_{KST}\rangle}_{L^2_{dR}}\rho(\xi)\,d\xi\\
 & = \int_0^\infty \langle  \mathcal{L}f, R^{-\frac12}\phi_{KST}\rangle_{L^2_{RdR}} \overline{\langle f, R^{-\frac12}\phi_{KST}\rangle}_{L^2_{RdR}}\rho(\xi)\,d\xi\\
 & =  \int_0^\infty \langle f, R^{-\frac12}\phi_{KST}\rangle_{L^2_{RdR}} \overline{\langle f, R^{-\frac12}\phi_{KST}\rangle}_{L^2_{RdR}}\xi\rho(\xi)\,d\xi\\
 & = \big\|\langle \mathcal{D}f, \phi(R, \xi)\rangle_{L^2_{RdR}}\big\|_{L^2_{\tilde{\rho}}}^2
 \end{align*}
 We shall from now on work with the Fourier representation
 \begin{equation}\label{eq:vepsrep}
 \mathcal{D}\veps(\tau, R) = \int_0^\infty \phi(R, \xi)\xb(\tau, \xi)\tilde{\rho}(\xi)\,d\xi
 \end{equation}
 We shall also use the notation
 \[
 \mathcal{F}\big( \mathcal{D}\veps(\tau, \cdot)\big)(\xi): = \xb(\tau, \xi).
 \]
 For $\calD\veps$, we have the following $L^{\infty}_{dR}$-estimate:
 \begin{lemma}\label{lem: Dveps Linfty}
 Let $\calD\veps$ be given by \eqref{eq:vepsrep}, and define $\|\cdot\|_{S_0}$ as in Proposition~\ref{prop:K_0Sbound} below. We have the following $L^{\infty}_{dR}$-estimate:
 \begin{align}\label{Dveps Linfty}
 \left\|\frac{\calD\veps(\tau,R)}{\langle\log R\rangle}\right\|_{L^{\infty}_{dR}}\lesssim\|\langle\xi\rangle^{-\frac{1}{2}}\xb(\tau,\xi)\|_{S_{0}}.
 \end{align}
 \end{lemma}
 
 \begin{proof}
 We write
\begin{align*}
\calD\veps(R)=\int_{0}^{\frac{1}{2}}\xb(\tau,\xi)\phi(R,\xi)\trho(\xi)d\xi+\int_{\frac{1}{2}}^{\infty}\xb(\tau,\xi)\phi(R,\xi)\trho(\xi)d\xi:=A+B.
\end{align*}
For $A$, we note that if $R\ll1$ and $R^{2}\xi\lesssim 1$, we have $\phi(R,\xi)\sim R^{2}\lesssim 1$. If $R\gtrsim1$ and $R^{2}\xi\lesssim 1$, we have $\phi(R,\xi)\sim\log R\lesssim |\log\xi|$. If $R^{2}\xi\gtrsim1$, we have $\phi(R,\xi)\sim|\log\xi|$, $\xi\ll 1$. In any case,

\begin{align*}
\frac{|A|}{\langle\log R\rangle}\lesssim\int_{0}^{\frac{1}{2}}\xb(\tau,\xi)|\log\xi|^{-2}d\xi\lesssim\|\xi^{\frac{1}{2}}\langle\log\xi\rangle^{-1-\kappa}\xb(\tau,\xi)\|_{L^{2}(\xi\leq\frac{1}{2})},\,0<\kappa<\frac12.
\end{align*}
For $B$, if $R^{2}\xi\ll 1$, $\phi(R,\xi)\sim R^{2}$. If $R^{2}\xi\gtrsim 1$, $\phi(R,\xi)\sim\xi^{-1}$. In any case, 

\begin{align*}
|B|\lesssim\int_{\frac{1}{2}}^{\infty}\xb(\tau,\xi)\xi^{-1}\trho(\xi)d\xi\lesssim\int_{\frac{1}{2}}^{\infty}\xb(\tau,\xi)\xi d\xi\lesssim\|\langle\xi\rangle^{2+\kappa}\xb(\tau,\xi)\|_{L^{2}(\xi\geq\frac{1}{2})}.
\end{align*}
 \end{proof}
 \subsection{The transference identity}
 Following the method in \cite{KST2}, a key issue arising when translating the equation \eqref{eq Dveps temp 2} to the Fourier side is the fact that the operator $R\partial_R$ does not translate to $-2\xi\partial_{\xi}$. Instead, one encounters an operator $\mathcal{K}$ which is defined by the relation
 \begin{align*}
 \ha{R\partial_{R}u}=-2\xi\partial_{\xi}\ha{u}+\calK\ha{u}.
\end{align*}
Then in analogy to Theorem 5.1 in \cite{KST2}, we deduce the following important 
\begin{proposition}\label{prop:Kstructure} The operator $\mathcal{K}$ is given by 
\[
\mathcal{K}f(\xi) = -2f(\xi) + \mathcal{K}_0f(\xi),
\]
where $\mathcal{K}_0$ is an integral operator with kernel 
\[
\frac{\tilde{\rho}(\eta)F(\xi, \eta)}{\xi - \eta},
\]
i. e. we have 
\[
 \mathcal{K}_0f(\xi) = \int_0^\infty \frac{\tilde{\rho}(\eta)F(\xi, \eta)}{\xi - \eta} f(\eta)\,d\eta,
 \]
 where the integral is in the principal value sense. The function $F$ is $C^2$ on $(0,\infty)\times (0,\infty)$, and satisfies the bounds 
 \begin{equation}\label{eq:Fbound1}\begin{split}
 &|F(\xi, \eta)|\lesssim 1,\,\xi + \eta<1\\
 &|F(\xi, \eta)|\lesssim \langle \xi\rangle^{-\frac{5}{4}}\langle\eta\rangle^{-\frac54}\langle \xi-\eta\rangle^{-1},\,\xi+\eta\geq 1. 
 \end{split}\end{equation}
 \begin{equation}\label{eq:Fbound2}\begin{split}
 &|\partial_{\xi}F(\xi, \eta)|\lesssim \langle\log\xi\rangle^3,\,\xi<1\\
 &|\partial_{\xi}F(\xi, \eta)|\lesssim \langle\xi\rangle^{-\frac32}\langle\eta\rangle^{-1},\,\xi\geq 1. 
 \end{split}\end{equation}
 \begin{equation}\label{eq:Fbound3}\begin{split}
 &|\partial_{\eta}F(\xi, \eta)|\lesssim \langle\log\eta\rangle^3,\,\eta<1\\
 &|\partial_{\eta}F(\xi, \eta)|\lesssim \langle\eta\rangle^{-\frac32}\langle\xi\rangle^{-1},\,\eta\geq 1. 
 \end{split}\end{equation}
\end{proposition}

\begin{proof}
 The proof consists of two parts. In the first part we compute the diagonal part of $\calK$ and in the second part we focus on the off-diagonal part $\calK_{0}$. By its definition, for $f\in C^{\infty}_{0}((0,\infty))$, $\calK$ is given by

\begin{align}\label{operator K pre}
\begin{split}
 \calK f(\eta):=&\left\langle\int_{0}^{\infty}f(\xi)R\partial_{R}\phi(R,\xi)\trho(\xi)d\xi,\phi(R,\eta)\right\rangle_{L^{2}_{RdR}}\\
 &+\left\langle\int_{0}^{\infty}2\xi\partial_{\xi}f(\xi)\phi(R,\xi)\trho(\xi)d\xi,\phi(R,\eta)\right\rangle_{L^{2}_{RdR}}\\
 =&\left\langle\int_{0}^{\infty}f(\xi)[R\partial_{R}-2\xi\partial_{\xi}]\phi(R,\xi)\trho(\xi)d\xi,\phi(R,\eta)\right\rangle_{L^{2}_{RdR}}\\
 &-2\left(1+\frac{\eta\trho'(\eta)}{\trho(\eta)}\right)f(\eta).
\end{split}
 \end{align}
 Since $f\in C^{\infty}_{0}((0,\infty))$, we are able to restrict $\xi,\eta$ on a compact subset of $(0,\infty)$. When $R^{2}\xi\leq 1$, $\left|[R\partial_{R}-2\xi\partial_{\xi}]\phi(R,\xi)\right|\lesssim 1$. Therefore in this case, for fixed $\xi,\eta$ in a compact subset of $(0,\infty)$, the kernel in consideration is bounded. For the case when $R^{2}\xi\gtrsim 1$, we have, from \cite{KST2}, 
 
 \begin{align*}
 \phi_{KST}(R,\xi)=\frac{1}{2}\left(a(\xi)\xi^{-\frac{1}{4}}e^{iR\xi^{\frac{1}{2}}}\left(1+\frac{3i}{8R\xi^{\frac{1}{2}}}\right)+\overline{a(\xi)}\xi^{-\frac{1}{4}}e^{-iR\xi^{\frac{1}{2}}}\left(1-\frac{3i}{8R\xi^{\frac{1}{2}}}\right)\right)+O(R^{-2})
\end{align*}
Here $O(\cdot)$ depends on $\xi$. Our Fourier basis $\phi$ is given by $\phi(R,\xi)=\calD\left(R^{-\frac{1}{2}}\phi_{KST}(R,\xi)\right)$. We have

\begin{align*}
 R^{-\frac{1}{2}}\phi_{KST}(R,\xi)=\frac{R^{-\frac{1}{2}}}{2}\left(a(\xi)\xi^{-\frac{1}{4}}e^{iR\xi^{\frac{1}{2}}}+\overline{a(\xi)}\xi^{-\frac{1}{4}}e^{-iR\xi^{\frac{1}{2}}}\right)+O\left(R^{-\frac{3}{2}}\right).
\end{align*}
Therefore $\phi(R,\xi)$ has the following asymptotic behavior

\begin{align}\label{expansion for phi}
 \phi(R,\xi)=\xi^{-1}\calD\left(R^{-\frac{1}{2}}\phi_{KST}(R,\xi)\right)=\frac{R^{-\frac{1}{2}}}{2}\left(ia(\xi)\xi^{-\frac{3}{4}}e^{iR\xi^{\frac{1}{2}}}-i\overline{a(\xi)}\xi^{-\frac{3}{4}}e^{-iR\xi^{\frac{1}{2}}}\right)+O\left(R^{-\frac{3}{2}}\right).
\end{align}
 Now we apply $R\partial_{R}-2\xi\partial_{\xi}$ to $\phi(R,\xi)$ and track the leading order terms. When $R\partial_{R}$ hits $\frac{R^{-\frac{1}{2}}}{2}$, we obtain

\begin{align*}
 -\frac{R^{-\frac{1}{2}}}{4}\left(ia(\xi)\xi^{-\frac{3}{4}}e^{iR\xi^{\frac{1}{2}}}-i\overline{a(\xi)}\xi^{-\frac{3}{4}}e^{-iR\xi^{\frac{1}{2}}}\right).
\end{align*}
 When $R\partial_{R}-2\xi\partial_{\xi}$ hits $e^{iR^{2}\xi}$, we get zero. So the other leading order contribution is given by

\begin{align*}
 -{R^{-\frac{1}{2}}}\left(i\xi\partial_{\xi}\left(a(\xi)\xi^{-\frac{3}{4}}\right)e^{iR\xi^{\frac{1}{2}}}-i\xi\partial_{\xi}\left(\overline{a(\xi)}\xi^{-\frac{3}{4}}\right)e^{-iR\xi^{\frac{1}{2}}}\right).
\end{align*}
So we have the following expansion for $(R\partial_{R}-2\xi\partial_{\xi})\phi(R,\xi)$:

\begin{align}\label{expansion for K kernel}
\begin{split}
 &(R\partial_{R}-2\xi\partial_{\xi})\phi(R,\xi)\\
 =&-\frac{R^{-\frac{1}{2}}}{4}\left(i\left(a(\xi)\xi^{-\frac{3}{4}}+4\xi\partial_{\xi}\left(a(\xi)\xi^{-\frac{3}{4}}\right)\right)e^{iR\xi^{\frac{1}{2}}}-i\left(\overline{a(\xi)}\xi^{-\frac{3}{4}}+4\xi\partial_{\xi}\left(\overline{a(\xi)}\xi^{-\frac{3}{4}}\right)\right)e^{-iR\xi^{\frac{1}{2}}}\right)+O(R^{-\frac{3}{2}}).
 \end{split}
\end{align}
The kernel of the part of $\calK$ is formally given by

\begin{align}\label{K kernel 1}
 \int_{0}^{\infty}(R\partial_{R}-2\xi\partial_{\xi})\phi(R,\xi)\trho(\xi)\overline{\phi(R,\eta)}RdR.
\end{align}
The delta measure is from the following contribution:

\begin{align}\label{K kernel 2}
\begin{split}
& -\frac{1}{8}\int_{0}^{\infty}\left(a(\xi)\xi^{-\frac{3}{4}}+4\xi\partial_{\xi}\left(a(\xi)\xi^{-\frac{3}{4}}\right)\right)\trho(\xi)\overline{a(\eta)}\eta^{-\frac{3}{4}}e^{iR(\xi^{\frac{1}{2}}-\eta^{\frac{1}{2}})}dR\\
 &-\frac{1}{8}\int_{0}^{\infty}\left(\overline{a(\xi)}\xi^{-\frac{3}{4}}+4\xi\partial_{\xi}\left(\overline{a(\xi)}\xi^{-\frac{3}{4}}\right)\right)\trho(\xi)a(\eta)\eta^{-\frac{3}{4}}e^{-iR(\xi^{\frac{1}{2}}-\eta^{\frac{1}{2}})}dR,
 \end{split}
\end{align}
 which is given by
 
 \begin{align}\label{K kernel 3}
 \begin{split}
  &-\frac{1}{4}\int_{0}^{\infty}\Re\left(\left(a(\xi)\xi^{-\frac{3}{4}}+4\xi\partial_{\xi}\left(a(\xi)\xi^{-\frac{3}{4}}\right)\right)\trho(\xi)\overline{a(\eta)}\eta^{-\frac{3}{4}}\right)e^{iR(\xi^{\frac{1}{2}}-\eta^{\frac{1}{2}})}dR\\
 =&-\frac{\pi}{2}\Re\left(\left(a(\xi)\xi^{-\frac{3}{4}}+4\xi\partial_{\xi}\left(a(\xi)\xi^{-\frac{3}{4}}\right)\right)\overline{a(\eta)}\eta^{-\frac{3}{4}}\right)\trho(\xi)\delta(\xi^{\frac{1}{2}}-\eta^{\frac{1}{2}})\\
 =&-\pi\xi^{\frac{1}{2}}\trho(\xi)\Re\left(\left(a(\xi)\xi^{-\frac{3}{4}}+4\xi\partial_{\xi}\left(a(\xi)\xi^{-\frac{3}{4}}\right)\right)\overline{a(\xi)}\xi^{-\frac{3}{4}}\right)\delta(\xi-\eta)\\
 =&-\pi\xi^{\frac{3}{4}}\rho(\xi)\Re\left(-2|a(\xi)|^{2}+4\xi^{\frac{1}{4}}a'(\xi)\overline{a(\xi)}\right)\delta(\xi-\eta)\\
 =&-\pi\xi^{\frac{3}{4}}\rho(\xi)\Re\left(-2|a(\xi)|^{2}\xi^{-\frac{3}{4}}+2\xi^{\frac{1}{4}}\left(|a(\xi)|^{2}\right)'\right)\delta(\xi-\eta)\\
 =&2\left(1+\frac{\xi\rho'(\xi)}{\rho(\xi)}\right)\delta(\xi-\eta)=2\frac{\xi\trho'(\xi)}{\trho(\xi)}\delta(\xi-\eta).
 \end{split}
\end{align}
 Combining this with \eqref{operator K pre}, we have
 \begin{align}\label{operator K}
 (\calK f)(\eta)=-2f(\eta)+(\calK_{0}f)(\eta).
\end{align}
 
Next we turn to the off-diagonal part $\calK_{0}$. Arguing as in \cite{KST2}, we see that $F(\xi,\eta)$ is given by
 
 \begin{align}\label{F expression}
  \begin{split}
   F(\xi,\eta)=&\left\langle W(R)\phi(R,\xi),\phi(R,\eta)\right\rangle_{RdR},\\
   W(R)=&[\tcalL,R\partial_{R}]-2\tcalL=\frac{16}{(R^{2}+1)^{2}}-\frac{32}{(1+R^{2})^{3}}
  \end{split}
 \end{align}
We record the behavior of $\phi(R,\xi)$ as follows, relying on the asymptotics developed in Section 4 of \cite{KST2}: When $R^{2}\xi\lesssim 1$,

\begin{align}\label{phi behavior 1}
 \begin{split}
&  |\phi(R,\xi)|\lesssim R^{2}\lesssim \langle\xi\rangle^{-1},\quad\textrm{when}\quad R\lesssim 1,\\
 &|\phi(R,\xi)|\lesssim\log R,\quad \textrm{when}\quad R\gg1.
 \end{split}
\end{align}
When $R^{2}\xi\gtrsim 1$,

\begin{align}\label{phi behavior 2}
 \begin{split}
  &|\phi(R,\xi)|\lesssim R^{-\frac{1}{2}}\xi^{-\frac{1}{4}}|\log\xi|\lesssim |\log\xi|,\quad \textrm{when}\quad \xi\ll1,\\
  &|\phi(R,\xi)|\lesssim R^{-\frac{1}{2}}\xi^{-\frac{5}{4}},\quad \textrm{when}\quad \xi\gtrsim 1.
 \end{split}
\end{align}
We start with the case when \underline{$\xi$ and $\eta$ are close}. We split the integral defining $F(\xi,\eta)$ as follows:

\begin{align*}
 F(\xi,\eta)=&\int_{0}^{\infty}W(R) \phi(R,\xi)\phi(R,\eta)RdR\\
 =&\int_{0}^{\min\{\xi^{-\frac{1}{2}},1\}}+\int_{\min\{\xi^{-\frac{1}{2}},1\}}^{\max\{\xi^{-\frac{1}{2}},1\}}+\int_{\max\{\xi^{-\frac{1}{2}},1\}}^{\infty}:=I+II+III.
\end{align*}
For $I$ when $\xi\gtrsim1$, we use the first estimate in \eqref{phi behavior 1} to obtain

\begin{align}\label{F diagonal Ia}
 |I|\lesssim \xi^{-1}\eta^{-1}\int_{0}^{\xi^{-\frac{1}{2}}}\frac{R}{(1+R^{2})^{2}}dR\lesssim\langle\xi\rangle^{-2}\langle\eta\rangle^{-1}.
\end{align}
When $\xi\ll1$, we use the second bound in \eqref{phi behavior 1} to obtain

\begin{align}\label{F diagonal Ib}
 |I|\lesssim\int_{0}^{1}\frac{(\log R)^{2}R}{(1+R^{2})^{2}}dR\lesssim 1.
\end{align}
For $II$ when $\xi\lesssim R^{-2}\ll1$, we use the second of \eqref{phi behavior 1} to obtain

\begin{align}\label{F diagonal IIa}
 |II|\lesssim\int_{1}^{\xi^{-\frac{1}{2}}}\frac{(\log R)^{2}R}{(1+R^{2})^{2}}dR\lesssim 1.
\end{align}
When $\xi\gtrsim1$, we use the second of \eqref{phi behavior 2} to obtain

\begin{align}\label{F diagonal IIb}
 |II|\lesssim \xi^{-\frac{5}{4}}\eta^{-\frac{5}{4}}\int_{\xi^{-\frac{1}{2}}}^{1}\frac{1}{(1+R^{2})^{2}}dR\lesssim\langle\xi\rangle^{-\frac{5}{4}}\langle\eta\rangle^{-\frac{5}{4}}
\end{align}
For $III$ when $\xi\gtrsim1$, the estimate is similar to \eqref{F diagonal IIb}. When $\xi\ll1$, we use the first of \eqref{phi behavior 2} and the fact $|\log\xi|=\log\xi^{-1}\lesssim |\log R|$ to estimate in the same way as \eqref{F diagonal IIa}. \emph{This completes the discussion when $\xi$ and $\eta$ are close.}

Now we consider the case when \underline{$\xi$ and $\eta$ are far away.} We start with the following integration by parts:

\begin{align*}
 \eta F(\xi,\eta)=&\left\langle W(R)\phi(R,\xi),\tcalL\phi(R,\eta)\right\rangle_{RdR}\\
 =&\left\langle[\tcalL,W(R)]\phi(R,\xi),\phi(R,\eta)\right\rangle_{RdR}+\xi F(\xi,\eta),
\end{align*}
which gives 

\begin{align}\label{F off diagonal pre}
 (\eta-\xi)F(\xi,\eta)=&-\left\langle\left(2W_{R}\partial_{R}+W_{RR}+\frac{1}{R}W_{R}\right)\phi(R,\xi),\phi(R,\eta)\right\rangle_{RdR}.
\end{align}
Here the integration by parts is justified using the fact that $W(R)$ decays fast enough as $R\rightarrow\infty$ and $\phi(0,\cdot)=0$. In view of the fact 

\begin{align*}
 W_{R}(R)=O(R),\quad W_{RR}=O(1),\quad \textrm{when}\quad R\rightarrow0,
\end{align*}
we have

\begin{align}\label{F off diagonal 1}
\begin{split}
 \left|(\eta-\xi)F(\xi,\eta)\right|\lesssim& \left|\left\langle\frac{1}{(1+R^{2})^{2}}\phi(R,\xi),\phi(R,\eta)\right\rangle_{RdR}\right|\\
\lesssim&\langle\xi\rangle^{-\frac{5}{4}}\langle\eta\rangle^{-\frac{5}{4}}.
 \end{split}
 \end{align}
Here to derive the last inequality above, we use the same argument as deriving \eqref{F diagonal Ia}-\eqref{F diagonal IIb}.

Next we turn to the derivative of $F$. Due to the symmetry between $\xi$ and $\eta$, we only consider the derivative with respect to $\xi$. Here the idea is roughly the same as for estimating $F(\xi,\eta)$. We start by listing the pointwise bounds on $\partial_{\xi}\phi(R,\xi)$. When $R^{2}\xi\lesssim 1$ we have (using the Taylor expansion for $\phi(R,\xi)$ in $R^{2}\xi$)

\begin{align}\label{phi deri behavior 1}
 \begin{split}
  &|\partial_{\xi}\phi(R,\xi)|\lesssim\langle\xi\rangle^{-2},\quad \textrm{when}\quad R\lesssim 1,\\
  &|\partial_{\xi}\phi(R,\xi)|\lesssim R^{2}|\log R|\quad \textrm{when}\quad R\gg1.
 \end{split}
\end{align}
When $R^{2}\xi\gtrsim 1$, differentiating in $\xi$ naturally gives a factor of $\xi^{-1}$. On the other hand, when the derivative hits $e^{iR\xi^{\frac{1}{2}}}$, we get a factor of $R\xi^{-\frac{1}{2}}$ and $\xi^{-1}\lesssim R\xi^{-\frac{1}{2}}$. So we have

\begin{align}\label{phi deri behavior 2}
 \begin{split}
  &|\partial_{\xi}\phi(R,\xi)|\lesssim |\log\xi|R\xi^{-\frac{1}{2}},\quad \textrm{when}\quad \xi\ll1,\\
  &|\partial_{\xi}\phi(R,\xi)|\lesssim\xi^{-1}R\xi^{-\frac{1}{2}},\quad\textrm{when}\quad \xi\gtrsim 1. 
 \end{split}
\end{align}
As before we break the integral defining $\partial_{\xi}F(\xi,\eta)$ into three parts:

\begin{align*}
 \partial_{\xi}F(\xi,\eta)=&\int_{0}^{\infty}W(R) \partial_{\xi}\phi(R,\xi)\phi(R,\eta)RdR\\
 =&\int_{0}^{\min\{\xi^{-\frac{1}{2}},1\}}+\int_{\min\{\xi^{-\frac{1}{2}},1\}}^{\max\{\xi^{-\frac{1}{2}},1\}}+\int_{\max\{\xi^{-\frac{1}{2}},1\}}^{\infty}:=I'+II'+III'.
\end{align*}
For $I'$, when $\xi\gtrsim1$ we use the first of \eqref{phi deri behavior 1} to bound

\begin{align}\label{F off diagonal Ia}
 |I'|\lesssim \langle\xi\rangle^{-3}\langle\eta\rangle^{-1}.
\end{align}
When $\xi\ll1$, we use the second of \eqref{phi deri behavior 1} to bound

\begin{align}\label{F off diagonal Ib}
 |I'|\lesssim\int_{0}^{1}\frac{R^{3}(\log R)^{2}}{(1+R^{2})^{2}}dR\lesssim 1.
\end{align}
For $II'$, when $\xi\gtrsim 1$, we use the second of \eqref{phi deri behavior 2} to bound

\begin{align}\label{F off diagonal IIa}
 |II'|\lesssim \xi^{-\frac{3}{2}}\eta^{-1}\int_{\xi^{-\frac{1}{2}}}^{1}\frac{R^{2}}{(1+R^{2})^{2}}dR\lesssim\langle\xi\rangle^{-\frac{3}{2}}\langle\eta\rangle^{-1}.
\end{align}
When $\xi\ll1$, we use the second of \eqref{phi deri behavior 1} to bound

\begin{align}\label{F off diagonal IIb}
 |II'|\lesssim \int_{1}^{\xi^{-\frac{1}{2}}}\frac{R^{3}(\log R)^{2}}{(1+R^{2})^{2}}dR\lesssim |\log\xi|^{3}.
\end{align}
The estimate for $III'$ is similar to deriving \eqref{F off diagonal IIa} and \eqref{F off diagonal IIb}. This completes the proof.
\end{proof}
As an important consequence for us, we observe the following 
\begin{proposition}\label{prop:K_0Sbound} Introduce the norm 
\[
\big\|x\big\|_{S_0}: = \big\|\langle\xi\rangle^{2+\kappa}\xi^{\frac12}\langle\log\xi\rangle^{-1-\kappa}x(\xi)\big\|_{L^2_{d\xi}(0,\infty)}. 
\]
for some $\frac12>\kappa>0$. Then we have the bound 
\[
\big\|\mathcal{K}_0x\big\|_{S_0}\lesssim_{\kappa}\big\|x\big\|_{S_0}
\]
Moreover, there is the following smallness gain: letting $\epsilon>0$, there is $\gamma(\kappa)>0$, such that we have 
\[
\big\|\chi_{\xi<\epsilon}\mathcal{K}_0x\big\|_{S_0} + \big\|\mathcal{K}_0\chi_{\eta<\epsilon}x\big\|_{S_0}+ \big\|\chi_{\xi>\epsilon^{-1}}\mathcal{K}_0x\big\|_{S_0} + \big\|\mathcal{K}_0\chi_{\eta>\epsilon^{-1}}x\big\|_{S_0}\lesssim_{\kappa}\langle\log\epsilon\rangle^{-\gamma(\kappa)}\big\|x\big\|_{S_0}.
\]

\end{proposition}
 
 \section{Translation of \eqref{eq Dveps temp 2} to the Fourier side}
 
 Move the last term in \eqref{eq Dveps temp 2} to the left and apply the distorted Fourier transform with respect to $\tilde{\calL}$. Writing 
 \[
 \mathcal{D}\veps(\tau, R) = \int_0^\infty \xb(\tau, \xi)\phi(R, \xi)\tilde{\rho}(\xi)\,d\xi,
 \]
we obtain 
\begin{align}\label{eq on Fourier side temp 6}
 \begin{split}
  &-\left(\partial_{\tau}-\frac{\lambda'}{\lambda}2\xi\partial_{\xi}+\frac{\lambda'}{\lambda}\calK\right)^{2}\xb-3\frac{\lambda'}{\lambda}\left(\partial_{\tau}-\frac{\lambda'}{\lambda}2\xi\partial_{\xi}+\frac{\lambda'}{\lambda}\calK\right)\xb-\xi\xb-\left(2\left(\frac{\lambda'}{\lambda}\right)^{2}+\left(\frac{\lambda'}{\lambda}\right)'\right)\xb\\
  =&\calF\left(\lambda^{-2}\calD\left(N(\veps)\right)\right)+\calF\left(\calR(\veps,\calD\veps)\right).
 \end{split}
\end{align}
where throughout we shall use the notation (recalling \eqref{eq:phidef})
\[
\calF(g)(\xi) = \int_0^\infty g(R)\phi(R, \xi)\,R\,dR.
\]
Then introduce the dilation type operator
\begin{equation}\label{eq:D_tau}
\calD_{\tau}:=\partial_{\tau}-\frac{\lambda'}{\lambda}2\xi\partial_{\xi}-\frac{\lambda'}{\lambda}
\end{equation}
Recalling the preceding subsection, we then infer the equation
\begin{align}\label{eq on Fourier side temp 7}
 \begin{split}
  &-\left(\calD_{\tau}-\frac{\lambda'}{\lambda}+\frac{\lambda'}{\lambda}\calK_{0}\right)^{2}\xb-3\frac{\lambda'}{\lambda}\left(\calD_{\tau}-\frac{\lambda'}{\lambda}+\frac{\lambda'}{\lambda}\calK_{0}\right)\xb-\xi\xb-\left(2\left(\frac{\lambda'}{\lambda}\right)^{2}+\left(\frac{\lambda'}{\lambda}\right)'\right)\xb\\
  =&\calF\left(\lambda^{-2}\calD\left(N(\veps)\right)\right)+\calF\left(\calR(\veps,\calD\veps)\right)\\
  \Rightarrow\quad &-\left(\calD_{\tau}-\frac{\lambda'}{\lambda}\right)^{2}\xb-3\frac{\lambda'}{\lambda}\left(\calD_{\tau}-\frac{\lambda'}{\lambda}\right)\xb-\xi\xb-\left(2\left(\frac{\lambda'}{\lambda}\right)^{2}+\left(\frac{\lambda'}{\lambda}\right)'\right)\xb\\
  &=2\frac{\lambda'}{\lambda}\calK_{0}\calD_{\tau}\xb+\left(\frac{\lambda'}{\lambda}\right)'\calK_{0}\xb+\frac{\lambda'}{\lambda}[\calD_{\tau},\calK_{0}]\xb+\left(\frac{\lambda'}{\lambda}\right)^{2}\calK_{0}^{2}\xb\\
  &+\left(\frac{\lambda'}{\lambda}\right)^2\calK_{0}\xb+\calF\left(\lambda^{-2}\calD\left(N(\veps)\right)\right)+\calF\left(\calR(\veps,\calD\veps)\right)\\
  =:&\frakR(\tau,\xb)+\calF\left(\lambda^{-2}\calD\left(N(\veps)\right)\right)+\calF\left(\calR(\veps,\calD\veps)\right).
 \end{split}
\end{align}
So finally we obtain:

\begin{align}\label{eq on Fourier side final}
\begin{split}
 -\left(\calD^{2}_{\tau}+\frac{\lambda'}{\lambda}\calD_{\tau}+\xi\right)\xb=\frakR(\tau,\xb)+\calF\left(\lambda^{-2}\calD\left(N(\veps)\right)\right)+\calF\left(\calR(\veps,\calD\veps)\right)=:f(\tau,\xb),
 \end{split}
\end{align}
where the expression on the left arises after simplification of the expression on the left of \eqref{eq on Fourier side temp 7}. Note that this formula is the exact analogue of (3.4) in \cite{KS1}, and we can correspondingly use analogous formulae for the explicit solution of the corresponding linear homogeneous and inhomogeneous problem. In particular, 
\begin{align}\label{free wave with data}
 \left(\calD_{\tau}^{2}+\frac{\lambda'}{\lambda}\calD_{\tau}+\xi\right)\xb(\tau,\xi)=0;\quad \xb(\tau_{0},\xi)=\xb_{1}(\xi),\quad \calD_{\tau}\xb(\tau_{0},\xi)=\xb_{1}(\xi)
\end{align}
is solved by 
\begin{align}\label{sol to free wave}
 \begin{split}
  \xb(\tau,\xi)=&\frac{\lambda(\tau)}{\lambda(\tau_{0})}\cos\left(\lambda(\tau)\xi^{\frac{1}{2}}\int_{\tau_{0}}^{\tau}\lambda(u)^{-1}du\right)\xb_{0}\left(\frac{\lambda(\tau)^{2}}{\lambda(\tau_{0})^{2}}\xi\right)\\
 &+\xi^{-\frac{1}{2}}\sin\left(\lambda(\tau)\xi^{\frac{1}{2}}\int_{\tau_{0}}^{\tau}\lambda(u)^{-1}du\right)\xb_{1}\left(\frac{\lambda(\tau)^{2}}{\lambda(\tau_{0})^{2}}\xi\right).
 \end{split}
\end{align}
while the inhomogeneous equation with trivial data 
\begin{align*}
 \left(\calD_{\tau}^{2}+\frac{\lambda'}{\lambda}\calD_{\tau}+\xi\right)\xb(\tau,\xi)=f(\tau,\xi)
\end{align*}
is solved by 
\begin{align}\label{sol to inhomo xb}
 \begin{split}
\xb(\tau,\xi)=
\xi^{-\frac{1}{2}}\int_{\tau_{0}}^{\tau}\sin\left(\lambda(\tau)\xi^{\frac{1}{2}}\int_{\sigma}^{\tau}\lambda(u)^{-1}du\right)f\left(\sigma,\frac{\lambda(\tau)^{2}}{\lambda(\sigma)^{2}}\xi\right)d\sigma.
 \end{split}
\end{align}
We shall next develop a functional framework and deduce bounds for the homogeneous propagator given by \eqref{sol to free wave}. Specifically, we shall require a good weighted $L^\infty$-bound: 
\begin{proposition}\label{prop:PropagatorLinftyBound}
Let 
\[
\calD\veps(R) = \int_0^\infty \phi(R, \xi)\xb(\tau,\xi)\tilde{\rho}(\xi)\,d\xi,\,\veps(R) = \phi_0(R)\int_0^R\phi_0^{-1}(s)\calD\veps(s)\,ds.
\]
with $\xb(\tau,\xi)$ given by \eqref{sol to free wave}. Then we have the bound 
\[
\left\|\frac{\calD\veps(\tau, R)}{\langle\log R\rangle R}\right\|_{L^\infty_{dR}}\lesssim \left(\frac{\lambda(\tau)}{\lambda(\tau_0)}\right)^{-1}\left[\big\|\xb_{0}\big\|_{S_0} + \big\|\xb_{1}\big\|_{S_1}\right],
\]
where we recall Proposition~\ref{prop:K_0Sbound} for the definition of $\big\|\cdot\big\|_{S_0}$ and we set $\big\|\cdot\big\|_{S_1}: = \big\|\xi^{-\frac12}\cdot\big\|_{S_0}$. 
\end{proposition}
This proposition is in fact a consequence of the next proposition and lemma. 
In addition, we shall also require good energy type bounds:
\begin{proposition}\label{prop:PropagatorEnergy} Let $\xb(\tau,\xi)$ be given by \eqref{sol to free wave}. Then we have (here $\kappa>0$ is as in the definition of $\|\cdot\|_{S_0}$)
\begin{align*}
\sup_{\tau\geq \tau_0}\frac{\lambda(\tau)}{\lambda(\tau_0)}\left\langle\log\frac{\lambda(\tau)}{\lambda(\tau_0)}\right\rangle^{1+\kappa}[\big\|\xb(\tau, \cdot)\big\|_{S_0} + \big\|\mathcal{D}_{\tau}\xb(\tau, \cdot)\big\|_{S_1}]\lesssim \big\|\bar{x}_0\big\|_{S_0} + \big\|\bar{x}_1\big\|_{S_1}. 
\end{align*}
\end{proposition}

We observe from the last section that control over $\|\xb\|_{S_0} + \|\mathcal{D}_{\tau}\xb\|_{S_{1}}$ implies control over 
\[
\|(\veps, \partial_{\tau}\veps)\big\|_{H^{3+\nu-}\times H^{2+\nu-}}, 
\]
provided $\veps(\tau, R) = \int_0^\infty\phi(R, \xi)\xb(\tau, \xi)\tilde{\rho}(\xi)\,d\xi$.

\section{Bounds for the ODE controlling the evolution of the resonance, \eqref{ctau ODE}}

Note that the operator 
\[
L_c: =  \left(\partial_{\tau}+\frac{\lambda'}{\lambda}\right)^{2}+\frac{\lambda'}{\lambda}\left(\partial_{\tau}+\frac{\lambda'}{\lambda}\right)
\]
with $\lambda(\tau) = \tau^{-1-\nu^{-1}}$ admits the fundamental system
\[
\phi_1(\tau) = \tau^{-1-\nu^{-1}},\,\phi_2(\tau) = \tau^{-1-\frac{2}{\nu}}.
\]
Then the equation 
\[
L_cy(\tau) = r(\tau),\,y(\tau_0) = y'(\tau_0) = 0
\]
admits the explicit solution 
\begin{equation}\label{eq:L_cParametrix}
y(\tau) = \nu\left(\tau^{-1-\nu^{-1}}\int_{\tau_0}^{\tau}\sigma^{2+\nu^{-1}}r(\sigma)\,d\sigma - \tau^{-1-2\nu^{-1}}\int_{\tau_0}^{\tau}\sigma^{2+2\nu^{-1}}r(\sigma)\,d\sigma\right)
\end{equation}

\section{Fine structure of $\mathcal{D}\veps$ and setup of iteration scheme}

\subsection{A simple $L^\infty$-bound}
We shall be working with the representation 
\begin{equation}\label{eq:DepsFourierRep}
\mathcal{D}\veps(\tau, R) = \int_0^\infty \phi(R, \xi)\xb(\tau,\xi)\tilde{\rho}(\xi)\,d\xi
\end{equation}
In order to determine how much regularity for $\mathcal{D}\veps(\tau, R)$, respectively how much decay we need for $\xb$ (in terms of $\xi$), the following lemma is useful:

\begin{lemma}\label{lem:epsbound} Assume \eqref{eq:DepsFourierRep}. Then, setting $\veps = \phi_0(R)\int_0^R[\phi_0(s)]^{-1}\mathcal{D}\veps(\tau, s)\,ds$, we have the bound 
\[
\left\|\frac{\veps(\tau, R)}{\langle\log\langle R\rangle\rangle R}\right\|_{L^\infty_{dR}}\lesssim \left\|\langle \xi\rangle^{1+\kappa}\xi^{\frac12}\langle\log\xi\rangle^{-1-\kappa}\xb(\tau,\xi)\right\|_{L^2_{d\xi}}.
\]
\end{lemma}
\begin{proof} We distinguish between the case $R<1$ and $R>1$: 
\\
{\it{(1) $R<1$}}. Here we have $\frac{ \phi_0(R)}{R}\lesssim 1$, and further $s<1$ in the inner integral and $[\phi_0(s)]^{-1}\lesssim s^{-1}$. Then further restricting to $s^2\xi<1$ and switching orders of integration, we bound the corresponding contribution by 
\begin{align*}
\int_0^\infty \big|\xb(\tau,\xi)\big|\left(\int_0^{\min\{\xi^{-\frac12},1\}}\left|\frac{\phi(s, \xi)}{s}\right|\,ds\right)\tilde{\rho}(\xi)\,d\xi&\lesssim \int_0^1 \big|\xb(\tau,\xi)\big|\tilde{\rho}(\xi)\,d\xi+ \int_1^\infty  \big|\xb(\tau,\xi)\big|\xi^{-1}\tilde{\rho}(\xi)\,d\xi\\
&\lesssim \left\|\langle \xi\rangle^{1+\kappa}\xi^{\frac12}\langle\log\xi\rangle^{-1-\kappa}\xb(\tau,\xi)\right\|_{L^2_{d\xi}}.
\end{align*}
Here we have taken advantage of the bound 
\[
\int_0^{\min\{\xi^{-\frac12},1\}}\left|\frac{\phi(s, \xi)}{s}\right|\,ds\lesssim \xi^{-1}
\]
in the regime $\xi>1$. 
\\
On the other hand, restricting to $s^2\xi\geq 1$ (which entails $\xi>1$), we use $|\phi(s, \xi)|\lesssim s^{-\frac12}\xi^{-\frac54}$, giving 
\[
\int_{\xi^{-\frac12}}^1\left|\frac{\phi(s, \xi)}{s}\right|\,ds\lesssim \xi^{-1},
\]
 and from there 
 \[
 \int_0^\infty \big|\xb(\tau,\xi)\big|\left(\int_{\xi^{-\frac12}}^1\left|\frac{\phi(s, \xi)}{s}\right|\,ds\right)\tilde{\rho}(\xi)\,d\xi\lesssim \big\|\langle \xi\rangle^{1+}\xi^{\frac12}\langle\log\xi\rangle^{-1-\kappa}\xb(\tau,\xi)\big\|_{L^2_{d\xi}}.
 \]
 {\it{(2) $R\geq 1$}}. Here use $\frac{ \phi_0(R)}{R}\lesssim R^{-2}$. The contribution to the $s$-integral from the region $s<1$ is handled like in the preceding case. For $s>1$, use the bounds 
 \[
 |\phi(s, \xi)|\lesssim \langle\log s\rangle,\,|\phi(s,\xi)|\lesssim s^{-\frac12}\xi^{-\frac54},\,\xi>1.
 \]
 It follows that for $s>1$, using the Cauchy-Schwarz inequality, we have 
 \[
\int_0^\infty \big|\xb(\tau,\xi)\big|\big|\phi(s, \xi)\big|\tilde{\rho}(\xi)\,d\xi\lesssim \langle\log s\rangle \big\|\langle \xi\rangle^{1+\kappa}\xi^{\frac12}\langle\log\xi\rangle^{-1-\kappa}\xb(\tau,\xi)\big\|_{L^2_{d\xi}}.
\]
 This in turn implies
 \begin{align*}
\left| \frac{ \phi_0(R)}{R}\int_1^R[\phi_0(s)]^{-1}\mathcal{D}\veps(\tau, s)\,ds\right|&\lesssim R^{-2}\int_1^R s \left|\int_0^\infty \big|\xb(\tau,\xi)\big|\big|\phi(s, \xi)\big|\tilde{\rho}(\xi)\,d\xi\right|\,ds\\
&\lesssim \langle\log\langle R\rangle\rangle  \big\|\langle \xi\rangle^{1+}\xi^{\frac12}\langle\log\xi\rangle^{-1-}\xb(\tau,\xi)\big\|_{L^2_{d\xi}}.
 \end{align*}
\end{proof}

The preceding lemma suggests that one should require $\xi^{\frac32+}\xb(\tau, \xi)$ to be in $L^2_{d\xi}$. This, however, turns out not to be quite enough, on account of the very delicate source term $h(\tau)$ in the equation \eqref{ctau ODE}  for  $c(\tau)$. Recall that this term is given by 
\[
h(\tau) = \lim_{R\rightarrow 0}R^{-1}\mathcal{D}^*\mathcal{D}\veps(\tau, R).
\]
Expanding $\mathcal{D}\veps$ as above and using the Taylor expansion around $R = 0$ for $\phi(R, \xi)$ which starts with $cR^2$, we find 
\[
h(\tau) = c\int_0^\infty \xb(\tau,\xi)\tilde{\rho}(\xi)\,d\xi.
\]
Control of this quantity via Cauchy Schwarz requires a bound on $\big\|\xb(\tau, \cdot)\big\|_{S_0}$, the latter as defined in Proposition~\ref{prop:K_0Sbound}. While it is straightforward to obtain control over such a norm for the `zeroth iterate', i. e. the linear forward propagator given by \eqref{sol to free wave}, this appears not quite possible for the Duhamel propagator of some of the source terms contained in $\lambda^{-2}\mathcal{D}(N(\veps))$ on the right hand side in \eqref{eq Dveps temp 2}. In fact, those terms in $N(\veps)$ in \eqref{eq:epsequation1} depending on $u^\nu$ will only be of regularity $H^{1+\nu-}$, whence after application of $\mathcal{D}$ of regularity $H^{\nu-}$, which on the Fourier side translates to $\xi^{1+\frac{\nu}{2}-}\xb(\tau, \xi)\in L^2_{d\xi}$. Application of the inhomogeneous parametrix \eqref{sol to inhomo xb} improves this to $\xi^{\frac32+\frac{\nu}{2}-}\xb(\tau, \xi)\in L^2_{d\xi}$, which, however, falls much short of the space $\big\|\cdot\big\|_{S_0}$. 
\\

The way out of this impasse will be to exploit the fact that the `singularity' of $u^\nu$ only occurs on the light cone $R = \nu\tau$, and hence far from the origin, which allows one to isolate the part of the Fourier transform that decays more slowly, and exhibit a rapid temporal oscillation for it, depending on the frequency. 

\subsection{Structure of $u^{\nu}$, and consequences for the structure of $\mathcal{D}\veps$.}

We shall call a function $\xb(\tau, \xi)$ admissible for large $\xi$, provided we have for some large $N$ and suitable functions $a_{kj}(\tau), b(\tau, \xi)$ the representation
\begin{equation}\label{eq:admissiblexb}
\xb(\tau, \xi) = \chi_{\xi>1}\sum_{N\geq k\geq 1,N\geq j\geq 0}\sum_{\pm}a_{kj}^{\pm}(\tau)\frac{e^{\pm i\nu\tau\xi^{\frac12}}}{\xi^{2+\frac{k\nu}{2}}}(\log\langle \xi\rangle)^j +  \chi_{\xi>1}\frac{b(\tau, \xi)}{\xi^{\frac52+\frac{\nu}{2}-}} + x_{good}(\tau, \xi),
\end{equation}
where $b(\tau,\xi)$ admits the representation 
\[
b(\tau, \xi) = \partial_{\tau}c(\tau, \xi) + d(\tau, \xi),\,\big|c(\tau, \xi)\big| + \big|d(\tau, \xi)\big|\lesssim \frac{1}{\xi^{\frac12}}\cdot\tau^{-1}\frac{\lambda(\tau_0)}{\lambda(\tau)}\left\langle\log\frac{\lambda(\tau)}{\lambda(\tau_0)}\right\rangle^{-1-\frac{\kappa}{2}},
\]
and such that we have the bounds 
\begin{align*}
&\sum_{\pm}\tau\big|a_{kj}^{\pm}(\tau)\big| + \sum_{\pm}\tau\big|(a_{kj}^{\pm})'(\tau)\big| + \tau\big\|(|b|+\xi^{\frac12}|c|+\xi^{\frac12}|d|)(\tau, \xi)\big\|_{L_{\xi}^{\infty}}\\&\hspace{4cm} + \tau\big\|\langle\xi\rangle^{-\frac12}\mathcal{D}_{\tau}b(\tau, \xi)\big\|_{L_{\xi}^{\infty}} + \big\|x_{good}(\tau, \xi)\big\|_{S_0} + \big\|\mathcal{D}_{\tau}x_{good}(\tau, \xi)\big\|_{S_1}\\&<\left(\frac{\lambda(\tau)}{\lambda(\tau_0)}\right)^{-1}\left\langle\frac{\lambda(\tau)}{\lambda(\tau_0)}\right\rangle^{-1-\frac{\kappa}{2}}
\end{align*}
where the small positive constant $\kappa>0$ is as in the definition of $\|\cdot\|_{S_0}$. \\
The idea here is that the leading terms 
\[
 \sum_{N\geq k\geq 1,N\geq j\geq 0}\sum_{\pm}a_{kj}^{\pm}(\tau)\frac{e^{\pm i\nu\tau\xi^{\frac12}}}{\xi^{2+\frac{k\nu}{2}}}(\log\langle \xi\rangle)^j + \frac{b(\tau, \xi)}{\xi^{\frac52+\frac{\nu}{2}-}}
 \]
 essentially reflect the structure of the Fourier transform of $u^\nu$, localized to the singularity around the light cone (the function being smooth elsewhere). While these terms have weaker $\xi$ decay than $x_{good}(\tau, \xi)$, they have additional oscillatory structure which allows for cancellations in certain integrals. Specifically, we have the following elementary 
 \begin{lemma}\label{lem:cintegralbound} With $\xb$ as in \eqref{eq:admissiblexb} for $\xi>1$, and $\big\|\chi_{\xi<1}\xb(\tau, \xi)\big\|_{S_0}<\left(\frac{\lambda(\tau)}{\lambda(\tau_0)}\right)^{-1}\left\langle\log\frac{\lambda(\tau)}{\lambda(\tau_0)}\right\rangle^{-1-\frac{\kappa}{2}}$, we have the formula
 \[
 \int_0^\infty\xb(\tau, \xi)\tilde{\rho}(\xi)\,d\xi = \partial_{\tau}A(\tau) + B(\tau)
 \]
 where the expression $\partial_{\tau}A(\tau)$ is in the distributional sense (i. e. the function $A(\tau)$ is not necessarily $C^1$ but is in $C^0$) and $\tau\big|A(\tau)\big| + \big|B(\tau)\big|\lesssim \left(\frac{\lambda(\tau)}{\lambda(\tau_0)}\right)^{-1}\left\langle\log\frac{\lambda(\tau)}{\lambda(\tau_0)}\right\rangle^{-1-\frac{\kappa}{2}}$.
 
 \end{lemma}
\begin{proof} This follows from the fact that $\tilde{\rho}\sim \xi^2$ when $\xi>1$ respectively $\tilde{\rho}\sim \frac{1}{\log^2\xi}$ when $\xi<1$, and integration by parts with respect to $\xi$ for large frequencies and the first two terms in \eqref{eq:admissiblexb}.

\end{proof}

\subsection{A technical remark concerning regularisation} We recall that the reductions leading to the equation for $c(\tau)$ required sufficient regularity of $\veps(\tau, R)$ to ensure the quadratic vanishing 
\[
\mathcal{D}\veps(\tau, R) = O(R^2)
\]
and furthermore ensure the existence of the quantity $h(\tau)$ in \eqref{ctau ODE}, or at least justify the integration by parts in the preceding lemma. All of this can be easily dealt with by regularising the singular (on account of the presence of $u^{\nu}$) terms in $N(\veps)$, which means applying a large frequency cutoff to $\mathcal{F}\big(\lambda^{-2}\mathcal{D}(N(\veps))\big)$, and later passing to the limit. Since all the estimates below will then be uniformly satisfied, the existence of the limit will be clearly satisfied. We shall henceforth suppress this technicality. 

\section{Key nonlinear estimates}

To verify that the preceding ansatz is consistent with the structure of the nonlinearity, we now have the following 
\begin{proposition}\label{prop:mulitlin1} Let $N(\veps)$ be given as in \eqref{eq:epsequation1} with an additional cutoff $\chi_{R\lesssim \tau}$ in front, and assume for now that 
\[
\veps(\tau, R) = \phi\big(\mathcal{D}\veps(\tau, \cdot)\big)(R) = \phi_0(R)\int_0^R[\phi_0(s)]^{-1}\mathcal{D}\veps(\tau, s)\,ds,
\]
where we assume 
\[
\mathcal{D}\veps(\tau, R) = \int_0^\infty \phi(R, \xi)\xb(\tau, \xi)\tilde{\rho}(\xi)\,d\xi,
\]
and $\chi_{\xi>1}\xb(\tau, \xi)$ is as in \eqref{eq:admissiblexb}, while we have $\chi_{\xi<1}\xb(\tau, \xi)\in S_0$. Then we conclude that 
\begin{equation}\label{eq:Nepsilon1}
\tilde{x}(\tau, \xi): = \xi^{-\frac{1}{2}}\int_{\tau_{0}}^{\tau}\sin\left(\lambda(\tau)\xi^{\frac{1}{2}}\int_{\sigma}^{\tau}\lambda(u)^{-1}du\right)\mathcal{F}\big(\mathcal{D}(\lambda^{-2}N(\veps))\big)\left(\sigma,\frac{\lambda(\tau)^{2}}{\lambda(\sigma)^{2}}\xi\right)d\sigma
\end{equation} 
has a similar structure: there is a splitting 
\begin{align*}
\chi_{\xi>1}\tilde{x}(\tau, \xi) = \sum_{N\geq k\geq 1,N\geq j\geq 0}\sum_{\pm}\tilde{a}_{kj}^{\pm}(\tau)\frac{e^{\pm i\nu\tau\xi^{\frac12}}}{\xi^{2+\frac{k\nu}{2}}}(\log\langle \xi\rangle)^j + \frac{\tilde{b}(\tau, \xi)}{\xi^{\frac52+\frac{\nu}{2}-}} + \tilde{x}_{good}(\tau, \xi)
\end{align*}
such that $\tilde{b}$ admits a representation in terms of functions $\tilde{c}, \tilde{d}$ analogous to the one of $b$(as detailed after \eqref{eq:admissiblexb}), and such that 
\begin{align*}
&\sup_{\tau\geq \tau_0}\frac{\lambda(\tau)}{\lambda(\tau_0)}\left\langle\log\frac{\lambda(\tau)}{\lambda(\tau_0)}\right\rangle^{1+\frac{\kappa}{2}}\big[\sum_{\pm}\tau\big|\tilde{a}_{kj}^{\pm}(\tau)\big| + \sum_{\pm}\tau\big|(\tilde{a}_{kj}^{\pm})'(\tau)\big| + \tau\big\|(|\tilde{b}| + \xi^{\frac12}|\tilde{c}|+\xi^{\frac12}|\tilde{d}|)(\tau, \xi)\big\|_{L_{\xi}^{\infty}}\\&\hspace{5cm} + \tau\big\|\langle\xi\rangle^{-\frac12}\mathcal{D}_{\tau}\tilde{b}(\tau, \xi)\big\|_{L_{\xi}^{\infty}} + \big\|\tilde{x}_{good}(\tau, \xi)\big\|_{S_0} + \big\|\mathcal{D}_{\tau}\tilde{x}_{good}(\tau, \xi)\big\|_{S_1}\big]\\
&\ll \sup_{\tau\geq \tau_0}\frac{\lambda(\tau)}{\lambda(\tau_0)}\left\langle\log\frac{\lambda(\tau)}{\lambda(\tau_0)}\right\rangle^{1+\frac{\kappa}{2}}\big[\sum_{\pm}\tau\big|a_{kj}^{\pm}(\tau)\big| + \sum_{\pm}\tau\big|(a_{kj}^{\pm})'(\tau)\big| + \tau\big\|(|b|+\xi^{\frac12}|c| + \xi^{\frac12}|d|)(\tau, \xi)\big\|_{L_{\xi}^{\infty}}\\&\hspace{5cm} + \tau\big\|\langle\xi\rangle^{-\frac12}\mathcal{D}_{\tau}b(\tau, \xi)\big\|_{L_{\xi}^{\infty}} + \big\|x_{good}(\tau, \xi)\big\|_{S_0} + \big\|\mathcal{D}_{\tau}x_{good}(\tau, \xi)\big\|_{S_1}\big]
\end{align*}
provided $\tau_0\gg 1$, and moreover, we also have the bound 
\begin{align*}
&\sup_{\tau\geq \tau_0}\frac{\lambda(\tau)}{\lambda(\tau_0)}\left\langle\log\frac{\lambda(\tau)}{\lambda(\tau_0)}\right\rangle^{1+\frac{\kappa}{2}}\big\|\chi_{\xi<1}\tilde{x}(\tau, \xi)\big\|_{S_0}\\
&\ll \sup_{\tau\geq \tau_0}\frac{\lambda(\tau)}{\lambda(\tau_0)}\left\langle\log\frac{\lambda(\tau)}{\lambda(\tau_0)}\right\rangle^{1+\frac{\kappa}{2}}\big[\sum_{\pm}[\big|a_{kj}^{\pm}(\tau)\big| + \big|(a_{kj}^{\pm})'(\tau)\big|] + \big\|(|b|+\xi^{\frac12}|c| + \xi^{\frac12}|d|)(\tau, \xi)\big\|_{L_{\xi}^{\infty}}\\&\hspace{5cm} + \big\|\langle\xi\rangle^{-\frac12}\mathcal{D}_{\tau}b(\tau, \xi)\big\|_{L_{\xi}^{\infty}} + \big\|x_{good}(\tau, \xi)\big\|_{S_0} + \big\|\mathcal{D}_{\tau}x_{good}(\tau, \xi)\big\|_{S_1}\big]
\end{align*}

\end{proposition}

\begin{proof} It consists in controlling the interactions of the various constituents of $\mathcal{D}\veps$ as well as $u^{\nu}$ in the various source terms. To begin with, we consider the term linear in $\veps$, given by 
\begin{align*}
\mathcal{D}\left(\frac{\cos(2u^\nu) - \cos(2Q(\lambda(t)r)}{R^2}\veps\right) & = \mathcal{D}\left(\frac{\cos(2u^\nu) - \cos(2Q(\lambda(t)r)}{R^2}\right) \veps\\
& + \frac{\cos(2u^\nu) - \cos(2Q(\lambda(t)r)}{R^2}\mathcal{D}\veps
\end{align*}

{\it{(1): Contribution of first term in \eqref{eq:epsequation1}}}.\\
Observe that after restricting smoothly to $R\leq \frac{\nu\tau}{2}$, the functions 
\[
\chi_{R\leq \frac{\nu\tau}{2}}\mathcal{D}\left(\frac{\cos(2u^\nu) - \cos(2Q(\lambda(t)r)}{R^2}\right)\sim \chi_{R\leq \frac{\nu\tau}{2}}\tau^{-2}\langle \log\langle R\rangle\rangle R^{-3},
\]
\[
\chi_{R\leq \frac{\nu\tau}{2}}\left(\frac{\cos(2u^\nu) - \cos(2Q(\lambda(t)r)}{R^2}\right)\sim \chi_{R\leq \frac{\nu\tau}{2}}\tau^{-2}\langle \log\langle R\rangle\rangle R^{-2}
\]
for large $R$, and the corresponding contributions shall be straightforward to estimate. We shall henceforth restrict the source terms to the region $R>\frac{\nu\tau}{2}$. To convert the structure of the Fourier coefficients $\xb(\tau, \xi)$ of $\mathcal{D}\veps$ to the `physical side', we have 
\begin{lemma}\label{lem:FourierToPhysical1} Let 
\[
\xb(\tau, \xi) = \sum_{N\geq k\geq 1,N\geq j\geq 0}\sum_{\pm}a_{kj}^{\pm}(\tau)\frac{e^{\pm i\nu\tau\xi^{\frac12}}}{\xi^{2+\frac{k\nu}{2}}}(\log\langle \xi\rangle)^j. 
\]
Then we have 
\begin{align*}
\int_0^\infty \phi(R, \xi)\xb(\tau, \xi)\tilde{\rho}(\xi)\,d\xi &= \chi_{R\gtrsim \tau}R^{-\frac12}\sum_{N\geq k\geq 1,N\geq l\geq 0}b_{kl}(\tau)(\nu\tau - R)^{\frac12+k\nu}(\log(\nu\tau-R))^l\\& +  \chi_{R\gtrsim \tau}R^{-\frac32}\sum_{N\geq k\geq 1,N\geq l\geq 0}c_{kl}(\tau)(\nu\tau - R)^{\frac32+k\nu}(\log(\nu\tau-R))^l\\
& + f_{smooth}(\tau, R), 
\end{align*}
where we have the bounds 
\[
\sum_{k,l}[\big|b_{kl}(\tau)\big| + \big| b'_{kl}(\tau)\big| + \tau^{-1}\big|c_{kl}(\tau)\big| + \tau^{-1}\big| c'_{kl}|(\tau)]\lesssim \sum_{k,l}[\big|a_{kl}^{\pm}(\tau)\big| + \big| (a_{kl}^{\pm})'|(\tau)],
\]
and also $ f_{smooth}(\tau, R)\in H^{3+\nu-}_{R\,dR}$. In fact, including an extra smooth cutoff $\chi_{|R-\nu\tau|\lesssim 1}$ in front of the two sums of singular terms and modifying $f_{smooth}$ accordingly, we also have 
\[
\big\|f_{smooth}(\tau, \cdot)\big\|_{H^{3+\nu-}_{R\,dR}}\lesssim  \sum_{k,l}\sum_{\pm}\big|a_{kl}^{\pm}(\tau)\big|.
\]
\end{lemma}
\begin{proof}(lemma) To begin with, one notes that if $\chi_{y<b}(y)$ smoothly truncates to the region $y<b$, then (with $\xb$ as in the statement of the lemma)
\[
\chi_{R<\frac{\nu\tau}{2}}\int_1^\infty \phi(R, \xi)\xb(\tau, \xi)\tilde{\rho}(\xi)\,d\xi
\]
is a $C^\infty$-function and in particular can be included in $f_{smooth}$; this follows from repeated integration by parts with respect to $\xi$. In the region $R\gtrsim \tau$, and using $\xi>1$, we can expand 
\[
\phi(R, \xi) = \sum_{\pm}R^{-\frac12}\frac{e^{\pm iR\xi^{\frac12}}}{\xi^{\frac54}}\sigma_{\pm}(R\xi^{\frac12}, R),
\]
where $\sigma_{\pm}$ has a Hankel type expansion just as in \cite{KST2}, with leading term a non-vanishing constant $c_{\pm}$. Thus to leading order we obtain 
\begin{align*}
\chi_{R\gtrsim\frac{\nu\tau}{2}}\int_1^\infty \phi(R, \xi)\xb(\tau, \xi)\tilde{\rho}(\xi)\,d\xi & = \sum_{\pm}\sum_{kj}\chi_{R\gtrsim\frac{\nu\tau}{2}}R^{-\frac12}\int_1^\infty c_{\pm}a_{kj}^{\pm}(\tau)\frac{e^{\pm i(R-\nu\tau)\xi^{\frac12}}}{\xi^{\frac54 + 2 + \frac{k\nu}{2}}}(\log \xi)^j\tilde{\rho}(\xi)\,d\xi\\
& + \text{l.o.t.}
\end{align*}
Then we use the somewhat delicate asymptotics as $\xi\rightarrow\infty$ for the spectral measure given by (see Prop. 4.6, Prop. 4.7 in \cite{KST2})
\[
\tilde{\rho}(\xi) = c\xi^2\big(1+d\xi^{-\frac12} + O(\xi^{-1}\log \xi)\big).
\]
We infer the relation 
\begin{align*}
\chi_{R\gtrsim\frac{\nu\tau}{2}}\int_1^\infty \phi(R, \xi)\xb(\tau, \xi)\tilde{\rho}(\xi)\,d\xi & = R^{-\frac12}\sum_{kj}b_{kj}(\tau)(\nu\tau - R)^{\frac12 + k\nu}(\log(\nu\tau - R))^j\\
& +R^{-\frac12}\sum_{kj}b_{kj}(\tau)(\nu\tau - R)^{\frac32 + k\nu}(\log(\nu\tau - R))^j  + f_{smooth}
\end{align*}
The next higher order term in the Hankel expansion of the symbol $\sigma_{\pm}(R\xi^{\frac12}, R)$, which is of the form $(R\xi^{\frac12})^{-1}\psi_1^+(R)$ (see \cite{KST2}), is seen to lead to a term of the form 
\[
\chi_{R\gtrsim \tau}R^{-\frac32}\sum_{N\geq k\geq 1,N\geq l\geq 0}c_{kl}(\tau)(\nu\tau - R)^{\frac32+k\nu}(\log(\nu\tau-R))^l,
\]
and the remaining errors are all of the form $f_{smooth}$. 
\end{proof}

Consider now terms of the form 
\begin{equation}\label{eq:linearineps}
\chi_{R\gtrsim \tau}\mathcal{D}\left(\frac{\cos(2u^\nu) - \cos(2Q(\lambda(t)r)}{R^2}\right) \veps,\, \chi_{R\gtrsim \tau}\frac{\cos(2u^\nu) - \cos(2Q(\lambda(t)r)}{R^2}\mathcal{D}\veps
\end{equation}
and $\veps$ as in the proposition. For the first term, using the asymptotic expansion of $u^{\nu} - Q(R)$ in \cite{KST2,GaoKr}, we can write, with $a=\frac{R}{\nu\tau}$,
\begin{equation}\label{eq:expansionfromkst}\begin{split}
\chi_{R\gtrsim \tau}\mathcal{D}\left(\frac{\cos(2u^\nu) - \cos(2Q(\lambda(t)r)}{R^2}\right) &= \chi_{R\gtrsim \tau}\frac{1}{\tau^2}O\left(\frac{\log R}{R^2}\right)\sum_{k, j\leq N,\,k\geq 1}(1-a)^{-\frac12 + k\nu}\left(\log(1-a)\right)^j\\
& + g_{smooth}
\end{split}\end{equation}
where the function $O\left(\frac{\log R}{R^2}\right)$ is $C^\infty$ and obeys symbol type behavior with respect to $R$. Furthermore, we have 
\[
\big\|g_{smooth}\big\|_{H^{3+}_{R\,dR}}\lesssim \tau^{-2}. 
\]
Then we distinguish among the three different types of terms in \eqref{eq:admissiblexb} whose sum constitutes the Fourier transform $\xb(\tau, \xi)$ of $\mathcal{D}\veps(\tau, R)$. 
\\

{\it{(i) Assume $\xb(\tau, \xi) = \sum_{N\geq k\geq 1,N\geq j\geq 0}\sum_{\pm}a_{kj}^{\pm}(\tau)\frac{e^{\pm i\nu\tau\xi^{\frac12}}}{\xi^{2+\frac{k\nu}{2}}}(\log\langle \xi\rangle)^j$.}} Due to Lemma~\ref{lem:FourierToPhysical1}, we have 
\begin{align*}
\mathcal{D}\veps(\tau, R) &=   \chi_{|R-\nu\tau|\lesssim 1}R^{-\frac12}\sum_{N\geq k\geq 1,N\geq l\geq 0}b_{kl}(\tau)(\nu\tau - R)^{\frac12+k\nu}(\log(\nu\tau-R))^l\\& +   \chi_{|R-\nu\tau|\lesssim 1}R^{-\frac32}\sum_{N\geq k\geq 1,N\geq l\geq 0}c_{kl}(\tau)(\nu\tau - R)^{\frac32+k\nu}(\log(\nu\tau-R))^l\\
& + f_{smooth}(\tau, R), 
\end{align*}
with the bounds indicated in that lemma. From this, we easily conclude that 
\begin{align*}
\veps(\tau, R) &= \chi_{|R-\nu\tau|\lesssim 1}R^{-\frac12}\sum_{N\geq k\geq 1,N\geq l\geq 0}\tilde{b}_{kl}(\tau)(\nu\tau - R)^{\frac32+k\nu}(\log(\nu\tau-R))^l\\& +   \chi_{|R-\nu\tau|\lesssim 1}R^{-\frac32}\sum_{N\geq k\geq 1,N\geq l\geq 0}\tilde{c}_{kl}(\tau)(\nu\tau - R)^{\frac52+k\nu}(\log(\nu\tau-R))^l\\
& + \tilde{f}_{smooth}(\tau, R) 
\end{align*}
with 
\begin{align*}
\big\| \tilde{f}_{smooth}(\tau, \cdot)\big\|_{H^{4+}_{R\,dR}}\lesssim \tau \sum_{k,l}\sum_{\pm}\big|a_{kl}^{\pm}(\tau)\big|,\,\big\| \tilde{f}_{smooth}(\tau, \cdot)\big\|_{L^\infty_{R\,dR}}\lesssim \tau^{\frac12}\sum_{k,l}\sum_{\pm}\big|a_{kl}^{\pm}(\tau)\big|.
 \end{align*}
Using the preceding representation for $\veps$ in the first term in \eqref{eq:linearineps} and further invoking \eqref{eq:expansionfromkst}, we find the representation 
\begin{equation}\label{eq:messyformula1}\begin{split}
&\chi_{R\gtrsim \tau}\mathcal{D}\left(\frac{\cos(2u^\nu) - \cos(2Q(\lambda(t)r)}{R^2}\right) \veps\\
=& \tilde{f}_{smooth}(\tau, R)\left(\chi_{R\gtrsim \tau}\frac{1}{\tau^2}O\left(\frac{\log R}{R^2}\right)\sum_{k, j\leq N,\,k\geq 1}(1-a)^{-\frac12 + k\nu}\big(\log(1-a)\big)^j
 + g_{smooth}\right)\\
& + \frac{1}{\tau^2}O\left(\frac{\log R}{R^2}\right) \chi_{|R-\nu\tau|\lesssim 1}\sum_{k, j\leq 2N,\,k\geq 1}\frac{d_{kl}(\tau)}{\tau^{-1 + k\nu}}(1-a)^{1+ k\nu}\big(\log(1-a)\big)^l\\
& +  \frac{1}{\tau^2}O\left(\frac{\log R}{R^2}\right) \chi_{|R-\nu\tau|\lesssim 1}\sum_{k, j\leq 2N,\,k\geq 1}\frac{e_{kl}(\tau)}{\tau^{-2 + k\nu}}(1-a)^{2+ k\nu}\big(\log(1-a)\big)^l\\
\end{split}\end{equation}
where we still have the bound 
\[
\sum_{kl}\big|d_{kl}(\tau)\big| + \tau^{-1}\big|e_{kl}(\tau)\big| \lesssim \sum_{kl}\sum_{\pm}\big|a_{kl}^{\pm}(\tau)\big|. 
\]
We now revert to the distorted Fourier transform of each of these expressions, which partly means `inverting' the preceding lemma. Label the three expressions on the right as $A = \tilde{f}_{smooth}(\tau, R)\big(\ldots\big)$, $B =  \frac{1}{\tau^2}O\left(\frac{\log R}{R^2}\right)\ldots$, $C =  \frac{1}{\tau^2}O\left(\frac{\log R}{R^2}\right) $. For the term $A$, write $\tilde{f}_{smooth}(\tau, R)\chi_{R\gtrsim \tau}\frac{1}{\tau^2}O\left(\frac{\log R}{R^2}\right) = \tilde{g}_{smooth}(\tau, R)$. Then split 
\begin{align*}
A &= [\tilde{g}_{smooth}(\tau, \nu\tau) + \partial_R \tilde{g}_{smooth}(\tau, \nu\tau)(R - \nu\tau)]\big(\ldots\big)\\& + [\tilde{g}_{smooth}(\tau, R) - \tilde{g}_{smooth}(\tau, \nu\tau) - \partial_R \tilde{g}_{smooth}(\tau, \nu\tau)(R - \nu\tau)]\big(\ldots\big)\\
& = : A_1 + A_2.
\end{align*}
where $\big(\ldots\big)$ denotes the function in $a$. Then one easily checks that 
\begin{equation}\label{eq:A1expansion}\begin{split}
\mathcal{F}\big(\chi_{R\lesssim \tau}A_1(\tau, \cdot)\big)(\xi) &= \sum_{1\leq k,j\leq N}\sum_{\pm}h_{kl}^{\pm}(\tau)\chi_{\xi>1}e^{\pm i\nu\tau\xi^{\frac12}}\xi^{-\frac32 - k\frac{\nu}{2}}\langle\log\xi\rangle^j \\
&+ \sum_{1\leq k,j\leq N}\sum_{\pm}i_{kl}^{\pm}(\tau)\chi_{\xi>1}e^{\pm i\nu\tau\xi^{\frac12}}\xi^{-\frac52 - k\frac{\nu}{2}}\langle\log\xi\rangle^j+ \alpha_1(\tau, \xi),
\end{split}\end{equation}
where we have the bounds 
\[
\sum_{1\leq k,j\leq N}\sum_{\pm}\left(\big|h_{kl}^{\pm}(\tau)\big| +\big|i_{kl}^{\pm}(\tau)\big|\right)   + \sum_{1\leq k,j\leq N}\sum_{\pm}\left(\big|(h_{kl}^{\pm})'(\tau)\big| + \big|(i_{kl}^{\pm})'(\tau)\big|\right)\lesssim \tau^{-3+} \sum_{k,l}\sum_{\pm}\big|a_{kl}^{\pm}(\tau)\big|,
\]
\[
\big\| \langle\xi\rangle^{\frac52+\frac{\nu}{2}-}\alpha_1(\tau, \xi)\big\|_{L^2_{d\xi}}\lesssim \tau^{-3+} \sum_{k,l}\sum_{\pm}\big|a_{kl}^{\pm}(\tau)\big|.
\]
Substituting the first two terms in \eqref{eq:A1expansion} into the Duhamel parametrix \eqref{eq:Nepsilon1} results in the expression 
\begin{align*}
\sum_{1\leq k,l\leq N}\sum_{\pm}\tilde{a}_{kl}^{\pm}(\tau)\frac{e^{\pm i\nu\tau\xi^{\frac12}}}{\xi^{2+\frac{k\nu}{2}}}\langle\log\xi\rangle^l + \frac{\tilde{b}(\tau, \xi)}{\xi^{\frac52+\frac{\nu}{2}}}
\end{align*}
with the bound
\begin{align*}
\sum_{1\leq k,l\leq N}\sum_{\pm}\frac{\lambda(\tau)}{\lambda(\tau_0)}\left\langle\log\frac{\lambda(\tau)}{\lambda(\tau_0)}\right\rangle^{1+\frac{\kappa}{2}}\tau\big|\tilde{a}_{kl}^{\pm}(\tau)\big|&\lesssim \sum_{1\leq k,l\leq N}\sum_{\pm}\frac{\lambda(\tau)}{\lambda(\tau_0)}\left\langle\log\frac{\lambda(\tau)}{\lambda(\tau_0)}\right\rangle^{1+\frac{\kappa}{2}}\tau\int_{\tau_0}^{\tau}\big|h_{kl}^{\pm}(\sigma)\big|\frac{\lambda^2(\sigma)}{\lambda^2(\tau)}\,d\sigma\\
&\ll \sup_{\tau\geq \tau_0}\sum_{1\leq k,l\leq N}\sum_{\pm}\frac{\lambda(\tau)}{\lambda(\tau_0)}\left\langle\log\frac{\lambda(\tau)}{\lambda(\tau_0)}\right\rangle^{1+\frac{\kappa}{2}}\left[\big|a_{kl}^{\pm}(\tau)\big|+\big|(a_{kl}^{\pm})'(\tau)\big|\right]
\end{align*}
Moreover, one checks after some integrations by parts that the expression $\frac{\tilde{b}(\tau, \xi)}{\xi^{\frac52 + \frac{\nu}{2}}}$ is a linear combination of terms of the form 
\begin{align*}
d_{kl}^{\pm}(\tau,\xi)e^{\pm i\nu\tau\xi^{\frac12}}\xi^{-\frac52 - k\nu}\langle\log\xi\rangle^l,\,h_{kl}^{\pm}(\tau_0)e^{\pm i\nu\tau\xi^{\frac12}}e^{\pm2 i\nu\tau_0\frac{\lambda(\tau)}{\lambda(\tau_0)}\xi^{\frac12}}\xi^{-\frac52 - k\nu}\langle\log\xi\rangle^l 
\end{align*}
where we have schematically 
\[
d_{kl}^{\pm}(\tau,\xi) = h_{kl}^{\pm}(\tau) + \int_{\tau_0}^{\tau}e^{\pm2 i\nu\sigma\frac{\lambda(\tau)}{\lambda(\sigma)}\xi^{\frac12}}[h_{kl}^{\pm}(\sigma) + (h_{kl}^{\pm})'(\sigma)]\left(\frac{\lambda(\sigma)}{\lambda(\tau)}\right)^{3+k\nu}\left(\log\frac{\lambda(\tau)}{\lambda(\sigma)}\right)^l\,d\sigma
\]
as well as similar expressions involving $i_{kl}^{\pm}$ instead of $h_{kl}^{\pm}$, and so the required assertions about $\tilde{b}(\tau, \xi)$ follow easily from the preceding bounds on $h_{kl}^{\pm}, i_{kl}^{\pm}$. 
On the other hand, substituting the third term $\alpha_1(\tau, \xi)$ in \eqref{eq:A1expansion} into the parametrix \eqref{eq:Nepsilon1} is easily seen to result in a term with the properties of $\tilde{x}_{good}(\tau, \xi)$ in the statement of the proposition. 
\\
Next, we consider the contribution of $A_2$(in the above decomposition of $A$ into two parts). Here we claim that the corresponding contribution can be incorporated into $\tilde{x}_{good}(\tau, \xi)$. Given that the function of $a$ in $A$ is just barely of class $L^2_{R\,dR}$, we have to gain three degrees of differentiability to land in $H^{3+\nu-}$, one of which comes from the Duhamel parametrix itself. For the remaining two degrees of freedom, consider\footnote{Here $\big(\ldots\big)$ denotes the singular function of $a$.} 
\begin{align*}
&\partial_R^2\big([\tilde{g}_{smooth}(\tau, R) - \tilde{g}_{smooth}(\tau, \nu\tau) - \partial_R \tilde{g}_{smooth}(\tau, \nu\tau)(R - \nu\tau)]\big(\ldots\big)\big)\\
& = \partial_R^2[\tilde{g}_{smooth}(\tau, R) - \tilde{g}_{smooth}(\tau, \nu\tau) - \partial_R \tilde{g}_{smooth}(\tau, \nu\tau)(R - \nu\tau)]\big(\ldots\big)\\
& + 2\partial_R[\tilde{g}_{smooth}(\tau, R) - \tilde{g}_{smooth}(\tau, \nu\tau) - \partial_R \tilde{g}_{smooth}(\tau, \nu\tau)(R - \nu\tau)]\partial_R\big(\ldots\big)\\
& +[\tilde{g}_{smooth}(\tau, R) - \tilde{g}_{smooth}(\tau, \nu\tau) - \partial_R \tilde{g}_{smooth}(\tau, \nu\tau)(R - \nu\tau)] \partial_R^2\big(\ldots\big)\\
\end{align*}
Then we get 
\begin{align*}
&\left\|\partial_R^2[\tilde{g}_{smooth}(\tau, R) - \tilde{g}_{smooth}(\tau, \nu\tau) - \partial_R \tilde{g}_{smooth}(\tau, \nu\tau)(R - \nu\tau)]\big(\ldots\big)\right\|_{L^2_{R\,dR}(R\sim \tau)}\\
&\lesssim \big\|\partial_R^2[\tilde{g}_{smooth}(\tau, \cdot)\big\|_{L^\infty_{R\,dR}(\tau\sim R)}\big\|\big(\ldots\big)\big\|_{L^2_{R\,dR}(R\sim \tau)}\\
&\lesssim \tau^{-2+}\sum_{k,l}\sum_{\pm}\big|a_{kl}^{\pm}(\tau)\big|.
\end{align*}
On the other hand, for the last term above, we have 
\begin{align*}
&\left\|[\tilde{g}_{smooth}(\tau, R) - \tilde{g}_{smooth}(\tau, \nu\tau) - \partial_R \tilde{g}_{smooth}(\tau, \nu\tau)(R - \nu\tau)] \partial_R^2\big(\ldots\big)\right\|_{L^2_{R\,dR}(R\sim \tau)}\\
&\lesssim \left\|\frac{[\ldots]}{(R-\nu\tau)^2}\right\|_{L^\infty_{R\,dR}(R\sim \tau)}\left\|(R-\nu\tau)^2\partial_R^2\big(\ldots\big)\right\|_{L^2_{R\,dR}(R\sim \tau)}
\end{align*}
Here we estimate the first factor via Sobolev embedding in $1$ dimension since we are in the radial context and localized away from the origin, which gives 
\begin{align*}
\left\|\frac{[\ldots]}{(R-\nu\tau)^2}\right\|_{L^\infty_{R\,dR}(R\sim \tau)}\lesssim \big\|\tilde{g}_{smooth}(\tau, R)\big\|_{H^3_{R\,dR}}
\end{align*}
In light of the definition of $\tilde{g}_{smooth}$ further above, we then infer the bound 
\begin{align*}
&\left\|\frac{[\ldots]}{(R-\nu\tau)^2}\right\|_{L^\infty_{R\,dR}(R\sim \tau)}\big\|(R-\nu\tau)^2\partial_R^2\big(\ldots\big)\big\|_{L^2_{R\,dR}(R\sim \tau)}\\
&\lesssim \tau^{-3}\sum_{k,l}\sum_{\pm}\big|a_{kl}^{\pm}(\tau)\big|
\end{align*}
The term with mixed derivatives above is estimated similarly, and so we conclude that 
\begin{align*}
&\big\|\partial_R^2\big([\tilde{g}_{smooth}(\tau, R) - \tilde{g}_{smooth}(\tau, \nu\tau) - \partial_R \tilde{g}_{smooth}(\tau, \nu\tau)(R - \nu\tau)]\big(\ldots\big)\big)\big\|_{L^2_{R\,dR}(R\sim \tau)}\\
&\lesssim  \tau^{-2+}\sum_{k,l}\sum_{\pm}\big|a_{kl}^{\pm}(\tau)\big|.
\end{align*}
In fact, one can apply $\nu-$ additional fractional derivatives here; it follows after straightforward considerations that substituting  $\chi_{R\sim \tau}A_2$ for $\mathcal{D}(N(\veps))$ in the Duhamel parametrix \eqref{eq:Nepsilon1}, the corresponding contribution to $\tilde{x}(\tau, \xi)$ can be absorbed into $\tilde{x}_{good}(\tau, \xi)$. 
\\
This concludes dealing with the contribution of $A$ in \eqref{eq:messyformula1} except for the contribution of the error term $\tilde{f}_{smooth}\cdot g_{smooth}$, which, however, is easily seen to lead to a term contributing to $\tilde{x}_{smooth}(\tau, \xi)$. 
\\
The contributions of the remaining terms $B, C$ are of course handled analogously and in fact are better in terms of smoothness, and can be incorporated into the terms $\frac{\tilde{b}(\tau, \xi)}{\xi^{\frac52 + \frac{\nu}{2}}}$, $\tilde{x}_{smooth}(\tau, \xi)$. 
\\

{\it{(ii) Assume $\xb(\tau, \xi) = \chi_{\xi>1}\frac{b(\tau, \xi)}{\xi^{\frac52+\frac{\nu}{2}}}$}} We use similar arguments as in the preceding case. Note that then $\mathcal{D}\veps\in H^{2+\nu-}_{R\,dR}$, which implies $\veps\in H^{3+\nu-}_{R\,dR}$. We cannot conclude that $\chi_{R\ll\tau}\mathcal{D}\veps\in C^\infty$ without invoking finer structure of $b$. However, this regularity is enough to handle the contribution away from the light cone, in view of the next lemma. To handle the contribution near the boundary, one decomposes 
\begin{align*}
\veps(\tau, R) &= \veps(\tau, \nu\tau) + \partial_R\veps(\tau, \nu\tau)(R-\nu\tau)\\
& + \veps(\tau, R) - \veps(\tau, \nu\tau) - \partial_R\veps(\tau, \nu\tau)(R-\nu\tau)\\
\end{align*}
Insertion of the first term on the right for $\veps$ in the first term in \eqref{eq:linearineps} and again expanding $u^\nu - Q$ as in the preceding case results in a contribution to 
\begin{align*}
\sum_{1\leq k,l\leq N}\sum_{\pm}\tilde{a}_{kl}^{\pm}(\tau)\frac{e^{\pm i\nu\tau\xi^{\frac12}}}{\xi^{2+\frac{k\nu}{2}}}\langle\log\xi\rangle^l + \frac{\tilde{b}(\tau, \xi)}{\xi^{\frac52+\frac{\nu}{2}-}}.
\end{align*}
On the other hand, for the second term in the above formula for $\veps$ we use 
\[
\left\|\frac{\veps(\tau, R) - \veps(\tau, \nu\tau) - \partial_R\veps(\tau, \nu\tau)(R-\nu\tau)}{(R-\nu\tau)^2}\right\|_{L^\infty_{R\,dR}}\lesssim \big\|\veps\big\|_{H^{3+}_{R\,dR}}
\]
which is bounded, and argue again as in (i). 
\\

{\it{(iii) $\xb(\tau, \xi) = x_{good}(\tau, \xi)$}}. The corresponding contribution can be absorbed into $\tilde{x}_{good}(\tau, \xi)$. In fact, near $R = \nu\tau$, we decompose $\veps(\tau, R)$ as in (ii) and argue correspondingly. Near the origin $R = 0$, we take advantage of the fact that the function 
\[
\mathcal{D}\left(\frac{\cos(2u^\nu) - \cos\left(2Q(\lambda(t)r)\right)}{R^2}\right)
\]
is smooth. 
\\

This concludes the estimates for the first term in \eqref{eq:linearineps}, and the second term there is handled analogously. 
\\

{\it{(2): Contribution of the nonlinear terms in \eqref{eq:epsequation1}, as well as the linear term localised away from the light cone.}} Splitting the nonlinear terms into a contribution near the origin $R\ll \tau$ and away from the origin $R\gtrsim \tau$, we handle the latter exactly as in the preceding case (1) by using expansions such as in Lemma~\ref{lem:FourierToPhysical1} for $\epsilon$ in addition to the expansion of $u^{\nu}$ obtained from \cite{KST2,GaoKr}. It remains then to deal with the contribution near the origin for both the nonlinear and linear terms, where $u^{\nu}$ is now of class $C^\infty$. To handle this case, we then use 
\begin{lemma}\label{lem:nonlinestimates1} Let 
\[
\mathcal{D}\veps(R) = \int_0^\infty\phi(R, \xi)\xb(\tau, \xi)\tilde{\rho}(\xi)\,d\xi
\]
with $\xb\in S_0$, and put $\veps = \phi\big(\mathcal{D}\veps\big)(R)$. Then letting $N_{1}(\veps)$ be one of the nonlinear expressions in \eqref{eq:epsequation1}, we infer that 
\begin{align*}
\tilde{x}_1(\tau,\xi): = \xi^{-\frac{1}{2}}\int_{\tau_{0}}^{\tau}\sin\left(\lambda(\tau)\xi^{\frac{1}{2}}\int_{\sigma}^{\tau}\lambda(u)^{-1}du\right)\mathcal{F}\big(\mathcal{D}(\chi_{R\ll\tau}\lambda^{-2}N_1(\veps))\big)\left(\sigma,\frac{\lambda(\tau)^{2}}{\lambda(\sigma)^{2}}\xi\right)d\sigma,\,j = 1,2,
\end{align*}
satisfy 
\begin{align*}
&\sup_{\tau\geq \tau_0}\frac{\lambda(\tau)}{\lambda(\tau_0)}\left\langle\log\frac{\lambda(\tau)}{\lambda(\tau_0)}\right\rangle^{1+\frac{\kappa}{2}}\big\|\tilde{x}_1(\tau,\xi)\big\|_{S_0} + \sup_{\tau\geq \tau_0}\frac{\lambda(\tau)}{\lambda(\tau_0)}\left\langle\log\frac{\lambda(\tau)}{\lambda(\tau_0)}\right\rangle^{1+\frac{\kappa}{2}}\big\|\mathcal{D}_{\tau}\tilde{x}_1(\tau,\xi)\big\|_{S_1}\\
&\lesssim_{\tau_0} \left[\sup_{\tau\geq \tau_0}\frac{\lambda(\tau)}{\lambda(\tau_0)}\left\langle\log\frac{\lambda(\tau)}{\lambda(\tau_0)}\right\rangle^{1+\frac{\kappa}{2}}\big\|\langle\xi\rangle^{-\frac12}\xb(\tau, \cdot)\big\|_{S_0}\right]^2 +  \left[\sup_{\tau\geq \tau_0}\frac{\lambda(\tau)}{\lambda(\tau_0)}\left\langle\log\frac{\lambda(\tau)}{\lambda(\tau_0)}\right\rangle^{1+\frac{\kappa}{2}}\big\|\langle\xi\rangle^{-\frac12}\xb(\tau, \cdot)\big\|_{S_0}\right]^3.
\end{align*}
Furthermore, if $N_{2}(\veps)$ is the linear expression in \eqref{eq:epsequation1}, then  defining $\tilde{x}_2(\tau,\xi)$ analogously, we have 
\begin{align*}
&\sup_{\tau\geq \tau_0}\frac{\lambda(\tau)}{\lambda(\tau_0)}\left\langle\log\frac{\lambda(\tau)}{\lambda(\tau_0)}\right\rangle^{1+\frac{\kappa}{2}}\big\|\tilde{x}_2(\tau,\xi)\big\|_{S_0} + \sup_{\tau\geq \tau_0}\frac{\lambda(\tau)}{\lambda(\tau_0)}\left\langle\log\frac{\lambda(\tau)}{\lambda(\tau_0)}\right\rangle^{1+\frac{\kappa}{2}}\big\|\mathcal{D}_{\tau}\tilde{x}_2(\tau,\xi)\big\|_{S_1}\\
&\lesssim \tau_0^{-1}\sup_{\tau\geq \tau_0}\frac{\lambda(\tau)}{\lambda(\tau_0)}\left\langle\log\frac{\lambda(\tau)}{\lambda(\tau_0)}\right\rangle^{1+\frac{\kappa}{2}}\big\|\langle\xi\rangle^{-\frac12}\xb(\tau, \cdot)\big\|_{S_0}.
\end{align*}
\end{lemma}
The fact that the norm on the right hand side of the preceding inequalities is weaker than the norm on the left hand side is simply a consequence of the smoothing effect of the Duhamel parametrix. Note that the norm 
\[
\|\langle\xi\rangle^{-\frac12}\cdot\|_{S_0}
\]
is controlled for the terms $\chi_{\xi>1}\frac{b(\tau, \xi)}{\xi^{\frac52+\frac{\nu}{2}}}$. 

\begin{proof}
(lemma) We start with several preliminary estimates. Let 

\begin{align}\label{inhomo flow}
 \xt(\tau,\xi):=\xi^{-\frac{1}{2}}\int_{\tau_{0}}^{\tau}\sin\left(\lambda(\tau)\xi^{\frac{1}{2}}\int_{\sigma}^{\tau}\lambda(u)^{-1}du\right)\xb\left(\sigma,\frac{\lambda(\tau)^{2}}{\lambda(\sigma)^{2}}\xi\right)d\sigma.
\end{align}
To estimate $\xt(\tau,\xi)$ in terms of $\xb(\tau,\xi)$, we need to distinguish between large and small frequency. For \underline{large frequency}, we have

\begin{align}\label{nonlinear origin large xi pre}
 \begin{split}
  \|\langle\xi\rangle^{\frac{5}{2}+\kappa}\langle\log\xi\rangle^{-1-\kappa}\xt(\tau,\xi)\|_{L^{2}(\xi\geq\frac{1}{2})}\lesssim&\left\|\int_{\tau_{0}}^{\tau}\langle\xi\rangle^{2+\kappa}\langle\log\xi\rangle^{-1-\kappa}|\xb|\left(\sigma,\frac{\lambda(\tau)^{2}}{\lambda(\sigma)^{2}}\xi\right)d\sigma\right\|_{L^{2}(\xi\geq\frac{1}{2})}\\
  \lesssim&\int_{\tau_{0}}^{\tau}\left\|\langle\xi\rangle^{2+\kappa}\langle\log\xi\rangle^{-1-\kappa}\xb\left(\sigma,\frac{\lambda(\tau)^{2}}{\lambda(\sigma)^{2}}\xi\right)\right\|_{L^{2}(\xi\geq\frac{1}{2})}d\sigma\\
  \lesssim&\int_{\tau_{0}}^{\tau}\left(\frac{\lambda(\tau)}{\lambda(\sigma)}\right)^{-5}\|\langle\xi\rangle^{2+\kappa}\xb(\sigma,\xi)\|_{L^{2}(\xi\geq\frac{1}{2})}d\sigma.
 \end{split}
\end{align}
For \underline{small frequency}, we have

\begin{align}\label{nonlinear origin small xi pre}
 \begin{split}
  \|\xi^{\frac{1}{2}}\langle\log\xi\rangle^{-1-\kappa}\xt(\tau,\xi)\|_{L^{2}(\xi\leq\frac{1}{2})}\lesssim&\int_{\tau_{0}}^{\tau}\left\|\langle\log\xi\rangle^{-1-\kappa}\xb\left(\sigma,\frac{\lambda(\tau)^{2}}{\lambda(\sigma)^{2}}\xi\right)\right\|_{L^{2}(\xi\leq\frac{1}{2})}d\sigma\\
  \lesssim&\int_{\tau_{0}}^{\tau}\left(\frac{\lambda(\tau)}{\lambda(\sigma)}\right)^{-1}\left\langle\log\frac{\lambda(\tau)}{\lambda(\sigma)}\right\rangle^{-1-\kappa}\|\langle\xi\rangle^{0+}\xb(\sigma,\xi)\|_{L^{2}_{d\xi}}d\sigma\\
 \lesssim&\int_{\tau_{0}}^{\tau}\left(\frac{\lambda(\tau)}{\lambda(\sigma)}\right)^{-1}\left\langle\log\frac{\lambda(\tau)}{\lambda(\sigma)}\right\rangle^{-1-\kappa}\|\langle\log R\rangle\calF^{-1}(\xb)(\sigma,R)\|_{L^{2-}_{RdR}\cap L^{2}_{RdR}}d\sigma.
 \end{split}
\end{align}
Now we are ready to estimate the nonlinear terms. We start with the \underline{small frequency}. Among the three terms in $\calD(N(\veps))$, here we only give the details for the ``linear" contribution from

\begin{align*}
&\chi_{R\ll\tau}\mathcal{D}\left(\frac{\cos(2u^\nu) - \cos(2Q(\lambda(t)r)}{R^2}\veps\right)\\
  =&\chi_{R\ll\tau} \mathcal{D}\left(\frac{\cos(2u^\nu) - \cos(2Q(\lambda(t)r)}{R^2}\right) \veps + \chi_{R\ll\tau}\frac{\cos(2u^\nu) - \cos(2Q(\lambda(t)r)}{R^2}\mathcal{D}\veps
\end{align*}
The other two terms are handled similarly. According to \cite{KST2}, we have

\begin{align}\label{NL coe behav general}
\chi_{R\ll\tau}\frac{\cos(2u^\nu) - \cos(2Q(\lambda(t)r)}{R^2}\sim\tau^{-2}\chi_{R\ll\tau}\frac{\log(1+R^{2})}{R^2}.
\end{align}
It follows that we have

\begin{align}\label{NL coe esti general L2}
\begin{split}
&\left\|\langle\log R\rangle^2\chi_{R\ll\tau}\frac{\cos(2u^\nu) - \cos(2Q(\lambda(t)r)}{R^2}\right\|_{L^{2}_{RdR}\cap L^{2-}_{RdR}}\lesssim\tau^{-2},\\
&\left\|\chi_{R\ll\tau}R\langle\log\langle R\rangle^2\rangle\mathcal{D}\left(\frac{\cos(2u^\nu) - \cos(2Q(\lambda(t)r)}{R^2}\right)\right\|_{L^{2}_{RdR}\cap L^{2-}_{RdR}}\lesssim\tau^{-2}.
\end{split}
\end{align}
Let $\xt_{N}(\tau,\xi)$ be the inhomogeneous flow contributed by the nonlinear term in consideration. Then in view of \eqref{nonlinear origin small xi pre}, we have

\begin{align}\label{NL esti small xi}
\begin{split}
&\|\xi^{\frac{1}{2}}\langle\log\xi\rangle^{-1-\kappa}\xt_{N}(\tau,\xi)\|_{L^{2}(\xi\leq\frac{1}{2})}\\
\lesssim&\int_{\tau_{0}}^{\tau}\left(\frac{\lambda(\tau)}{\lambda(\sigma)}\right)^{-1}\left\langle\log\frac{\lambda(\tau)}{\lambda(\sigma)}\right\rangle^{-1-\kappa}\left\|\langle\log R\rangle^2\frac{\cos(2u^\nu) - \cos(2Q(\lambda(t)r)}{R^2}\right\|_{L^{2}_{RdR}\cap L^{2-}_{RdR}}\left\|\frac{\calD\veps(\sigma,\cdot)}{\langle \log R\rangle}\right\|_{L^{\infty}_{dR}}d\sigma\\
&+\int_{\tau_{0}}^{\tau}\left(\frac{\lambda(\tau)}{\lambda(\sigma)}\right)^{-1}\left\langle\log\frac{\lambda(\tau)}{\lambda(\sigma)}\right\rangle^{-1-\kappa}\left\|R\langle\log\langle R\rangle^2\rangle\calD\left(\frac{\cos(2u^\nu) - \cos(2Q(\lambda(t)r)}{R^2}\right)\right\|_{L^{2}_{RdR}\cap L^{2-}_{RdR}}\left\|\frac{\veps}{R\langle\log\langle R\rangle\rangle}(\sigma,\cdot)\right\|_{L^{\infty}_{dR}}d\sigma\\
\lesssim&\int_{\tau_{0}}^{\tau}\left(\frac{\lambda(\tau)}{\lambda(\sigma)}\right)^{-1}\left\langle\log\frac{\lambda(\tau)}{\lambda(\sigma)}\right\rangle^{-1-\kappa}\sigma^{-2}\|\langle\xi\rangle^{-\frac{1}{2}}\xb(\sigma,\xi)\|_{S_{0}}d\sigma\\
\lesssim&\int_{\tau_{0}}^{\tau}\sigma^{-2}d\sigma\left(\frac{\lambda(\tau)}{\lambda(\tau_{0})}\right)^{-1}\left\langle\log\frac{\lambda(\tau)}{\lambda(\tau_0)}\right\rangle^{-1-\frac{\kappa}{2}}\sup_{\tau\geq\tau_{0}}\left(\frac{\lambda(\tau)}{\lambda(\tau_{0})}\right)\left\langle\log\frac{\lambda(\tau)}{\lambda(\tau_{0})}\right\rangle^{1+\frac{\kappa}{2}}\|\langle\xi\rangle^{-\frac{1}{2}}\xb(\tau,\cdot)\|_{S_{0}}\\
\lesssim&\tau_{0}^{-1}\left(\frac{\lambda(\tau)}{\lambda(\tau_{0})}\right)^{-1}\left\langle\log\frac{\lambda(\tau)}{\lambda(\tau_0)}\right\rangle^{-1-\frac{\kappa}{2}}\sup_{\tau\geq\tau_{0}}\left(\frac{\lambda(\tau)}{\lambda(\tau_{0})}\right)\left\langle\log\frac{\lambda(\tau)}{\lambda(\tau_{0})}\right\rangle^{1+\frac{\kappa}{2}}\|\langle\xi\rangle^{-\frac{1}{2}}\xb(\tau,\cdot)\|_{S_{0}}
\end{split}
\end{align}
For \underline{large frequency}, we proceed in a different way. First, in addition to \eqref{NL coe esti general L2}, we also need the following $L^{\infty}$ estimates:

\begin{align}\label{NL coe esti general Linfty}
 \begin{split}
&\left\|\chi_{R\ll\tau}\frac{\cos(2u^\nu) - \cos(2Q(\lambda(t)r)}{R^2}\right\|_{L^{\infty}_{dR}}\lesssim\tau^{-2},\\
&\left\|\chi_{R\ll\tau}\mathcal{D}\left(\frac{\cos(2u^\nu) - \cos(2Q(\lambda(t)r)}{R^2}\right)\right\|_{L^{\infty}_{dR}}\lesssim\tau^{-2}.
\end{split}
\end{align}
We also need the following ``Plancherel-type'' estimate, whose proof is by a direct computation: Let $f(R)$ be a smooth function with sufficient decay as $R\rightarrow\infty$, then for a fixed $j\in\bbN$ and $\alpha\in\bbR^{+}$, we have

\begin{align}\label{plancherel esti}
 \|\xi^{1+\alpha}\calF(f)(\xi)\|_{L^{2}(\xi\sim 2^{j})}\lesssim 2^{j\alpha}\|f(\cdot)\|_{L^{2}_{RdR}}.
\end{align}
We consider an inhomogeneous dyadic decomposition in frequency space and we localize $\xi\sim 2^{j}$ for $j\in\bbN$. Let $\eta_{S}$ be localized as $\eta_{S}\sim 2^{k}$ for $k\in\bbN, k\leq j$ and $\eta_{L}$ be localized as $\eta_{L}\sim 2^{l}$ for $l\in\bbN, l\geq j$, because when $j,k,l\sim 1$, the dyadic decomposition can be easily summed. In fact here we assume that $j,k,l\gg 1$. In view of the smoothness for $\chi_{R\ll\tau}\frac{\cos(2u^\nu) - \cos(2Q(\lambda(t)r)}{R^2}$, we have

\begin{align}\label{NL coe localized}
 \begin{split}
&  \left\|P_{\eta_{L}}\left(\chi_{R\ll\tau}\frac{\cos(2u^\nu) - \cos(2Q(\lambda(t)r)}{R^2}\right)\right\|_{L^{2}_{RdR}}\lesssim\tau^{-2}\eta_{L}^{-N-1},\\
&\left\|P_{\eta_{L}}\left(\chi_{R\ll\tau}R\langle\log\langle R\rangle\rangle\calD\left(\frac{\cos(2u^\nu) - \cos(2Q(\lambda(t)r)}{R^2}\right)\right)\right\|_{L^{2}_{RdR}}\lesssim\tau^{-2}\eta_{L}^{-N}
 \end{split}
\end{align}
for arbitrarily large $N\in\bbN$. Moreover, for $\eta=\eta_{S}, \eta_{L}$, we have

\begin{align}\label{Dveps L2 localized}
\begin{split}
& \|\chi_{R\ll\tau}P_{\eta}\calD\veps\|_{L^{2}_{RdR}}\lesssim \eta^{-1-\kappa}\|\langle\cdot\rangle^{-\frac{1}{2}}\calF(P_{\eta}\calD\veps)(\tau,\cdot)\|_{S_{0}},\\
\Rightarrow\quad &\|\chi_{R\ll\tau}\tcalL P_{\eta_{S}}\calD\veps\|_{L^{2}_{RdR}}\lesssim \eta^{-\kappa}_{S}\|\langle\cdot\rangle^{-\frac{1}{2}}\calF(P_{\eta_{S}}\calD\veps)(\tau,\cdot)\|_{S_{0}},\\
&\|\chi_{R\ll\tau}(\tcalL)^{2}P_{\eta_{S}}\calD\veps\|_{L^{2}_{RdR}}\lesssim \eta^{1-\kappa}_{S}\|\calF(P_{\eta_{S}}\calD\veps)(\tau,\cdot)\|_{S_{0}}.
 \end{split}
\end{align}
For $P_{\eta}\veps:=\phi(P_{\eta}\calD\veps)$ we note the large frequency bounds 

\begin{align}\label{veps L2 localized}
 \begin{split}
& \|\chi_{R\ll\tau}P_{\eta}\veps\|_{L^{\infty}_{RdR}}\lesssim \eta^{-\frac{1}{2}-\kappa}\|\langle\cdot\rangle^{-\frac{1}{2}}\xb(\tau,\cdot)\|_{S_{0}},\\
 &\|\chi_{R\ll\tau}\tcalL P_{\eta_{S}}\veps\|_{L^{2}_{RdR}}\lesssim \eta^{-\frac{1}{2}-\kappa}_{S}\|\langle\cdot\rangle^{-\frac{1}{2}}\xb(\tau,\cdot)\|_{S_{0}},\\
&\|\chi_{R\ll\tau}(\tcalL)^{2}P_{\eta_{S}}\veps\|_{L^{2}_{RdR}}\lesssim \eta^{\frac{1}{2}-\kappa}_{S}\|\langle\cdot\rangle^{-\frac{1}{2}}\xb(\tau,\cdot)\|_{S_{0}}.
 \end{split}
\end{align}
For notational simplicity, we denote $\calC(R):=\frac{\cos(2u^\nu) - \cos(2Q(\lambda(t)r)}{R^2}$. Localizing $\xi\sim 2^{j}$, we start with the estimate for 

\begin{align}\label{NL esti large xi 1}
 \begin{split}
 & \|\xi^{2+\kappa}\calF\left(\chi_{R\ll\tau}\calC(R)\calD\veps\right)\|_{L^{2}(\xi\gtrsim 1)}\\
 \lesssim&\left\|\xi^{2+\kappa}\left\langle\phi(R,\xi),\chi_{R\ll\tau}\calC(R)\calD\veps(R)\right\rangle_{RdR}\right\|_{L^{2}(\xi\gtrsim 1)}\\
 \lesssim&\left\|\xi^{2+\kappa}\left\langle\phi(R,\xi),\chi_{R\ll\tau}\calC(R)P_{\eta_{L}}\calD\veps(R)\right\rangle_{RdR}\right\|_{L^{2}(\xi\gtrsim 1)}\\
 &+\left\|\xi^{2+\kappa}\left\langle\phi(R,\xi),\chi_{R\ll\tau}\calC(R)P_{\eta_{S}}\calD\veps(R)\right\rangle_{RdR}\right\|_{L^{2}(\xi\gtrsim 1)}
 \end{split}
\end{align}
For the first term on the right hand side of \eqref{NL esti large xi 1}, we bound it as

\begin{align}\label{NL esti large xi 1a}
 \begin{split}
  &\xi^{1+\kappa}\|\chi_{R\ll\tau}\calC(\sigma,\cdot)\|_{L^{\infty}_{dR}}\|P_{\eta_{L}}\calD\veps(\sigma,\cdot)\|_{L^{2}_{RdR}}\\
  \lesssim&\sigma^{-2}\xi^{1+\kappa}\eta_{L}^{-1-\kappa}\|\langle\cdot\rangle^{-\frac{1}{2}}\calF(P_{\eta_{L}}\calD\veps)(\sigma,\cdot)\|_{S_{0}}.
 \end{split}
\end{align}
For the second term in \eqref{NL esti large xi 1}, we decompose it as

\begin{align}\label{NL esti large xi 1b pre}
\begin{split}
 &\left\|\xi^{2+\kappa}\left\langle\phi(R,\xi),\chi_{R\ll\tau}P_{\eta_{L}}\calC(R)P_{\eta_{S}}\calD\veps(R)\right\rangle_{RdR}\right\|_{L^{2}(\xi\gtrsim 1)}\\
 &+\left\|\xi^{2+\kappa}\left\langle\phi(R,\xi),\chi_{R\ll\tau}P_{\eta_{S}}\calC(R)P_{\eta_{S}}\calD\veps(R)\right\rangle_{RdR}\right\|_{L^{2}(\xi\gtrsim 1)}
 \end{split}
\end{align}
For the first term in \eqref{NL esti large xi 1b pre}, we bound it as

\begin{align}\label{NL esti large 1b A}
 \begin{split}
  &\xi^{1+\kappa}\|\chi_{R\ll\tau}P_{\eta_{L}}\calC(\sigma,\cdot)\|_{L^{2}_{RdR}}\|P_{\eta_{S}}\calD\veps(\sigma,\cdot)\|_{L^{\infty}_{dR}}\\
  \lesssim&\sigma^{-2}\xi^{1+\kappa}\eta_{L}^{-1-N}\|\langle\cdot\rangle^{-\frac{1}{2}}\xb(\sigma,\cdot)\|_{S_{0}}
 \end{split}
\end{align}
For the second term in \eqref{NL esti large xi 1b pre}, we use integration by parts:

\begin{align}\label{NL esti large 1b B}
 \begin{split}
  &\left\|\xi^{2+\kappa}\left\langle\phi(R,\xi),\chi_{R\ll\tau}P_{\eta_{S}}\calC(R)P_{\eta_{S}}\calD\veps(R)\right\rangle_{RdR}\right\|_{L^{2}(\xi\gtrsim 1)}\\
  =&\left\|\xi^{1+\kappa}\left\langle\tilde{\calL}\phi(R,\xi),\chi_{R\ll\tau}P_{\eta_{S}}\calC(R)P_{\eta_{S}}\calD\veps(R)\right\rangle_{RdR}\right\|_{L^{2}(\xi\gtrsim 1)}\\
  =&\left\|\xi^{1+\kappa}\left\langle\phi(R,\xi),\tilde{\calL}\left(\chi_{R\ll\tau}P_{\eta_{S}}\calC(R)P_{\eta_{S}}\calD\veps(R)\right)\right\rangle_{RdR}\right\|_{L^{2}(\xi\gtrsim 1)}\\
  =&\left\|\xi^{\kappa}\left\langle\phi(R,\xi),\tilde{\calL}^{2}\left(\chi_{R\ll\tau}P_{\eta_{S}}\calC(R)P_{\eta_{S}}\calD\veps(R)\right)\right\rangle_{RdR}\right\|_{L^{2}(\xi\gtrsim 1)}\\
  \lesssim&\sigma^{-2}\xi^{-1+\kappa}\eta_{S}^{1-\kappa}\|\calF(P_{\eta_{S}}\calD\veps)(\sigma,\cdot)\|_{S_{0}}.
 \end{split}
\end{align}
Here we bound the contribution from $C(\sigma,\cdot)$ in $L^{\infty}$ and the contribution from $\calD\veps$ in $L^{2}_{RdR}$. Finally we simply sum the dyadic decomposition by taking the $\ell^{2}$-norm on both sides of the localized estimates and using Cauchy-Schwarz. Plugging the obtained estimates into \eqref{nonlinear origin large xi pre}, we get the desired result. The contribution from the term $\chi_{R\ll\tau}\calD\calC(R)\veps(R)$ are bounded similarly, by exploiting the bound on $\big\|\veps\big\|_{L^\infty}$ as well as further integrations by parts. 
%
%
This finishes the proof for the lemma.
\end{proof}
The proposition follows from the lemma.
\end{proof}

\section{Key linear estimates}
\subsection{The transference operator}
In this section we show that admissibility of the Fourier coefficients is preserved when applying the Duhamel parametrix to the non-local linear terms in \eqref{eq on Fourier side temp 7}. In fact, it will be important that this combination of operations entails a smoothing effect. 

\begin{proposition}\label{prop:nonlocallinear1} Let $F$ be either $2\frac{\lambda'}{\lambda}\mathcal{K}_0\mathcal{D}_{\tau}\xb$ or $\frac{\lambda'}{\lambda}[\mathcal{D}_{\tau}, \mathcal{K}_0]\xb$. Then if $\xb$ is admissible, so is 
\[
\tilde{x}(\tau,\xi): = \xi^{-\frac{1}{2}}\int_{\tau_{0}}^{\tau}\sin\left(\lambda(\tau)\xi^{\frac{1}{2}}\int_{\sigma}^{\tau}\lambda(u)^{-1}du\right)F\left(\sigma,\frac{\lambda(\tau)^{2}}{\lambda(\sigma)^{2}}\xi\right)d\sigma
\]
and the same bounds as in Proposition~\ref{prop:mulitlin1} hold but with $\ll$ replaced by $\lesssim $ involving a universal constant. 
\\
On the other hand, if $F$ is one of the terms $\left(\frac{\lambda'}{\lambda}\right)'\mathcal{K}_0\xb$, $\left(\frac{\lambda'}{\lambda}\right)^2\mathcal{K}_0\xb$, $\left(\frac{\lambda'}{\lambda}\right)^2\mathcal{K}_0^2\xb$, we get the corresponding bound as in Proposition~\ref{prop:mulitlin1} with a smallness gain (which can be made arbitrarily small by choosing $\tau_0$ large enough). 

\end{proposition}
\begin{proof} The argument is essentially the same for all of these terms, and so we give the details for the most delicate case, the expression $F = 2\frac{\lambda'}{\lambda}\mathcal{K}_0\mathcal{D}_{\tau}\xb$. We will treat various cases depending on which of the three expressions in \eqref{eq:admissiblexb} represents $\xb$. We observe that the fact that one gains extra smallness for the terms with a factor $\left(\frac{\lambda'}{\lambda}\right)'$, $\left(\frac{\lambda'}{\lambda}\right)^2$ comes from their faster decay $\sim\tau^{-2}$, whence choosing $\tau_0$ large enough suffices to gains smallness for these terms. 
\\

{\it{(1): $\xb =  \chi_{\xi>1}\sum_{N\geq k\geq 1,N\geq j\geq 0}\sum_{\pm}a_{kj}^{\pm}(\tau)\frac{e^{\pm i\nu\tau\xi^{\frac12}}}{\xi^{2+\frac{k\nu}{2}}}(\log\langle \xi\rangle)^j $}} Here we show that $\tilde{x}$ will be of the form 
\[
\tilde{x}(\tau, \xi) = \chi_{\xi>1}\frac{\tilde{b}(\tau, \xi)}{\xi^{\frac52 + \frac{\nu}{2}-}} + \tilde{x}_{good}(\tau, \xi).
\]
To see this, we distinguish between three different interactions inside the integral constituting $\mathcal{K}_0\mathcal{D}_{\tau}\xb$. Schematically expand out the expression for $\tilde{x}(\tau, \xi)$ in this case as 
\begin{align*}
\tilde{x}(\tau, \xi) = \int_{\tau_0}^{\tau}\frac{\sin\left[\lambda(\tau)\xi^{\frac12}\int_{\sigma}^{\tau}\lambda^{-1}(u)\,du\right]}{\xi^{\frac12}}\sigma^{-1}\int_0^\infty \frac{F\left(\frac{\lambda^2(\tau)}{\lambda^2(\sigma)}\xi, \eta\right)\tilde{\rho}(\eta)}{\frac{\lambda^2(\tau)}{\lambda^2(\sigma)}\xi - \eta}\chi_{\eta>1}a_{kj}^{\pm}(\sigma)\frac{e^{\pm i\nu\sigma\eta^{\frac12}}}{\eta^{\frac32 + \frac{k\nu}{2}}}\big(\log\langle\eta\rangle\big)^j\,d\eta d\sigma
\end{align*}

{\it{(1.i). $\frac{\lambda^2(\tau)}{\lambda^2(\sigma)}\xi\ll \eta$}}. Here we place $\tilde{x}$ into the portion $\tilde{x}_{good}(\tau, \xi)$. In fact, note that using Proposition~\ref{prop:Kstructure}, we can bound
\begin{align*}
\left|\chi_{\eta>1}\chi_{\frac{\lambda^2(\tau)}{\lambda^2(\sigma)}\xi\ll \eta}\frac{F\left(\frac{\lambda^2(\tau)}{\lambda^2(\sigma)}\xi, \eta\right)\tilde{\rho}(\eta)}{\frac{\lambda^2(\tau)}{\lambda^2(\sigma)}\xi - \eta}\right|\lesssim \eta^{-\frac54}\left(\frac{\lambda^2(\tau)}{\lambda^2(\sigma)}\xi\right)^{-\frac54},
\end{align*}
and so we have for $\xi>1$
\begin{align*}
&\left|\int_0^\infty \chi_{\frac{\lambda^2(\tau)}{\lambda^2(\sigma)}\xi\ll \eta}\frac{F\left(\frac{\lambda^2(\tau)}{\lambda^2(\sigma)}\xi, \eta\right)\tilde{\rho}(\eta)}{\frac{\lambda^2(\tau)}{\lambda^2(\sigma)}\xi - \eta}\chi_{\eta>1}a_{kj}^{\pm}(\sigma)\frac{e^{\pm i\nu\sigma\eta^{\frac12}}}{\eta^{\frac32 + \frac{k\nu}{2}}}\left(\log\langle\eta\rangle\right)^j\,d\eta\right|\\
&\lesssim \big|a_{kj}^{\pm}(\sigma)\big|\left(\frac{\lambda^2(\tau)}{\lambda^2(\sigma)}\xi\right)^{-3-\frac{k\nu}{2}+}
\end{align*}
Then calling, by abuse of notation calling the corresponding contribution again $\tilde{x}(\tau, \xi)$, we easily infer the bound 
\begin{align*}
\sup_{\tau\geq \tau_0}\frac{\lambda(\tau)}{\lambda(\tau_0)}\left\langle\log\frac{\lambda(\tau)}{\lambda(\tau_0)}\right\rangle^{1+\frac{\kappa}{2}}\big\|\xi^{\frac52 + \frac{k\nu}{2}-}\tilde{x}(\tau, \xi)\big\|_{L^2_{d\xi}(\xi>1)}\lesssim \sup_{\sigma\geq \tau_0}\frac{\lambda(\sigma)}{\lambda(\tau_0)}\left\langle\log\frac{\lambda(\sigma)}{\lambda(\tau_0)}\right\rangle^{1+\frac{\kappa}{2}}\big|a_{kj}^{\pm}(\sigma)\big|.
\end{align*}
The remaining estimate 
\begin{align*}
\sup_{\tau\geq \tau_0}\frac{\lambda(\tau)}{\lambda(\tau_0)}\left\langle\log\frac{\lambda(\tau)}{\lambda(\tau_0)}\right\rangle^{1+\frac{\kappa}{2}}\big\|\xi^{\frac12}\left\langle\log\xi\right\rangle^{-1-\kappa}\tilde{x}(\tau, \xi)\big\|_{L^2_{d\xi}(\xi<1)}\lesssim \sup_{\sigma\geq \tau_0}\frac{\lambda(\sigma)}{\lambda(\tau_0)}\left\langle\log\frac{\lambda(\sigma)}{\lambda(\tau_0)}\right\rangle^{1+\frac{\kappa}{2}}\big|a_{kj}^{\pm}(\sigma)\big|.
\end{align*}
is also easy to see, where we refer to Proposition~\ref{prop:K_0Sbound} for the definition of the norm $\big\|\cdot\big\|_{S_0}$ which has to be used to control the small frequency term. 
\\

{\it{(1.ii). $\frac{\lambda^2(\tau)}{\lambda^2(\sigma)}\xi\sim\eta$}}. Here the corresponding kernel bounds are much weaker, and so we shall place this contribution into the somewhat weaker term $\chi_{\xi>1}\frac{\tilde{b}(\tau, \xi)}{\xi^{\frac52 + \frac{\nu}{2}-}}$ in the large frequency regime. For this we have to bound for $\xi>1$ the expression
\begin{align*}
&\left|\int_{\tau_0}^{\tau}\frac{\sin[\lambda(\tau)\xi^{\frac12}\int_{\sigma}^{\tau}\lambda^{-1}(u)\,du]|}{\xi^{\frac12}}\sigma^{-1}\int_0^\infty \chi_{\frac{\lambda^2(\tau)}{\lambda^2(\sigma)}\xi\sim\eta}\frac{\xi^{\frac52+\frac{\nu}{2}-}F\left(\frac{\lambda^2(\tau)}{\lambda^2(\sigma)}\xi, \eta\right)\tilde{\rho}(\eta)}{\frac{\lambda^2(\tau)}{\lambda^2(\sigma)}\xi - \eta}\chi_{\eta>1}a_{kj}^{\pm}(\sigma)\frac{e^{\pm i\nu\sigma\eta^{\frac12}}}{\eta^{\frac32 + \frac{k\nu}{2}}}\big(\log\langle\eta\rangle\big)^j\,d\eta d\sigma\right|\\
& =: \big|\tilde{b}(\tau, \xi)\big|,
\end{align*}
where the $\eta$-integral is understood in the principal value sense. In fact, using Proposition~\ref{prop:Kstructure}, we get 
\begin{align*}
&\left|\chi_{\frac{\lambda^2(\tau)}{\lambda^2(\sigma)}\xi\sim\eta}\xi^{\frac52+\frac{\nu}{2}-}F\left(\frac{\lambda^2(\tau)}{\lambda^2(\sigma)}\xi, \eta\right)\tilde{\rho}(\eta)\chi_{\eta>1}a_{kj}^{\pm}(\sigma)\frac{e^{\pm i\nu\sigma\eta^{\frac12}}}{\eta^{\frac32 + \frac{k\nu}{2}}}\big(\log\langle\eta\rangle\big)^j\right|\\
&\lesssim \left|a_{kj}^{\pm}(\sigma)\right|\xi^{\frac52+\frac{\nu}{2}-}\left(\frac{\lambda^2(\tau)}{\lambda^2(\sigma)}\xi\right)^{-2-\frac{k\nu}{2}-}
\end{align*}
and from here we easily infer
\begin{align*}
\sup_{\tau\geq \tau_0}\frac{\lambda(\tau)}{\lambda(\tau_0)}\left\langle\log\frac{\lambda(\tau)}{\lambda(\tau_0)}\right\rangle^{1+\frac{\kappa}{2}}\big\|\tilde{b}(\tau, \xi)\big\|_{L^\infty_{d\xi}(\xi>1)}\lesssim \sup_{\sigma\geq \tau_0}\frac{\lambda(\sigma)}{\lambda(\tau_0)}\left\langle\log\frac{\lambda(\sigma)}{\lambda(\tau_0)}\right\rangle^{1+\frac{\kappa}{2}}\big|a_{kj}^{\pm}(\sigma)\big|
\end{align*}
To get the representation $\tilde{b}(\tau, \xi) = \partial_{\tau}\tilde{c}(\tau, \xi) + \tilde{d}(\tau, \xi)$ with 
\[
\big|\tilde{c}(\tau, \xi)\big| + \big|\tilde{d}(\tau, \xi)\big|\lesssim \xi^{-\frac12}\frac{\lambda(\tau_0)}{\lambda(\tau)}\left\langle\log\frac{\lambda(\tau)}{\lambda(\tau_0)}\right\rangle^{-1-\frac{\kappa}{2}},
\]
write the product of the oscillating factors $\sin[\ldots]$, $e^{\pm i\nu\sigma\eta^{\frac12}}$ after passing to the variable 
\[
\tilde{\eta}: = \frac{\lambda^2(\sigma)}{\lambda^2(\tau)}\eta,
\]
as a linear combination of oscillating terms of the form 
\[
e^{\pm i\nu\tau\xi^{\frac12}\pm i\nu\sigma\frac{\lambda(\tau)}{\lambda(\sigma)}\left(\xi^{\frac12} + \tilde{\eta}^{\frac12}\right)},\,e^{\pm i\nu\tau\xi^{\frac12}\pm i\nu\sigma\frac{\lambda(\tau)}{\lambda(\sigma)}\left(\xi^{\frac12} - \tilde{\eta}^{\frac12}\right)}.
\]
For the first type of phase function one writes it as 
\[
\frac{\partial_{\tau}\left(e^{\pm i\nu\tau\xi^{\frac12}\pm i\nu\sigma\frac{\lambda(\tau)}{\lambda(\sigma)}(\xi^{\frac12} + \tilde{\eta}^{\frac12})}\right)}{\pm i\nu\xi^{\frac12}\pm  i\nu\sigma\frac{\lambda_{\tau}(\tau)}{\lambda(\sigma)}\left(\xi^{\frac12} + \tilde{\eta}^{\frac12}\right)}
\]
and performs integration by parts with respect to $\tau$ to get the desired $c(\tau, \xi)$, noting that we restrict to $\xi\sim \tilde{\eta}$.
\\
For the second type of phase function, write it as 
\[
e^{\pm i\nu\tau\xi^{\frac12}\pm i\nu\sigma\frac{\lambda(\tau)}{\lambda(\sigma)}\left(\xi^{\frac12} - \tilde{\eta}^{\frac12}\right)} = \frac{\partial_{\tau}e^{\pm i\nu\tau\xi^{\frac12}}}{\pm i\nu\xi^{\frac12}}\cdot e^{\pm i\nu\sigma\frac{\lambda(\tau)}{\lambda(\sigma)}\left(\xi^{\frac12} - \tilde{\eta}^{\frac12}\right)},
\]
and again performs integration by parts with respect to $\tau$. Observe that any potential factors $ i\nu\sigma\frac{\lambda_{\tau}(\tau)}{\lambda(\sigma)}\left(\xi^{\frac12} - \tilde{\eta}^{\frac12}\right)$ arising that way are handled by using the improved off-diagonal decay of the kernel $F(\xi, \eta)$. 
\\

{\it{(1.iii). $\frac{\lambda^2(\tau)}{\lambda^2(\sigma)}\xi\gg\eta$}}. This is handled like case (1.i). 
\\

{\it{(2): $\xb(\tau, \xi) = \chi_{\xi>1}\frac{\tilde{b}(\tau,\xi)}{\xi^{\frac52 + \frac{\nu}{2}-}}$.}} Here we claim that the corresponding contribution to the Duhamel parametrix applied to $F = 2\frac{\lambda'}{\lambda}\mathcal{K}_0\mathcal{D}_{\tau}\xb$ can be placed into $\tilde{x}_{good}(\tau, \xi)$. 
This is done by splitting into the cases {\it{(2.i) - (2.iii)}} just as in (1). We deal here with the most delicate case {\it{(2.ii)}}, i. e. the case when $\frac{\lambda^2(\tau)}{\lambda^2(\sigma)}\xi\sim \eta$. Fully spelled out, this is the expression 
\begin{align*}
\int_{\tau_0}^{\tau}\frac{\sin\left[\lambda(\tau)\xi^{\frac12}\int_{\sigma}^{\tau}\lambda^{-1}(u)\,du\right]}{\xi^{\frac12}}\sigma^{-1}\int_0^\infty \chi_{\frac{\lambda^2(\tau)}{\lambda^2(\sigma)}\xi\sim\eta}\frac{F\left(\frac{\lambda^2(\tau)}{\lambda^2(\sigma)}\xi, \eta\right)\tilde{\rho}(\eta)}{\frac{\lambda^2(\tau)}{\lambda^2(\sigma)}\xi - \eta}\chi_{\eta>1}\frac{\mathcal{D}_{\sigma}b(\sigma, \eta)}{\eta^{\frac52 + \frac{\nu}{2}-}}\,d\eta d\sigma
\end{align*}
In order to ensure that this can be incorporated into $\tilde{x}_{good}$ for large frequencies $\xi>1$, we need to check that 
\begin{align*}
\left\|\int_{\tau_0}^{\tau}\xi^{2+\kappa}\sin\left[\lambda(\tau)\xi^{\frac12}\int_{\sigma}^{\tau}\lambda^{-1}(u)\,du\right]\sigma^{-1}\int_0^\infty \chi_{\frac{\lambda^2(\tau)}{\lambda^2(\sigma)}\xi\sim\eta}\frac{F\left(\frac{\lambda^2(\tau)}{\lambda^2(\sigma)}\xi, \eta\right)\tilde{\rho}(\eta)}{\frac{\lambda^2(\tau)}{\lambda^2(\sigma)}\xi - \eta}\chi_{\eta>1}\frac{\mathcal{D}_{\sigma}b(\sigma, \eta)}{\eta^{\frac52 + \frac{\nu}{2}-}}\,d\eta d\sigma\right\|_{L^2_{d\xi}(\xi>1)}
\end{align*}
is bounded for sufficiently small $\kappa>0$. Observe that 
\[
\left|\chi_{\frac{\lambda^2(\tau)}{\lambda^2(\sigma)}\xi\sim\eta>1}F\left(\frac{\lambda^2(\tau)}{\lambda^2(\sigma)}\xi, \eta\right)\tilde{\rho}(\eta)\right|\lesssim \eta^{-\frac12}, 
\]
and then, arguing as in \cite{KST2} to infer the desired $L^2$-bound for the Hilbert type operator above, we bound the preceding $L^2$-norm by 
\begin{align*}
\lesssim \int_{\tau_0}^{\tau}\sigma^{-1}\left(\frac{\lambda^2(\tau)}{\lambda^2(\sigma)}\right)^{-2-\kappa}\left\|\eta^{-\frac12-\frac{\nu}{2}+\kappa+}\right\|_{L^2_{d\eta}(\eta>1)}\left\|\eta^{-\frac12}\mathcal{D}_{\sigma}b(\sigma, \eta)\right\|_{L^\infty_{d\eta}}\,d\sigma
\end{align*}
It follows that the above $L^2_{d\xi}$-norm, weighted by $\frac{\lambda(\tau)}{\lambda(\tau_0)}\left\langle\log\frac{\lambda(\tau)}{\lambda(\tau_0)}\right\rangle^{1+\frac{\kappa}{2}}$, is bounded by 
\[
\lesssim \sup_{\sigma\geq \tau_0}\frac{\lambda(\sigma)}{\lambda(\tau_0)}\left\langle\log\frac{\lambda(\sigma)}{\lambda(\tau_0)}\right\rangle^{1+\frac{\kappa}{2}}\left\|\eta^{-\frac12}\mathcal{D}_{\sigma}b(\sigma, \eta)\right\|_{L^\infty_{d\eta}},
\]
{\it{(3): $\xb(\tau, \xi) = x_{good}(\tau, \xi)$.}} In this case the argument recovering the sharp time decay is a little more delicate. We again split into different frequency interactions in terms of $\xi$ and $\eta$. 
\\

{\it{(3.a): $\xi<1$}}.
\\

{\it{(3.a.i). $\frac{\lambda^2(\tau)}{\lambda^2(\sigma)}\xi\sim \eta$. }}  Introduce the new variable $\tilde{\eta} = \frac{\lambda^2(\sigma)}{\lambda^2(\tau)}\eta$. Then we need to bound the expression
\begin{align*}
\left\|\int_{\tau_0}^{\tau}\frac{\sin[\lambda(\tau)\xi^{\frac12}\int_{\sigma}^{\tau}\lambda^{-1}(u)\,du]}{\langle\log\xi\rangle^{1+\kappa}}\beta_{\nu}(\sigma)\int_0^\infty\chi_{\xi\sim\tilde{\eta}}\frac{F\left(\frac{\lambda^2(\tau)}{\lambda^2(\sigma)}\xi, \frac{\lambda^2(\tau)}{\lambda^2(\sigma)}\tilde{\eta}\right)}{\xi - \tilde{\eta}}\tilde{\rho}\left(\frac{\lambda^2(\tau)}{\lambda^2(\sigma)}\tilde{\eta}\right)\mathcal{D}_{\sigma}x\left(\sigma, \frac{\lambda^2(\tau)}{\lambda^2(\sigma)}\tilde{\eta}\right)\,d\tilde{\eta} d\sigma\right\|_{L^2_{d\xi}(\xi<1)}
\end{align*}
 Note that further restricting to $\frac{\lambda^2(\tau)}{\lambda^2(\sigma)}\tilde{\eta}<1$, we have 
 \[
 \chi_{\xi\sim\tilde{\eta}}\langle\log\xi\rangle^{-1-\kappa}\lesssim \left\langle \log\frac{\lambda(\tau)}{\lambda(\sigma)}\right\rangle^{-1-\kappa}
 \]
 Thus, under this further restriction, the preceding $L^2_{d\xi}$-norm is bounded by 
 \begin{align*}
&\lesssim\int_{\tau_0}^{\tau}\frac{\lambda(\sigma)}{\lambda(\tau)}\left\langle \log\frac{\lambda(\tau)}{\lambda(\sigma)}\right\rangle^{-1-\kappa}\sigma^{-1}\frac{\lambda(\tau_0)}{\lambda(\sigma)}\left\langle\log\frac{\lambda(\sigma)}{\lambda(\tau_0)}\right\rangle^{-1-\frac{\kappa}{2}}\,d\sigma\\&\hspace{4cm}\cdot\sup_{\sigma\geq \tau_0}
\frac{\lambda(\sigma)}{\lambda(\tau_0)}\left\langle\log\frac{\lambda(\sigma)}{\lambda(\tau_0)}\right\rangle^{1+\frac{\kappa}{2}}\big\|\langle\xi\rangle^{2+\kappa}\langle\log\xi\rangle^{-1-\kappa}\mathcal{D}_{\sigma}x(\sigma, \cdot)\big\|_{L^2_{d\xi}}
 \end{align*}
 Dividing into the cases $\lambda(\sigma)><\sqrt{\lambda(\tau)\lambda(\tau_0)}$, one easily infers that the above $\sigma$-integral is bounded by 
\[
\lesssim \frac{\lambda(\tau_0)}{\lambda(\tau)}\left\langle\log\frac{\lambda(\tau)}{\lambda(\tau_0)}\right\rangle^{-1-\frac{\kappa}{2}},
\]
thus giving the required bound. Assuming on the other hand that $\frac{\lambda^2(\tau)}{\lambda^2(\sigma)}\tilde{\eta}>1$ (while maintaining the other localizations), we can absorb the weight $\tilde{\rho}\left(\frac{\lambda^2(\tau)}{\lambda^2(\sigma)}\tilde{\eta}\right)$ into $\left\|\int_{0}^{\infty}...\mathcal{D}_{\sigma}x\left(\sigma, \frac{\lambda^2(\tau)}{\lambda^2(\sigma)}\tilde{\eta}\right)\,d\tilde{\eta} d\sigma\right\|_{L^2_{d\xi}(\xi<1)}
$ while still gaining a decay $\left\langle\frac{\lambda^2(\tau)}{\lambda^2(\sigma)}\tilde{\eta}\right\rangle^{-\frac54}$ from $F(\cdot, \cdot)$, which then results in a gain 
\[
 \chi_{\xi\sim\tilde{\eta}}\langle\log\xi\rangle^{-1-\kappa}\left\langle\frac{\lambda^2(\tau)}{\lambda^2(\sigma)}\tilde{\eta}\right\rangle^{-\frac54}\lesssim \left\langle \log\frac{\lambda(\tau)}{\lambda(\sigma)}\right\rangle^{-1-\kappa},
 \]
 and the desired bound follows from this as before.
\\

{\it{(3.a.ii). $\frac{\lambda^2(\tau)}{\lambda^2(\sigma)}\xi\ll \gg \eta$. }} These cases are handled analogously to the preceding case taking into account an extra gain in $\min\left\{\frac{\xi}{\eta}, \frac{\eta}{\xi}\right\}$.
\\

{\it{(3.b): $\xi>1$}}. Here we have to bound the expression 
\begin{align*}
\left\|\int_{\tau_0}^{\tau}\frac{\xi^{2+\kappa}}{\langle\log\xi\rangle^{1+\kappa}}\sin\left[\lambda(\tau)\xi^{\frac12}\int_{\sigma}^{\tau}\lambda^{-1}(u)\,du\right]\beta_{\nu}(\sigma)\int_0^\infty\frac{F\left(\frac{\lambda^2(\tau)}{\lambda^2(\sigma)}\xi, \frac{\lambda^2(\tau)}{\lambda^2(\sigma)}\tilde{\eta}\right)}{\xi - \tilde{\eta}}\tilde{\rho}\left(\frac{\lambda^2(\tau)}{\lambda^2(\sigma)}\tilde{\eta}\right)\mathcal{D}_{\sigma}x\left(\sigma, \frac{\lambda^2(\tau)}{\lambda^2(\sigma)}\tilde{\eta}\right)\,d\tilde{\eta} d\sigma\right\|_{L^2_{d\xi}(\xi>1)}
\end{align*}
As before we split into various frequency interactions: 
\\

{\it{(3.b.i): $\xi>1$, $\xi\sim \tilde{\eta}$}}. Here use that 
\begin{align*}
&\left\|\frac{\xi^{2+\kappa}}{\langle\log\xi\rangle^{1+\kappa}}\int_0^\infty\frac{F\left(\frac{\lambda^2(\tau)}{\lambda^2(\sigma)}\xi, \frac{\lambda^2(\tau)}{\lambda^2(\sigma)}\tilde{\eta}\right)}{\xi - \tilde{\eta}}\tilde{\rho}\left(\frac{\lambda^2(\tau)}{\lambda^2(\sigma)}\tilde{\eta}\right)\mathcal{D}_{\sigma}x\left(\sigma, \frac{\lambda^2(\tau)}{\lambda^2(\sigma)}\tilde{\eta}\right)\,d\tilde{\eta}\right\|_{L^2_{d\xi}(\xi>1)}\\
&\lesssim \left(\frac{\lambda^2(\tau)}{\lambda^2(\sigma)}\right)^{-3-\kappa+}\left\|\frac{(\cdot)^{2+\kappa}}{\langle\log(\cdot)\rangle^{1+\kappa}}\mathcal{D}_{\sigma}x(\sigma, \cdot)\right\|_{L^2},
\end{align*}
and so the above long integral expression can be bounded by 
\begin{align*}
&\lesssim \int_{\tau_0}^{\tau}\beta_{\nu}(\sigma)\left(\frac{\lambda^2(\tau)}{\lambda^2(\sigma)}\right)^{-3-\kappa+}\frac{\frac{\lambda(\tau_0)}{\lambda(\sigma)}}{\left\langle\log\frac{\lambda(\sigma)}{\lambda(\tau_0)}\right\rangle^{1+\frac{\kappa}{2}}}\,d\sigma\cdot\left(\sup_{\sigma\geq\tau_0}\frac{\lambda(\sigma)}{\lambda(\tau_0)}\left\langle\log\frac{\lambda(\sigma)}{\lambda(\tau_0)}\right\rangle^{1+\frac{\kappa}{2}}\left\|\frac{(\cdot)^{2+\kappa}}{\langle\log(\cdot)\rangle^{1+\kappa}}\mathcal{D}_{\sigma}x(\sigma, \cdot)\right\|_{L^2}\right)\\
&\lesssim\frac{\lambda(\tau_0)}{\lambda(\tau)}\left\langle\log\frac{\lambda(\tau)}{\lambda(\tau_0)}\right\rangle^{-1-\frac{\kappa}{2}}\cdot\left(\sup_{\sigma\geq\tau_0}\frac{\lambda(\sigma)}{\lambda(\tau_0)}\left\langle\log\frac{\lambda(\sigma)}{\lambda(\tau_0)}\right\rangle^{1+\frac{\kappa}{2}}\left\|\frac{(\cdot)^{2+\kappa}}{\langle\log(\cdot)\rangle^{1+\kappa}}\mathcal{D}_{\sigma}x(\sigma, \cdot)\right\|_{L^2}\right).\\
\end{align*}
This is the required bound for this contribution. 
\\

The remaining cases $\xi\ll\tilde{\eta}$, $\xi\gg\tilde{\eta}$ are handled similarly.

\end{proof}

In analogy to Proposition~\ref{prop:K_0Sbound}, we can force smallness in {\it{all}} of the bounds in the preceding proposition, provided we include further localizers to extreme frequencies, either small or large: 
\begin{proposition}\label{prop:nonlocallinear2} Let $F$ be either $2\frac{\lambda'}{\lambda}\mathcal{K}_0\mathcal{D}_{\tau}\xb$ or $\frac{\lambda'}{\lambda}[\mathcal{D}_{\tau}, \mathcal{K}_0]\xb$. Then if $\xb$ is admissible, the expressions
\begin{align*}
&\xi^{-\frac{1}{2}}\chi_{\xi<\epsilon}\int_{\tau_{0}}^{\tau}\sin\left(\lambda(\tau)\xi^{\frac{1}{2}}\int_{\sigma}^{\tau}\lambda(u)^{-1}du\right)F\left(\sigma,\frac{\lambda(\tau)^{2}}{\lambda(\sigma)^{2}}\xi\right)d\sigma\\
&\xi^{-\frac{1}{2}}\chi_{\xi>\epsilon^{-1}}\int_{\tau_{0}}^{\tau}\sin\left(\lambda(\tau)\xi^{\frac{1}{2}}\int_{\sigma}^{\tau}\lambda(u)^{-1}du\right)F\left(\sigma,\frac{\lambda(\tau)^{2}}{\lambda(\sigma)^{2}}\xi\right)d\sigma\\
\end{align*}
obey the same bounds as in Proposition~\ref{prop:mulitlin1} with a smallness constant of the form $\langle\log\epsilon\rangle^{-\gamma(\kappa)}$ for some $\gamma(\kappa)>0$, and $\kappa$ as in the definition of $S_0$. 
\\
Similarly, if we replace $F$ by $2\frac{\lambda'}{\lambda}\chi_{\xi>\epsilon^{-1}}\mathcal{K}_0\mathcal{D}_{\tau}\xb$ or by $2\frac{\lambda'}{\lambda}\chi_{\xi<\epsilon}\mathcal{K}_0\mathcal{D}_{\tau}\xb$. 
\end{proposition}
\subsection{Other linear contributions}
Here we consider the contribution from $\calF\left(\calR(\veps,\calD\veps)\right)$ on the right hand side of \eqref{eq on Fourier side temp 7}. $\calR(\veps,\calD\veps)$ is given in \eqref{eq Dveps temp 2}:
\begin{align*}
 \calR(\veps,\calD\veps)=&-\frac{4R}{(R^{2}+1)^{2}}\left(2\left(\frac{\lambda'}{\lambda}\right)^{2}+\left(\frac{\lambda'}{\lambda}\right)'\right)\veps+\left(\frac{\lambda'}{\lambda}\right)^{2}R\partial_{R}\left(\frac{4R}{(R^{2}+1)^{2}}\right)\veps\\
 &-2\frac{\lambda'}{\lambda}\left(\partial_{\tau}+\frac{\lambda'}{\lambda}R\partial_{R}\right)\left(\frac{4R}{(R^{2}+1)^{2}}\veps\right)
\end{align*}
The first line on the right hand side above can be treated in the same way as for the linear contribution from $\calF\left(\lambda^{-2}\calD(N(\veps))\right)$. We focus on the second line right hand side above. For notational simplicity, we denote
\begin{align}\label{notation for calR}
 U(R):=-\frac{4R}{(R^{2}+1)^{2}},\quad \calS(\veps,\calD\veps):=2\frac{\lambda'}{\lambda}\left(\partial_{\tau}+\frac{\lambda'}{\lambda}R\partial_{R}\right)\left(U(R)\veps\right).
\end{align}
The Fourier transform of $\calS(\veps,\calD\veps)$ is given by:
\begin{align}\label{Fourier of calS 1}
 \begin{split}
  \calF\left(\calS(\veps,\calD\veps)\right)(\tau,\xi)=&2\frac{\lambda'}{\lambda}\left(\partial_{\tau}-2\frac{\lambda'}{\lambda}\xi\partial_{\xi}\right)\calF\left(U(R)\veps\right)+2\frac{\lambda'}{\lambda}\calK\calF\left(U(R)\veps\right)\\
  =&2\frac{\lambda'}{\lambda}\left(\partial_{\tau}-2\frac{\lambda'}{\lambda}\xi\partial_{\xi}\right)\calF\left(U(R)\veps\right)+2\left(\frac{\lambda'}{\lambda}\right)^2\calK_{0}\calF\left(U(R)\veps\right)-4\left(\frac{\lambda'}{\lambda}\right)^2\calF\left(U(R)\veps\right).
 \end{split}
\end{align}
Again, the last two terms on the right hand side in \eqref{Fourier of calS 1} is handled in the same way as the linear contribution in $\calF\left(\lambda^{-2}\calD(N(\veps))\right)$. For the other term we have the following:
\begin{proposition}\label{prop:nonlocallinear R} Let $F$ be $2\frac{\lambda'}{\lambda}\left(\partial_{\tau}-2\frac{\lambda'}{\lambda}\xi\partial_{\xi}\right)\calF(U(R)\veps)$. Then if $\xb$ is admissible, so is 
\[
\tilde{x}(\tau,\xi): = \xi^{-\frac{1}{2}}\int_{\tau_{0}}^{\tau}\sin\left(\lambda(\tau)\xi^{\frac{1}{2}}\int_{\sigma}^{\tau}\lambda(u)^{-1}du\right)F\left(\sigma,\frac{\lambda(\tau)^{2}}{\lambda(\sigma)^{2}}\xi\right)d\sigma
\]
and the same bounds as in Proposition~\ref{prop:mulitlin1} hold but with $\ll$ replaced by $\lesssim $ involving a universal constant. 
\end{proposition}
\begin{proof}
We start by writing $\calF(U(R)\veps)$ as a linear operator applied on $\xb=\calF(\calD\veps)$, assuming in the large frequency regime that $\veps(\tau,R)=\int_{0}^{\infty}\xb(\tau,\eta)R^{-\frac{1}{2}}\phi_{KST}(R,\eta)\rho(\eta)d\eta$. In fact, note that 
\[
\int_{0}^{\infty}\xb(\tau,\eta)R^{-\frac{1}{2}}\phi_{KST}(R,\eta)\rho(\eta)d\eta = c\phi_0(R)\int_{0}^{\infty}\xb(\tau,\eta)\rho(\eta)d\eta + O(R^2),
\]
and inserting the first term on the right for $\veps$ results in a contribution to $\calF(\calD\veps)$ with a kernel of the form $O(\langle\xi\rangle^{-N})\cdot\rho(\eta)$, which gives a bounded operator from $S_1$ to $S_0$.

\begin{align}\label{Fourier of U veps}
 \begin{split}
  \calF(U(R)\veps)(\tau,\xi)=&\int_{0}^{\infty}U(R)\veps(\tau,R)\phi(R,\xi)RdR\\
  =&\int_{0}^{\infty}U(R)\left(\int_{0}^{\infty}\xb(\tau,\eta)R^{-\frac{1}{2}}\phi_{KST}(R,\eta)\rho(\eta)d\eta\right)\phi(R,\xi)RdR\\
  =&\int_{0}^{\infty}\left\langle U(R)R^{-\frac{1}{2}}\eta^{-1}\phi_{KST}(R,\eta),\phi(R,\xi)\right\rangle_{RdR}\xb(\tau,\eta)\trho(\eta)d\eta\\
  =:&\int_{0}^{\infty}\calJ(\xi,\eta)\xb(\tau,\eta)\trho(\eta)d\eta=:(\calJ\xb)(\tau,\xi).
 \end{split}
\end{align}
So the first term on the right hand side of \eqref{Fourier of calS 1} can be written as, modulo the terms which can be treated in the same way as for $\calF\left(\lambda^{-2}\calD(N(\veps))\right)$,
\begin{align}\label{Fourier of calS 2}
 2\frac{\lambda'}{\lambda}\calJ(\calD_{\tau}\xb)+2\frac{\lambda'}{\lambda}[\calD_{\tau},\calJ]\xb.
\end{align}
So the problem reduces to showing that the kernel of the operator $\calJ$ behaves as well as the kernel of $\calK_{0}$. When \underline{$\eta\gtrsim1$}, we start by computing
\begin{align}\label{def kernel G}
\begin{split}
 \xi\calJ(\xi,\eta)=&\left\langle U(R)R^{-\frac{1}{2}}\eta^{-1}\phi_{KST}(R,\eta),\tcalL\phi(R,\xi)\right\rangle_{RdR}\\
 =&\left\langle[\tcalL, U(R)]R^{-\frac{1}{2}}\eta^{-1}\phi_{KST}(R,\eta),\phi(R,\xi)\right\rangle_{RdR}\\
 &+\left\langle U(R)(\tcalL-\calL)\left(\eta^{-1}R^{-\frac{1}{2}}\phi_{KST}(R,\eta)\right),\phi(R,\xi)\right\rangle_{RdR}\\
 &+\left\langle U(R)\calL\left(\eta^{-1}R^{-\frac{1}{2}}\phi_{KST}(R,\eta)\right),\phi(R,\xi)\right\rangle_{RdR}\\
 =&\left\langle[\tcalL, U(R)]\left(R^{-\frac{1}{2}}\eta^{-1}\phi_{KST}(R,\eta)\right),\phi(R,\xi)\right\rangle_{RdR}\\
 &+\left\langle U(R)(\tcalL-\calL)\left(\eta^{-1}R^{-\frac{1}{2}}\phi_{KST}(R,\eta)\right),\phi(R,\xi)\right\rangle_{RdR}\\
 &+\eta\calJ(\xi,\eta)\\
 =:&G(\xi,\eta)+\eta\calJ(\xi,\eta).
 \end{split}
\end{align}
So we need to prove that $G(\xi,\eta)$ behaves as good as $F(\xi,\eta)$ at least. To compute $G(\xi,\eta)$, we list the formulas for $\calD, \calD^{*}$ and $\calL,\tcalL$:
\begin{align}\label{expression various operators}
\begin{split}
 &\calD=\partial_{R}+\frac{1}{R}-\frac{2}{R(R^{2}+1)},\quad \calD^{*}=-\partial_{R}-\frac{1}{R}-\frac{1-R^{2}}{R(1+R^{2})},\\
 &\calL=\calD^{*}\calD=-\partial^{2}_{R}-\frac{1}{R}\partial_{R}+\frac{1-6R^{2}+R^{4}}{R^{2}(1+R^{2})^{2}},\\
 &\tcalL=\calD\calD^{*}=-\partial_{R}^{2}-\frac{1}{R}\partial_{R}+\frac{4}{R^{2}(1+R^{2})}.\\
 \Rightarrow\quad &\tcalL-\calL=\frac{-R^{4}+10R^{2}+3}{(R^{2}+1)^{2}R^{2}}.
\end{split}
 \end{align}
For the commutator $[\tcalL,U(R)]$, we have, in view of \eqref{F off diagonal pre},
\begin{align}\label{commutator tcalL U}
 [\tcalL,U(R)]=-2U_{R}\partial_{R}-U_{RR}-\frac{1}{R}U_{R}.
\end{align}
We observe that
\begin{align*}
 [\tcalL,U(R)]+\tcalL-\calL=&-2U_{R}\partial_{R}-U_{RR}-\frac{1}{R}U_{R}+\frac{-R^{4}+10R^{2}+3}{R^{2}(R^{2}+1)^{2}}\\
 =&-2U_{R}\calD+\frac{1}{R}U_{R}-U_{RR}-\frac{4}{R(R^{2}+1)}U_{R}+\frac{-R^{4}+10R^{2}+3}{R^{2}(R^{2}+1)^{2}}\\
 =:&-2U_{R}\calD+V(R).
\end{align*}
So we have
\begin{align}\label{rewrite G}
 G(\xi,\eta)=\left\langle-2U_{R}(R)\phi(R,\eta),\phi(R,\xi)\right\rangle_{RdR}+\left\langle V(R)\eta^{-1}R^{-\frac{1}{2}}\phi_{KST}(R,\eta),\phi(R,\xi)\right\rangle_{RdR}.
\end{align}
For the first term on the right hand side of \eqref{rewrite G}, we note that $U_{R}(R)$ has the same asymptotic behavior as $W(R)$ when $R\rightarrow0$ or $R\rightarrow\infty$. So we only need to look at the second term. First $V(R)$ has the following asymptotic behavior:
\begin{align}\label{V behavior}
 V(R)\sim O(R^{-2}),\quad \textrm{when}\quad R\rightarrow0\quad \textrm{and}\quad R\rightarrow\infty.
\end{align}
According to \cite{KST2} and in view of $\eta\gtrsim1$,
\begin{align}\label{modified phi KST behavior}
\begin{split}
 &\eta^{-1}R^{-\frac{1}{2}}\phi_{KST}(R,\eta)\sim R\eta^{-1},\quad \textrm{when}\quad R\eta^{\frac{1}{2}}\ll1,\\
 &\eta^{-1}R^{-\frac{1}{2}}\phi_{KST}(R,\eta)\sim R^{-\frac{1}{2}}\eta^{-\frac{7}{4}},\quad \textrm{when}\quad R\eta^{\frac{1}{2}}\gtrsim1.
 \end{split}
\end{align}
If \underline{$\xi\leq\eta$}, we split the second integral in \eqref{rewrite G} as
\begin{align*}
\int_{0}^{\eta^{-\frac{1}{2}}}+\int_{\eta^{-\frac{1}{2}}}^{\xi^{-\frac{1}{2}}}+\int_{\xi^{-\frac{1}{2}}}^{\infty}:=I+II+III.
\end{align*}
In view of \eqref{phi behavior 1}-\eqref{phi behavior 2} and \eqref{modified phi KST behavior} we have
\begin{align}\label{G bound I}
 \begin{split}
  |I|\lesssim \eta^{-1}\int_{0}^{\eta^{-\frac{1}{2}}}R^{-2}RR^{2}RdR\lesssim \eta^{-\frac{5}{2}}.
 \end{split}
\end{align}
$II$ is bounded by (if $\xi\leq 1$)
\begin{align}\label{G bound IIa}
\begin{split}
 |II|\lesssim&\eta^{-\frac{7}{4}}\int_{\eta^{-\frac{1}{2}}}^{\xi^{-\frac{1}{2}}}
 =\int_{\eta^{-\frac{1}{2}}}^{1}...+\int_{1}^{\xi^{-\frac{1}{2}}}...\\
 \lesssim&\eta^{-\frac{7}{4}}\int_{\eta^{-\frac{1}{2}}}^{1}R^{-2}R^{-\frac{1}{2}}R^{2}RdR+\eta^{-\frac{7}{4}}\int_{1}^{\xi^{-\frac{1}{2}}}R^{-2}R^{-\frac{1}{2}}|\log R|RdR\lesssim \eta^{-\frac{7}{4}}.
 \end{split}
\end{align}
If $\xi\geq 1$, we have
\begin{align}\label{G bound IIb}
 \begin{split}
  |II|\lesssim &\eta^{-\frac{7}{4}}\int_{\eta^{-\frac{1}{2}}}^{\xi^{-\frac{1}{2}}}R^{-2}R^{-\frac{1}{2}}R^{2}RdR\lesssim \eta^{-\frac{7}{4}}\xi^{-\frac{3}{4}}.
 \end{split}
\end{align}
For $III$, if $\xi\geq 1$ we have
\begin{align}\label{G bound IIIa}
 \begin{split}
 |III|\lesssim&\eta^{-\frac{7}{4}}\xi^{-\frac{5}{4}}\int_{\xi^{-\frac{1}{2}}}^{1}R^{-2}R^{-\frac{1}{2}}R^{-\frac{1}{2}}RdR+\eta^{-\frac{7}{4}}\xi^{-\frac{5}{4}}\int_{1}^{\infty}R^{-2}R^{-\frac{1}{2}}R^{-\frac{1}{2}}RdR\\
 \lesssim&\xi^{-\frac{5}{4}}\eta^{-\frac{5}{4}}|\log\xi|.
 \end{split}
\end{align}
If $\xi\leq 1$, we have
\begin{align}\label{G bound IIIb}
 \begin{split}
  |III|\lesssim&\eta^{-\frac{7}{4}}|\log\xi|\int_{\xi^{-\frac{1}{2}}}^{\infty}R^{-2}R^{-\frac{1}{2}}RdR\lesssim\eta^{-\frac{7}{4}}\xi^{\frac{1}{4}}|\log\xi|\lesssim\eta^{-\frac{7}{4}}.
 \end{split}
\end{align}
If \underline{$\xi\geq\eta$}, we split the integral as
\begin{align*}
 \int_{0}^{\xi^{-\frac{1}{2}}}+\int_{\xi^{-\frac{1}{2}}}^{\eta^{-\frac{1}{2}}}+\int_{\eta^{-\frac{1}{2}}}^{1}+\int_{1}^{\infty}:=IV+V+VI+VII.
\end{align*}
We have
\begin{align}\label{G bound IV to VII}
 \begin{split}
  |IV|\lesssim&\eta^{-1}\int_{0}^{\xi^{-\frac{1}{2}}}R^{-2}RR^{2}RdR\lesssim\eta^{-1}\xi^{-\frac{3}{2}},\\
  |V|\lesssim&\eta^{-1}\xi^{-\frac{5}{4}}\int_{\xi^{-\frac{1}{2}}}^{\eta^{-\frac{1}{2}}}R^{-2}RR^{-\frac{1}{2}}RdR\lesssim\eta^{-\frac{5}{4}}\xi^{-\frac{5}{4}},\\
  |VI|\lesssim&\xi^{-\frac{5}{4}}\eta^{-\frac{3}{2}}\int_{\eta^{-\frac{1}{2}}}^{1}R^{-\frac{1}{2}}R^{-2}RdR\lesssim\xi^{-\frac{5}{4}}\eta^{-\frac{5}{4}},\\
  |VII|\lesssim&\eta^{-\frac{7}{4}}\xi^{-\frac{5}{4}}\int_{1}^{\infty}R^{-2}R^{-\frac{1}{2}}R^{-\frac{1}{2}}RdR\lesssim\eta^{-\frac{7}{4}}\xi^{-\frac{5}{4}}.
 \end{split}
\end{align}
Now we turn to the case when \underline{$\eta\ll1$} by separating the resonant part in $\veps$ by writing $\calS(\veps,\calD\veps)$ as
\begin{align}\label{calS resonant}
 \calS(\veps,\calD\veps)=2\frac{\lambda'}{\lambda}\left(\partial_{\tau}+\frac{\lambda'}{\lambda}R\partial_{R}\right)\left(U(R)\phi_{0}(R)\int_{0}^{R}[\phi_{0}(s)]^{-1}(\calD\veps)(s)ds+U(R)c(\tau)\phi_{0}(R)\right).
\end{align}
The contribution of the resonant part to $\calF(\calS(\veps,\calD\veps))$ is bounded in the same way as its contribution to $\calF\left(\lambda^{-2}\calD\left(N(\veps)\right)\right)$. For the non-resonant part we write the contribution as
\begin{align}\label{Fourier of U non resonant}
\begin{split}
 &\calF\left(U(R)\phi_{0}(R)\int_{0}^{R}[\phi_{0}(s)]^{-1}(\calD\veps)(s)ds\right)(\tau,\xi)\\
 =&\int_{0}^{\infty}U(R)\phi_{0}(R)\left(\int_{0}^{R}[\phi_{0}(s)]^{-1}(\calD\veps)(\tau,s)ds\right)\phi(R,\xi)RdR\\
  =&\int_{0}^{\infty}U(R)\phi_{0}(R)\left(\int_{0}^{R}[\phi_{0}(s)]^{-1}\int_{0}^{\infty}\xb(\tau,\eta)\phi(s,\eta)\trho(\eta)d\eta ds\right)\phi(R,\xi)RdR\\
  =&\int_{0}^{\infty}\left\langle U(R)\phi_{0}(R)\int_{0}^{R}[\phi_{0}(s)]^{-1}\phi(s,\eta)ds,\phi(R,\xi)\right\rangle_{RdR}\xb(\tau,\eta)\trho(\eta)d\eta\\
  =:&\int_{0}^{\infty}\tcalJ(\xi,\eta)\xb(\tau,\eta)\trho(\eta)d\eta=:(\tcalJ\xb)(\tau,\xi).
 \end{split}
\end{align}
To investigate the operator $\tcalJ$, we start by estimating $\int_{0}^{R}[\phi_{0}(s)]^{-1}\phi(s,\eta)ds$. When $R^{2}\eta\lesssim 1$, using \eqref{phi behavior 1} we have
\begin{align}\label{phi phi behavior 1}
 \begin{split}
  &\left|\int_{0}^{R}[\phi_{0}(s)]^{-1}\phi(s,\eta)ds\right|\lesssim R^{2},\quad\textrm{when}\quad R\lesssim 1,\\
  &\left|\int_{0}^{R}[\phi_{0}(s)]^{-1}\phi(s,\eta)ds\right|\lesssim R^{2}|\log R|,\quad\textrm{when}\quad R\gg1.
 \end{split}
\end{align}
When $R^{2}\eta\gtrsim 1$, using \eqref{phi behavior 2} we have
\begin{align}\label{phi phi behavior 2}
 \begin{split}
  &\left|\int_{0}^{R}[\phi_{0}(s)]^{-1}\phi(s,\eta)ds\right|\lesssim R^{2}|\log R|,\quad \textrm{when}\quad \eta\ll1.
 \end{split}
\end{align}
To estimate $\tcalJ$, we first consider the case $\xi\leq\eta$. We split the integral defining $\tcalJ$ as
\begin{align*}
 \int_{0}^{1}+\int_{1}^{\eta^{-\frac{1}{2}}}+\int_{\eta^{-\frac{1}{2}}}^{\xi^{-\frac{1}{2}}}+\int_{\xi^{-\frac{1}{2}}}^{\infty}=:I'+II'+III'+IV'.
\end{align*}
We use \eqref{phi behavior 1}-\eqref{phi behavior 2} and \eqref{phi phi behavior 1}-\eqref{phi phi behavior 2} to obtain
\begin{align}\label{tcalJ bound a}
 \begin{split}
  |I'|\lesssim&\int_{0}^{1}RRR^{2}R^{2}RdR\lesssim 1,\\
  |II'|+|III'|\lesssim&\int_{1}^{\xi^{-\frac{1}{2}}}R^{-4}R^{2}|\log R|^{2}RdR\lesssim|\log\xi|^{3},\\
  |IV'|\lesssim&\int_{\xi^{-\frac{1}{2}}}^{\infty}R^{-4}R^{2}|\log R|R^{-\frac{1}{2}}\xi^{-\frac{1}{4}}|\log\xi|RdR\lesssim|\log\xi|.
 \end{split}
\end{align}
If $\eta\ll1\leq \xi$, we split the integral as
\begin{align*}
 \int_{0}^{\xi^{-\frac{1}{2}}}+\int_{\xi^{-\frac{1}{2}}}^{\infty}=:V'+VI'.
\end{align*}
We have, again using \eqref{phi behavior 1}-\eqref{phi behavior 2} and \eqref{phi phi behavior 1}-\eqref{phi phi behavior 2},
\begin{align}\label{tcalJ bound b}
 \begin{split}
  |V'|\lesssim &\int_{0}^{\xi^{-\frac{1}{2}}}R^{2}R^{2}R^{2}RdR\lesssim\xi^{-4}.
 \end{split}
\end{align}
The estimate for $VI'$ is a little more delicate. We recall that for $R^{2}\xi\gtrsim1$ and $\xi\gtrsim1$,
\begin{align*}
 &\phi(R,\xi)=\calD\left(\xi^{-1}R^{-\frac{1}{2}}\phi_{KST}(R,\xi)\right),\quad \textrm{and}\quad \xi^{-1}R^{-\frac{1}{2}}\phi_{KST}(R,\xi)\sim\xi^{-\frac{7}{4}}e^{iR\xi^{\frac{1}{2}}}\sigma(R\xi^{\frac{1}{2}},R),\\
 \Rightarrow\quad&\phi(R,\xi)\sim R^{-\frac{1}{2}}\xi^{-\frac{5}{4}}e^{iR\xi^{\frac{1}{2}}}\sigma(R\xi^{\frac{1}{2}},R)+R^{-\frac{3}{2}}\xi^{-\frac{7}{4}}e^{iR\xi^{\frac{1}{2}}}\sigma(R\xi^{\frac{1}{2}},R).
\end{align*}
Here $\sigma(R\xi^{\frac{1}{2}},R)$ denotes a class of functions given in Proposition 5.6 in [KST]. Let us denote $\phi_{nr}(R,\eta):=\phi_{0}(R)\int_{0}^{R}[\phi_{0}(s)]^{-1}\phi(s,\eta)ds$. In view of the fact $e^{iR\xi^{\frac{1}{2}}}=-i\xi^{-\frac{1}{2}}\partial_{R}\left(e^{iR\xi^{\frac{1}{2}}}\right)$, we have
\begin{align}\label{tcalJ bound c1}
 \begin{split}
  VI'\sim&\int_{\xi^{-\frac{1}{2}}}^{\infty}U(R)\phi_{nr}(R,\eta)\xi^{-\frac{7}{4}}\partial_{R}\left(e^{iR\xi^{\frac{1}{2}}}\right)\left(R^{-\frac{1}{2}}
  \sigma(R\xi^{\frac{1}{2}},R)+R^{-\frac{3}{2}}\xi^{-\frac{1}{2}}\sigma(R\xi^{\frac{1}{2}},R)\right)RdR.
 \end{split}
\end{align}
Let $I(R)$ be the function in the integrand of \eqref{tcalJ bound c1} except the factor $\xi^{-\frac{7}{4}}\partial_{R}\left(e^{iR\xi^{\frac{1}{2}}}\right)$. Therefore the integration in \eqref{tcalJ bound c1} can be written as
\begin{align}\label{tcalJ bound c1 prime}
 VI'\sim\int_{\xi^{-\frac{1}{2}}}^{\infty}I(R)\xi^{-\frac{7}{4}}\partial_{R}\left(e^{iR\xi^{\frac{1}{2}}}\right)dR.
\end{align}
The function $I(R)$ has the following asymptotic behavior for $k=0,1,2,3,4$:
\begin{align}\label{IR behavior}
\begin{split}
 &R^{k}I^{(k)}(R)\sim R^{-\frac{3}{2}}|\log R|,\quad \textrm{when}\quad R\rightarrow\infty,\quad R^{k}I^{(k)}(R)\sim R^{\frac{9}{2}},\quad\textrm{when}\quad R\rightarrow\xi^{-\frac{1}{2}}.
 \end{split}
\end{align}
We perform integration by parts:
\begin{align}\label{tcalJ bound c2}
\begin{split}
 VI'\sim&-\int_{\xi^{-\frac{1}{2}}}^{\infty}\xi^{-\frac{7}{4}}I'(R)e^{iR\xi^{\frac{1}{2}}}dR+I(R)\xi^{-\frac{7}{4}}e^{iR\xi^{\frac{1}{2}}}|^{\infty}_{\xi^{-\frac{1}{2}}}\\
 \sim&\int_{\xi^{-\frac{1}{2}}}^{\infty}\xi^{-\frac{9}{4}}I''(R)e^{iR\xi^{\frac{1}{2}}}dR-I'(R)\xi^{-\frac{9}{4}}e^{iR\xi^{\frac{1}{2}}}|^{\infty}_{\xi^{-\frac{1}{2}}}+I(R)\xi^{-\frac{7}{4}}e^{iR\xi^{\frac{1}{2}}}|^{\infty}_{\xi^{-\frac{1}{2}}}\\
 \sim&-\int_{\xi^{-\frac{1}{2}}}^{\infty}\xi^{-\frac{11}{4}}I'''(R)e^{iR\xi^{\frac{1}{2}}}dR+I''(R)\xi^{-\frac{11}{4}}e^{iR\xi^{\frac{1}{2}}}|^{\infty}_{\xi^{-\frac{1}{2}}}\\
 &-I'(R)\xi^{-\frac{9}{4}}e^{iR\xi^{\frac{1}{2}}}|^{\infty}_{\xi^{-\frac{1}{2}}}+I(R)\xi^{-\frac{7}{4}}e^{iR\xi^{\frac{1}{2}}}|^{\infty}_{\xi^{-\frac{1}{2}}}\\
 \sim&\int_{\xi^{-\frac{1}{2}}}^{\infty}\xi^{-\frac{13}{4}}I^{(4)}(R)e^{iR\xi^{\frac{1}{2}}}dR-I'''(R)\xi^{-\frac{13}{4}}e^{iR\xi^{\frac{1}{2}}}|^{\infty}_{\xi^{-\frac{1}{2}}}+I''(R)\xi^{-\frac{11}{4}}e^{iR\xi^{\frac{1}{2}}}|^{\infty}_{\xi^{-\frac{1}{2}}}\\
 &-I'(R)\xi^{-\frac{9}{4}}e^{iR\xi^{\frac{1}{2}}}|^{\infty}_{\xi^{-\frac{1}{2}}}+I(R)\xi^{-\frac{7}{4}}e^{iR\xi^{\frac{1}{2}}}|^{\infty}_{\xi^{-\frac{1}{2}}}\\
 \Rightarrow\quad\quad|VI'|\lesssim&\xi^{-\frac{13}{4}}.
 \end{split}
\end{align}
If $\eta\leq\xi\leq1$, the estimate is similar to deriving \eqref{tcalJ bound a} and we simply interchange $\xi^{-\frac{1}{2}}$ and $\eta^{-\frac{1}{2}}$ and the bound in the second estimate in \eqref{tcalJ bound a} would be $|\log\eta|^{3}$.
\end{proof}

The preceding considerations imply that the operator $\calJ$ is completely analogous to the operator $\mathcal{K}_0$, and in the future, it suffices to deal with the latter. 

\section{Interactions between the resonant and non-resonant parts}

\subsection{The effect of $c(\tau)$ on $\mathcal{D}\veps$}
Keeping in mind the decomposition \eqref{eqn:epsdecomp}, we have so far analyzed the contributions of the `non-resonant part' given by $\phi(\mathcal{D}\veps)$ to the linear and nonlinear source terms in $N(\veps)$. On the other hand, the resonant part $c(\tau)\phi_0(R)$ clearly also contributes to these terms. Recalling the equation \eqref{ctau ODE} as well as the formula \eqref{eq:L_cParametrix}, we expect $c(\tau)$ to obey bounds with two powers less of decay than $\mathcal{D}\veps$, of the form 
\[
\big|c(\tau)\big| + \tau\big|c'(\tau)|\lesssim \tau^2\frac{\lambda(\tau_0)}{\lambda(\tau)}\left\langle\log\frac{\lambda(\tau)}{\lambda(\tau_0)}\right\rangle^{-1-\frac{\kappa}{2}}. 
\]
We now analyze the effect of substituting a term of the form $\veps = c(\tau)\phi_0(R)$ into $N(\veps)$. This is straightforward for the nonlinear terms: 
\begin{proposition}\label{prop:resonantnonlin} Let $N_1(\veps)$ be one of the two nonlinear terms in \eqref{eq:epsequation1}, and assume that $\veps(\tau, R) = c(\tau)\phi_0(R)$. Also, assume $\nu\ll1$. Then denoting by 
\[
\|c\|: = \sup_{\tau\geq \tau_0}\left[|c(\tau)| + \tau|c'(\tau)|\right]\tau^{-2}\frac{\lambda(\tau)}{\lambda(\tau_0)}\left\langle\log\frac{\lambda(\tau)}{\lambda(\tau_0)}\right\rangle^{1+\frac{\kappa}{2}},
\]
and letting 
\[
\tilde{x}(\tau, \xi): = \xi^{-\frac{1}{2}}\int_{\tau_{0}}^{\tau}\sin\left(\lambda(\tau)\xi^{\frac{1}{2}}\int_{\sigma}^{\tau}\lambda(u)^{-1}du\right)\mathcal{F}\left(\lambda^{-2}\mathcal{D}(N_1(\veps))\right)\left(\sigma,\frac{\lambda(\tau)^{2}}{\lambda(\sigma)^{2}}\xi\right)d\sigma
\]
then $\tilde{x}$ is admissible in that we have a decomposition 
\begin{align*}
\tilde{x}(\tau, \xi) = \sum_{N\geq k\geq 1,N\geq j\geq 0}\sum_{\pm}\chi_{\xi>1}\tilde{a}_{kj}^{\pm}(\tau)\frac{e^{\pm i\nu\tau\xi^{\frac12}}}{\xi^{2+\frac{k\nu}{2}}}(\log\langle \xi\rangle)^j + \chi_{\xi>1}\frac{\tilde{b}(\tau, \xi)}{\xi^{\frac52+\frac{\nu}{2}-}} + \tilde{x}_{good}(\tau, \xi)
\end{align*}
with $\tilde{b}$ admitting a representation as $\partial_{\tau}\tilde{c}+\tilde{d}$ and the bound
\begin{align*}
&\sup_{\tau\geq \tau_0}\frac{\lambda(\tau)}{\lambda(\tau_0)}\left\langle\log\frac{\lambda(\tau)}{\lambda(\tau_0)}\right\rangle^{1+\frac{\kappa}{2}}\big[\sum_{\pm}\tau\big|\tilde{a}_{kj}^{\pm}(\tau)\big| + \sum_{\pm}\tau\big|(\tilde{a}_{kj}^{\pm})'(\tau)\big| + \tau\big\|(|\tilde{b}| + |\xi^{\frac12}\tilde{c}| + \xi^{\frac12}|\tilde{d}|)(\tau, \xi)\big\|_{L_{\xi}^{\infty}}\\&\hspace{6.5cm} + \tau\big\|\langle\xi\rangle^{-\frac12}\mathcal{D}_{\tau}\tilde{b}(\tau, \xi)\big\|_{L_{\xi}^{\infty}} + \big\|\tilde{x}_{good}(\tau, \xi)\big\|_{S_0} + \big\|\mathcal{D}_{\tau}\tilde{x}_{good}(\tau, \xi)\big\|_{S_1}\big]\\
&\lesssim_{\tau_0}\|c\|^2 + \|c\|^3.
\end{align*}

\end{proposition}

The proof of this is a straightforward variation on that of Proposition~\ref{prop:mulitlin1}. Observe that the rapid decay of $c(\tau)$ (due to $\nu\ll1$) gives big gains in the nonlinear interactions and in particular compensates for the $\tau^2$-decay loss compared to the decay of $\mathcal{D}\veps$.
Things are naturally more delicate for the {\it{linear}} term in $N(\veps)$, i. e. the {\it{first term}} in \eqref{eq:epsequation1}. Here the leading order of decay of $u^{\nu}-Q(R)$ plays a crucial role in compensating for the weaker decay of $c(\tau)$ to recover the same decay rate as $\mathcal{D}\veps$. Note that here we have to go one level deeper in the iteration and also invoke fine structure of $c''(\tau)$: 
\begin{proposition}\label{prop:resonantlin} Let $N_2(\veps)$ be the linear term in \eqref{eq:epsequation1}, and assume that $\veps(\tau, R) = c(\tau)\phi_0(R)$. Also, assume that $c'(\tau)$ admits the representation \eqref{eq:c'explicit} with $A(\tau), B(\tau)$ satisfying 
\[
\sup_{\tau\geq\tau_0}\frac{\lambda(\tau)}{\lambda(\tau_0)}\left\langle\log\frac{\lambda(\tau)}{\lambda(\tau_0)}\right\rangle^{1+\frac{\kappa}{2}}\left[\tau |A(\tau)| + |B(\tau)|\right]\leq \alpha.
\]
Then letting 
\[
\tilde{x}(\tau, \xi): = \xi^{-\frac{1}{2}}\int_{\tau_{0}}^{\tau}\sin\left(\lambda(\tau)\xi^{\frac{1}{2}}\int_{\sigma}^{\tau}\lambda(u)^{-1}du\right)\mathcal{F}\left(\lambda^{-2}\mathcal{D}(N_2(\veps))\right)\left(\sigma,\frac{\lambda(\tau)^{2}}{\lambda(\sigma)^{2}}\xi\right)d\sigma,
\]
there is a decomposition
\begin{align*}
\tilde{x}(\tau, \xi) = \sum_{N\geq k\geq 1,N\geq j\geq 0}\sum_{\pm}\chi_{\xi>1}\tilde{a}_{kj}^{\pm}(\tau)\frac{e^{\pm i\nu\tau\xi^{\frac12}}}{\xi^{2+\frac{k\nu}{2}}}(\log\langle \xi\rangle)^j + \chi_{\xi>1}\frac{\tilde{b}(\tau, \xi)}{\xi^{\frac52+\frac{\nu}{2}-}} + \tilde{x}_{good}(\tau, \xi)
\end{align*}
where $\tilde{b}$ admits the usual decomposition in terms of $\tilde{c}, \tilde{d}$, and with the bound
\begin{align*}
&\sup_{\tau\geq \tau_0}\frac{\lambda(\tau)}{\lambda(\tau_0)}\left\langle\log\frac{\lambda(\tau)}{\lambda(\tau_0)}\right\rangle^{1+\frac{\kappa}{2}}\big[\sum_{\pm}\tau\big|\tilde{a}_{kj}^{\pm}(\tau)\big| + \sum_{\pm}\tau\big|(\tilde{a}_{kj}^{\pm})'(\tau)\big| + \tau\big\|(|\tilde{b}| + |\xi^{\frac12}\tilde{c}| + |\xi^{\frac12}\tilde{d}|)(\tau, \xi)\big\|_{L_{\xi}^{\infty}}\\&\hspace{6.5cm} + \tau\big\|\langle\xi\rangle^{-\frac12}\mathcal{D}_{\tau}\tilde{b}(\tau, \xi)\big\|_{L_{\xi}^{\infty}} + \big\|\tilde{x}_{good}(\tau, \xi)\big\|_{S_0} + \big\|\mathcal{D}_{\tau}\tilde{x}_{good}(\tau, \xi)\big\|_{S_1}\big]\\
&\lesssim \big|c(\tau_0)\big| + \alpha.
\end{align*}
In fact, we can replace the preceding by the more precise (with $D$ denoting a constant only depending on $\nu$)
\[
\leq D\big|c(\tau_0)\big| + \gamma(\tau_0)\alpha
\]
and $\lim_{\tau_0\rightarrow\infty}\gamma(\tau_0) = 0$, provided we omit the term $\big\|\tilde{x}_{good}(\tau, \xi)\big\|_{S_0}$ on the left.  
\end{proposition}
\begin{proof} The `singular' contributions which lead to the first two terms in $\tilde{x}(\tau, \xi)$ are straightforward to handle since they arise from the `boundary' of the cone where $R\sim\sigma$. Here we shall concentrate on the part away from the boundary, and specifically the case of small frequencies $\xi<1$, where the required time decay is more subtle. Thus let now 
\[
N_2(\veps) = \mathcal{D}\left[\chi_{R\ll\tau}\frac{\sin(u^{\nu} - Q(R))\sin(u^{\nu}+Q(R))}{R^2}\veps\right]
\]
We need to recover the exact decay rate for $\|\tilde{x}(\tau, \xi)\|_{S_0}$, for $\xi<1$. Letting $\chi(x)$ be a smooth cutoff localizing to $x\leq 1$, say, we split 
\[
\tilde{x}(\tau, \xi) = \tilde{x}_1(\tau, \xi) + \tilde{x}_2(\tau, \xi),
\]
where $\tilde{x}_1(\tau, \xi)$ is defined by inclusion of $\chi\left(C^{-1}\frac{\xi^{\frac12}}{\frac{\lambda(\sigma)}{\lambda(\tau)}\sigma^{-1}}\right)$ inside the integral giving $\tilde{x}(\tau, \xi)$; here $C$ is a large constant to be chosen later. Then using that (see [],[]) we have 
\begin{equation}\label{eq:uny-Q}
u^{\nu} - Q(R) = O\left(\frac{R}{\tau^2}\left\langle\log(R)\right\rangle\right).
\end{equation}
We first estimate the contribution of the small frequency term $\tilde{x}_1(\tau, \xi)$, where we take crucial advantage of the weight $\langle\log\xi\rangle^{-1-\kappa}$ in the definition of $\|\cdot\|_{S_0}$: 
\begin{align*}
&\big\|\tilde{x}_1(\tau, \xi)\big\|_{S_0}\\ &\sim\big\|\xi^{\frac12}\langle\log\xi\rangle^{-1-\kappa}\tilde{x}_1(\tau, \xi)\big\|_{L^2_{d\xi}}\\
&\lesssim_C \int_{\tau_0}^{\tau}\frac{\lambda(\sigma)}{\lambda(\tau)}\sigma^{-1}\left\langle\log\left(\frac{\lambda(\sigma)}{\lambda(\tau)}\sigma^{-1}\right)\right\rangle^{-1-\kappa}\left\|\mathcal{F}\left(\mathcal{D}\left[\chi_{R\ll\tau}\frac{\sin(u^{\nu} - Q(R))\sin(u^{\nu}+Q(R))}{R^2}\veps\right]\right)(\sigma, \cdot)\right\|_{L^\infty_{d\xi}}\,d\sigma
\end{align*}
Since with our current choice of $\veps$ and in light of \eqref{eq:uny-Q} we have the easily verified bound 
\begin{align*}
\left\|\mathcal{F}\left(\mathcal{D}\left[\chi_{R\ll\tau}\frac{\sin(u^{\nu} - Q(R))\sin(u^{\nu}+Q(R))}{R^2}\veps\right]\right)(\sigma, \cdot)\right\|_{L^\infty_{d\xi}}\lesssim \frac{\lambda(\tau_0)}{\lambda(\sigma)}\left\langle\log\frac{\lambda(\sigma)}{\lambda(\tau_0)}\right\rangle^{-1-\frac{\kappa}{2}}\|c\|, 
\end{align*}
we can bound the preceding integral expression by 
\begin{align*}
&\|c\|\int_{\tau_0}^{\tau}\frac{\lambda(\sigma)}{\lambda(\tau)}\sigma^{-1}\left\langle\log\left(\frac{\lambda(\sigma)}{\lambda(\tau)}\sigma^{-1}\right)\right\rangle^{-1-\kappa}\frac{\lambda(\tau_0)}{\lambda(\sigma)}\left\langle\log\frac{\lambda(\sigma)}{\lambda(\tau_0)}\right\rangle^{-1-\frac{\kappa}{2}}\,d\sigma\\
&\lesssim \|c\|\log^{-\frac{\kappa}{2}}(\tau_0)\frac{\lambda(\tau_0)}{\lambda(\tau)}\left\langle\log\frac{\lambda(\tau)}{\lambda(\tau_0)}\right\rangle^{-\frac{\kappa}{2}}. 
\end{align*}
This gives the required bound even with $\ll$ by picking $\tau_0$ very large, upon observing that $\|c\|\lesssim |c(\tau_0)| + \alpha$. \\
It then remains to consider the expression $\tilde{x}_2(\tau, \xi)$, where we shall rely on twofold integration by parts with respect to $\sigma$. 
Thus write 
\begin{align*}
\tilde{x}_2(\tau, \xi)&:= \xi^{-1}\frac{\lambda(\sigma)}{\lambda(\tau)}\cos\left(\lambda(\tau)\xi^{\frac{1}{2}}\int_{\sigma}^{\tau}\lambda(u)^{-1}du\right)\left[1-\chi\left(C^{-1}\frac{\xi^{\frac12}}{\frac{\lambda(\sigma)}{\lambda(\tau)}\sigma^{-1}}\right)\right]\mathcal{F}\big(\lambda^{-2}\mathcal{D}(N_2(\veps))\big)\left(\sigma,\frac{\lambda(\tau)^{2}}{\lambda(\sigma)^{2}}\xi\right)\big|_{\tau_0}^{\tau}\\
& - \int_{\tau_0}^{\tau}\frac{\cos\left(\lambda(\tau)\xi^{\frac{1}{2}}\int_{\sigma}^{\tau}\lambda(u)^{-1}du\right)}{\xi}\frac{\partial}{\partial\sigma}\left(\frac{\lambda(\sigma)}{\lambda(\tau)}\left[1-\chi\left(C^{-1}\frac{\xi^{\frac12}}{\frac{\lambda(\sigma)}{\lambda(\tau)}\sigma^{-1}}\right)\right]\mathcal{F}\big(\lambda^{-2}\mathcal{D}(N_2(\veps))\big)\right)\left(\sigma,\frac{\lambda(\tau)^{2}}{\lambda(\sigma)^{2}}\xi\right)\,d\sigma
\end{align*}
For the boundary terms, the one with $\sigma = \tau$ does not gain any smallness: we have 
\begin{align*}
\left\|\frac{\xi^{\frac12}}{\langle\log\xi\rangle^{1+\kappa}}\cdot \xi^{-1}[1-\chi(C^{-1}\tau\xi^{\frac12})]\mathcal{F}\big(\lambda^{-2}\mathcal{D}(N_2(\veps))\big)\left(\tau, \xi\right)\right\|_{L^2_{d\xi}(\xi<1)}\lesssim \|c\|\frac{\lambda(\tau_0)}{\lambda(\tau)}\left\langle\log\frac{\lambda(\tau)}{\lambda(\tau_0)}\right\rangle^{-1-\frac{\kappa}{2}}. 
\end{align*}
For the second boundary contribution evaluated at $\sigma = \tau_0$, we bound it by exploiting the rapid decay for large frequencies of the Fourier transform: 
\begin{align*}
\left|\mathcal{F}\big(\lambda^{-2}\mathcal{D}(N_2(\veps))\big)\left(\tau_0,\frac{\lambda(\tau)^{2}}{\lambda(\tau_0)^{2}}\xi\right)\right|\lesssim \big|c(\tau_0)\big|\left\langle \frac{\lambda(\tau)^{2}}{\lambda(\tau_0)^{2}}\xi\right\rangle^{-N}
\end{align*}
for arbitrary $N>1$. We conclude that the second boundary contribution is bounded by 
\begin{align*}
&\big|c(\tau_0)\big|\frac{\lambda(\tau_0)}{\lambda(\tau)}\left\|\frac{\xi^{\frac12}}{\langle\log\xi\rangle^{1+\kappa}}\cdot \xi^{-1}\left[1-\chi\left(C^{-1}\frac{\lambda(\tau)}{\lambda(\tau_0)}\tau_0\xi^{\frac12}\right)\right]\left\langle \frac{\lambda(\tau)^{2}}{\lambda(\tau_0)^{2}}\xi\right\rangle^{-N}\right\|_{L^2_{d\xi}(\xi<1)}\\&\lesssim_{\tau_0}\big|c(\tau_0)\big|\frac{\lambda(\tau_0)}{\lambda(\tau)}\left\langle\log\frac{\lambda(\tau)}{\lambda(\tau_0)}\right\rangle^{-1-\kappa},
\end{align*}
which gives the desired decay rate. 
\\
Now consider the integral expression in the above formula for $\tilde{x}_2(\tau, \xi)$. For fixed time $\sigma$, we bound the suitably weighted $L^2_{d\xi}(\xi<1)$-norm of the integrand by 
\begin{align*}
&\lesssim \sigma^{-1}\frac{\lambda(\sigma)}{\lambda(\tau)}\left\|\xi^{-\frac12}\langle\log\xi\rangle^{-1-\kappa}\left[1-\chi\left(C^{-1}\frac{\lambda(\tau)}{\lambda(\sigma)}\sigma\xi^{\frac12}\right)\right]\left\langle\frac{\lambda^2(\tau)}{\lambda^2(\sigma)}\xi\right\rangle^{-N}\right\|_{L^2_{d\xi}(\xi<1)}\cdot\|c\|\frac{\lambda(\tau_0)}{\lambda(\sigma)}\left\langle\log\frac{\lambda(\sigma)}{\lambda(\tau_0)}\right\rangle^{-1-\frac{\kappa}{2}}\\
&\sim \sigma^{-1}\frac{\lambda(\sigma)}{\lambda(\tau)}\log^{\frac12}(\sigma)\left\langle\log\frac{\lambda(\tau)}{\lambda(\sigma)}\sigma\right\rangle^{-1-\kappa}\cdot\|c\|\frac{\lambda(\tau_0)}{\lambda(\sigma)}\left\langle\log\frac{\lambda(\sigma)}{\lambda(\tau_0)}\right\rangle^{-1-\frac{\kappa}{2}}\\
\end{align*}
which falls short of the decay required to reproduce the rate $\frac{\lambda(\tau_0)}{\lambda(\tau)}\left\langle\log\frac{\lambda(\tau)}{\lambda(\tau_0)}\right\rangle^{-1-\frac{\kappa}{2}}$. However, we observe that if the operator $\frac{\partial}{\partial{\sigma}}$ in the above integral term constituting $\tilde{x}_2(\tau, \xi)$ falls on $c(\sigma)$ in $N_2(\veps)$ (recalling that we assume $\veps(\sigma, R) = c(\sigma)\phi_o(R)$) and, referring to \eqref{eq:c'explicit} below (which is a simple consequence of Lemma~\ref{lem:cintegralbound}), produces the term $A(\sigma)$, then we gain an extra power $\sigma^{-1}$ and easily produce the required bound, with a smallness gain (when choosing $\tau_0\gg 1$). Removing this particular term (but retaining all other terms) and callling the resulting expression $\frac{\partial}{\partial{\sigma}}\big(\ldots\big)'$, we then perform another integration by parts with respect to $\sigma$,  resulting in one boundary term at $\sigma = \tau_0$ which is treated exactly like the preceding one, in addition to the following integral term: 
\begin{align*}
&\int_{\tau_0}^{\tau}\frac{\sin\left(\lambda(\tau)\xi^{\frac{1}{2}}\int_{\sigma}^{\tau}\lambda(u)^{-1}du\right)}{\xi^{\frac32}}\\
\cdot&\frac{\partial}{\partial{\sigma}}\left[\frac{\lambda(\sigma)}{\lambda(\tau)}\frac{\partial}{\partial{\sigma}}\left(\frac{\lambda(\sigma)}{\lambda(\tau)}\left[1-\chi\left(C^{-1}\frac{\xi^{\frac12}}{\frac{\lambda(\sigma)}{\lambda(\tau)}\sigma^{-1}}\right)\right]\mathcal{F}\big(\mathcal{D}(N_2(\veps))\big)\left(\sigma,\frac{\lambda(\tau)^{2}}{\lambda(\sigma)^{2}}\xi\right)\right)'\right]\,d\sigma
\end{align*}
Here the weighted $L^2_{d\xi}(\xi<1)$-norm of the integrand for fixed time $\sigma$ is bounded by 
\begin{align*}
&\lesssim \left\|\xi^{-1}\langle\log\xi\rangle^{-1-\kappa}\left[1-\chi\left(C^{-1}\frac{\xi^{\frac12}}{\frac{\lambda(\sigma)}{\lambda(\tau)}\sigma^{-1}}\right)\right]\right\|_{L^2_{d\xi}}\frac{\lambda^2(\sigma)}{\lambda^2(\tau)}\sigma^{-2}\cdot \alpha\frac{\lambda(\tau_0)}{\lambda(\sigma)}\left\langle\log\frac{\lambda(\sigma)}{\lambda(\tau_0)}\right\rangle^{-1-\frac{\kappa}{2}}\\
&\lesssim \alpha C^{-1}\frac{\lambda(\tau_0)}{\lambda(\tau)}\sigma^{-1}\left\langle\log\frac{\lambda(\tau)}{\lambda(\sigma)}\right\rangle^{-1-\kappa}\left\langle\log\frac{\lambda(\sigma)}{\lambda(\tau_0)}\right\rangle^{-1-\frac{\kappa}{2}},\\
\end{align*}
and this now reproduces the required decay upon integration in $\sigma$, and even with a smallness gain when picking $C$ sufficiently large. 

We next repeat the same steps but for the slightly modified quantity $\mathcal{D}_{\tau}\tilde{x}$. Splitting as before into the contributions $\mathcal{D}_{\tau}\tilde{x}_{1,2}$, the small frequency term $\mathcal{D}_{\tau}\tilde{x}_1$ is handled exactly as before. As for the other term, we have to work more carefully to gain smallness. Integration by parts leads to 
\begin{align*}
\mathcal{D}_{\tau}\tilde{x}_2(\tau, \xi)&:= \xi^{-\frac12}\frac{\lambda(\sigma)}{\lambda(\tau)}\sin\left(\lambda(\tau)\xi^{\frac{1}{2}}\int_{\sigma}^{\tau}\lambda(u)^{-1}du\right)\left[1-\chi\left(C^{-1}\frac{\xi^{\frac12}}{\frac{\lambda(\sigma)}{\lambda(\tau)}\sigma^{-1}}\right)\right]\mathcal{F}\big(\mathcal{D}(N_2(\veps))\big)\left(\sigma,\frac{\lambda(\tau)^{2}}{\lambda(\sigma)^{2}}\xi\right)\big|_{\tau_0}^{\tau}\\
& - \int_{\tau_0}^{\tau}\frac{\sin\left(\lambda(\tau)\xi^{\frac{1}{2}}\int_{\sigma}^{\tau}\lambda(u)^{-1}du\right)}{\xi^{\frac12}}\frac{\partial}{\partial{\sigma}}\left(\frac{\lambda(\sigma)}{\lambda(\tau)}\left[1-\chi\left(C^{-1}\frac{\xi^{\frac12}}{\frac{\lambda(\sigma)}{\lambda(\tau)}\sigma^{-1}}\right)\right]\mathcal{F}\big(\mathcal{D}(N_2(\veps))\big)\left(\sigma,\frac{\lambda(\tau)^{2}}{\lambda(\sigma)^{2}}\xi\right)\right)\,d\sigma
\end{align*}
Only the boundary term for $\sigma = \tau_0$ survives (and this gets treated as before), while the remaining integral requires another integration by parts; more precisely, we remove the case when $\frac{\partial}{\partial{\sigma}}$ falls on $c(\sigma)$ inside $N_2(\veps)$ and thereby produces $A(\sigma)$, referring to \eqref{eq:c'explicit}, calling the remaining term $\frac{\partial}{\partial{\sigma}}\big(\ldots\big)'$. Then we apply another integration by parts with respect to $\sigma$, 
resulting in an additional boundary term at $\sigma = \tau$ of the form 
\begin{align*}
\xi^{-1}\frac{\partial}{\partial{\sigma}}\left(\frac{\lambda(\sigma)}{\lambda(\tau)}\left[1-\chi\left(C^{-1}\frac{\xi^{\frac12}}{\frac{\lambda(\sigma)}{\lambda(\tau)}\sigma^{-1}}\right)\right]\mathcal{F}\big(\mathcal{D}(N_2(\veps))\big)\left(\sigma,\frac{\lambda(\tau)^{2}}{\lambda(\sigma)^{2}}\xi\right)\right)'\big|_{\sigma = \tau}
\end{align*}
in addition to another boundary term at $\sigma=\tau_0$ being handled like the preceding one. The preceding boundary term is then estimated in the usual weighted $L^2_{d\xi}$-space (this time using the weight $\langle\log\xi\rangle^{-1-\kappa}$ in light of the definition of $\big\|\cdot\big\|_{S_1}$) by 
\begin{align*}
&\lesssim \big\|\xi^{-1}\langle\log\xi\rangle^{-1-\kappa}\tau^{-1}[1-\chi(C^{-1}\tau\xi^{\frac12})]\langle\xi\rangle^{-N}\big\|_{L^2_{d\xi}(\xi<1)}\cdot\alpha\frac{\lambda(\tau_0)}{\lambda(\tau)}\left\langle\log\frac{\lambda(\tau)}{\lambda(\tau_0)}\right\rangle^{-1-\kappa}\\
&\lesssim \alpha C^{-1}\frac{\lambda(\tau_0)}{\lambda(\tau)}\left\langle\log\frac{\lambda(\tau)}{\lambda(\tau_0)}\right\rangle^{-1-\kappa},
\end{align*}
which gives the desired smallness by choosing $C$ large enough (of course independently of $\tau_0$). 
Furthermore, we arrive at the integral expression
\begin{align*}
 &\int_{\tau_0}^{\tau}\frac{\cos\left(\lambda(\tau)\xi^{\frac{1}{2}}\int_{\sigma}^{\tau}\lambda(u)^{-1}du\right)}{\xi}\\
 \cdot&\frac{\partial}{\partial{\sigma}}\left(\frac{\lambda(\sigma)}{\lambda(\tau)}\frac{\partial}{\partial{\sigma}}\left(\frac{\lambda(\sigma)}{\lambda(\tau)}\left[1-\chi\left(C^{-1}\frac{\xi^{\frac12}}{\frac{\lambda(\sigma)}{\lambda(\tau)}\sigma^{-1}}\right)\right]\mathcal{F}\big(\mathcal{D}(N_2(\veps))\big)\left(\sigma,\frac{\lambda(\tau)^{2}}{\lambda(\sigma)^{2}}\xi\right)\right)'\right)\,d\sigma
\end{align*}
Then the corresponding weighted $L^2_{d\xi}(\xi<1)$-norm is bounded by 
\begin{align*}
\lesssim&  \alpha\int_{\tau_0}^{\tau}\left\|\xi^{-1}\langle\log\xi\rangle^{-1-\kappa}\frac{\lambda(\sigma)}{\lambda(\tau)}\sigma^{-1}\left[1-\chi\left(C^{-1}\frac{\xi^{\frac12}}{\frac{\lambda(\sigma)}{\lambda(\tau)}\sigma^{-1}}\right)\right]\left\langle \frac{\lambda(\tau)^{2}}{\lambda(\sigma)^{2}}\xi\right\rangle^{-N}\right\|_{L^2_{d\xi}(\xi<1)}\\
\cdot&\frac{\lambda(\sigma)}{\lambda(\tau)}\sigma^{-1}\frac{\lambda(\tau_0)}{\lambda(\sigma)}\left\langle\log\frac{\lambda(\sigma)}{\lambda(\tau_0)}\right\rangle^{-1-\frac{\kappa}{2}}\,d\sigma\\
\lesssim& \alpha C^{-1}\frac{\lambda(\tau_0)}{\lambda(\tau)}\int_{\tau_0}^{\tau}\sigma^{-1}\left\langle\log\frac{\lambda(\tau)}{\lambda(\sigma)}\right\rangle^{-1-\kappa}\left\langle\log\frac{\lambda(\sigma)}{\lambda(\tau_0)}\right\rangle^{-1-\frac{\kappa}{2}}\,d\sigma\\
\lesssim& \alpha C^{-1}\frac{\lambda(\tau_0)}{\lambda(\tau)}\left\langle\log\frac{\lambda(\tau)}{\lambda(\tau_0)}\right\rangle^{-1-\frac{\kappa}{2}}
\end{align*}
This again furnishes the required bound with a smallness gain, provided we choose $C$ very large. This concludes the required small frequency bounds, i. e. for the regime $\xi<1$. For completeness' sake we also consider here the large frequency bounds, for the term $\mathcal{D}_{\tau}\tilde{x}_{good}(\tau, \xi)$, again arising from the term $N_2(\veps)$ as displayed above. We proceed as before via integration by parts, which leads us to bound the expressions 
\begin{align*}
\left\|\xi^{-\frac12}\frac{\lambda(\sigma)}{\lambda(\tau)}\sin\left(\lambda(\tau)\xi^{\frac{1}{2}}\int_{\sigma}^{\tau}\lambda(u)^{-1}du\right)\left[1-\chi\left(C^{-1}\frac{\xi^{\frac12}}{\frac{\lambda(\sigma)}{\lambda(\tau)}\sigma^{-1}}\right)\right]\mathcal{F}\big(\mathcal{D}(N_2(\veps))\big)\left(\sigma,\frac{\lambda(\tau)^{2}}{\lambda(\sigma)^{2}}\xi\right)\Big|_{\tau_0}^{\tau}\right\|_{S_1(\xi>1)},
\end{align*}
\begin{align*}
\left\|\int_{\tau_0}^{\tau}\frac{\sin\left(\lambda(\tau)\xi^{\frac{1}{2}}\int_{\sigma}^{\tau}\lambda(u)^{-1}du\right)}{\xi^{\frac12}}\frac{\partial}{\partial{\sigma}}\left(\frac{\lambda(\sigma)}{\lambda(\tau)}\left[1-\chi\left(C^{-1}\frac{\xi^{\frac12}}{\frac{\lambda(\sigma)}{\lambda(\tau)}\sigma^{-1}}\right)\right]\mathcal{F}\big(\mathcal{D}(N_2(\veps))\big)\left(\sigma,\frac{\lambda(\tau)^{2}}{\lambda(\sigma)^{2}}\xi\right)\right)\,d\sigma\right\|_{S_1(\xi>1)}
\end{align*}
The only relevant boundary term $\sigma = \tau_0$) is then easily seen to be bounded by $\lesssim_{\tau_0}\big|c(\tau_0)\big|\big(\frac{\lambda(\tau_0)}{\lambda(\tau)}\big)^N$, which is much better than what we need. On the other hand, the integral term can in principle be estimated directly, but in order to gain smallness, we have to do another integration by parts, as before. This leads to a boundary term at $\sigma = \tau$ which is estimated by 
\begin{align*}
&\lesssim \big\|\xi^{1+\kappa-}\langle\log\xi\rangle^{-1-\kappa}\tau^{-1}[1-\chi(C^{-1}\tau\xi^{\frac12})]\langle\xi\rangle^{-N}\big\|_{L^2_{d\xi}(\xi>1)}\cdot\|c\|\frac{\lambda(\tau_0)}{\lambda(\tau)}\left\langle\log\frac{\lambda(\tau)}{\lambda(\tau_0)}\right\rangle^{-1-\kappa}\\
&\lesssim \tau_0^{-1}\cdot\|c\|\frac{\lambda(\tau_0)}{\lambda(\tau)}\left\langle\log\frac{\lambda(\tau)}{\lambda(\tau_0)}\right\rangle^{-1-\kappa}\\
\end{align*}
which gives the required smallness gain by choosing $\tau_0\gg 1$. The resulting integral term is also small by a factor $\tau_0^{-1}$, we omit the straightforward details.  
\end{proof}
In a similar vein, we need to deal with the contributions of the remaining terms on the right hand side of \eqref{eq Dveps temp 2} comprised by $\calR(\veps, \calD\veps)$. 

\begin{proposition}\label{prop:anotherresonantlin} Let $\tilde{N}_2(\veps)$ be the expression
\[
\left[\left(\frac{\lambda'}{\lambda}\right)' + \left(\frac{\lambda'}{\lambda}\right)^2\right]\left(\frac{4R}{(1+R^2)^2}\veps\right)
\]
and assume that $\veps(\tau, R) = c(\tau)\phi_0(R)$. Also, assume that $c'(\tau)$ admits the representation \eqref{eq:c'explicit} with $A(\tau), B(\tau)$ satisfying 
\[
\sup_{\tau\geq\tau_0}\frac{\lambda(\tau)}{\lambda(\tau_0)}\left\langle\log\frac{\lambda(\tau)}{\lambda(\tau_0)}\right\rangle^{1+\frac{\kappa}{2}}[\tau\big|A(\tau)\big| + \big|B(\tau)\big|]\leq \alpha.
\]
Then letting 
\[
\tilde{x}(\tau, \xi): = \xi^{-\frac{1}{2}}\int_{\tau_{0}}^{\tau}\sin\left(\lambda(\tau)\xi^{\frac{1}{2}}\int_{\sigma}^{\tau}\lambda(u)^{-1}du\right)\mathcal{F}\left(\tilde{N}_2(\veps)\right)\left(\sigma,\frac{\lambda(\tau)^{2}}{\lambda(\sigma)^{2}}\xi\right)d\sigma,
\]
we have the bound
\begin{align*}
&\sup_{\tau\geq \tau_0}\frac{\lambda(\tau)}{\lambda(\tau_0)}\left\langle\log\frac{\lambda(\tau)}{\lambda(\tau_0)}\right\rangle^{1+\frac{\kappa}{2}}\big[\big\|\tilde{x}(\tau, \xi)\big\|_{S_0} + \big\|\mathcal{D}_{\tau}\tilde{x}(\tau, \xi)\big\|_{S_1}\big]\\
&\lesssim \big|c(\tau_0)\big| + \alpha.
\end{align*}
One may include a factor $\gamma(\tau_0)$ with $\lim_{\tau_0\rightarrow\infty}\gamma(\tau_0) = 0$ in front of $\alpha$, provided one omits the term $ \big\|\tilde{x}_{good}(\tau, \xi)\big\|_{S_0}$ on the left. A similar statement also applies to the expression 
\[
\chi_{R\gtrsim\tau}\frac{\lambda'}{\lambda}\left(\partial_{\tau} + \frac{\lambda'}{\lambda}R\partial_R\right)\left(\frac{4R}{(1+R^2)^2}\veps\right),
\]
and a more complicated statement applies to the expression 
\[
\chi_{R\ll\tau}\frac{\lambda'}{\lambda}\left(\partial_{\tau} + \frac{\lambda'}{\lambda}R\partial_R\right)\left(\frac{4R}{(1+R^2)^2}\veps\right),
\]

if one goes deeper into the iteration, i. e. invokes the fine structure of $B(\tau)$, as detailed in Lemma~\ref{lem:reiterate2} below. 
\end{proposition}
The proof of this is analogous to the one of the preceding proposition, and hence omitted, except of course for the last type of terms, which will be dealt with below. 

\subsection{The effect of $\mathcal{D}\veps$ on $c(\tau)$.} Here we analyze the solution of \eqref{ctau ODE}, using \eqref{eq:L_cParametrix} and assuming that the initial data $c(\tau_0), c'(\tau_0)$ are fixed. There are two source terms, $h(\tau), n(\tau)$, of which $h(\tau)$ is naturally the more delicate one, as it decays only as fast as $\mathcal{D}\veps$. 
\\

{\it{Contribution of the principal source term $h(\tau)$.}} Write 
\[
\mathcal{D}\veps(\tau, R) = \int_0^\infty \phi(R, \xi)\xb(\tau, \xi)\tilde{\rho}(\xi)\,d\xi.
\]
Then recall the expansion $\phi(R, \xi) = cR^2 + O(R^4\xi)$ with some $c\neq 0$ as $R\rightarrow 0$. This then formally implies 
\begin{equation}\label{eq:hfourierformula}
h(\tau) = \lim_{R\rightarrow 0}R^{-1}\mathcal{D}^*\mathcal{D}\veps = c'\int_0^\infty x(\tau, \xi)\tilde{\rho}(\xi)\,d\xi
\end{equation}
Now when $\xb$ is admissible in the sense of admitting a representation formula \eqref{eq:admissiblexb}, the expression on the right is not defined in a pointwise sense (as far as the first two terms in \eqref{eq:admissiblexb} are concerned), but only as a distribution in $\tau$, in light of Lemma~\ref{lem:cintegralbound}. Thus when invoking formula \eqref{eq:L_cParametrix}, we have to interpret it accordingly to obtain 
\begin{lemma}\label{lem:hbound} Assume that $\xb$ is admissible and further that $r(\tau) = h(\tau)$ as before, interpreted as in Lemma~\ref{lem:cintegralbound}. Then defining $c(\tau) = y(\tau)$ as in \eqref{eq:L_cParametrix}, we have 
\[
\big|y(\tau)\big| + \tau\big|y'(\tau)\big|\lesssim \tau^2\frac{\lambda(\tau)}{\lambda(\tau_0)}\left\langle\log\frac{\lambda(\tau)}{\lambda(\tau_0)}\right\rangle^{-1-\frac{\kappa}{2}}\|\xb\|,
\]
where we set 
\begin{align*}
\|\xb\|: = \sup_{\tau\geq \tau_0}\frac{\lambda(\tau)}{\lambda(\tau_0)}\left\langle\log\frac{\lambda(\tau)}{\lambda(\tau_0)}\right\rangle^{1+\frac{\kappa}{2}}\big[\sum_{\pm}\tau\big|a_{kj}^{\pm}(\tau)\big|& + \sum_{\pm}\tau\big|(a_{kj}^{\pm})'(\tau)\big|+ \tau\big\|(|b|+|\xi^{\frac12}c| + |\xi^{\frac12}d|)(\tau, \xi)\big\|_{L_{\xi}^{\infty}}\\& + \tau\big\|\langle\xi\rangle^{-\frac12}\mathcal{D}_{\tau}b(\tau, \xi)\big\|_{L_{\xi}^{\infty}} + \big\|x_{good}(\tau, \xi)\big\|_{S_0} + \big\|\mathcal{D}_{\tau}x_{good}(\tau, \xi)\big\|_{S_1}\big]
\end{align*}

\end{lemma}
 
\begin{remark}\label{rem:nosmallness1} Note that as is to be expected, there is no smallness gain here. 
\end{remark}
\begin{proof} Writing 
\[
h(\tau) = \partial_{\tau}A(\tau) + B(\tau)
\]
as in Lemma~\ref{lem:cintegralbound} and with 
\[
\tau\big|A(\tau)\big| + \big|B(\tau)\big|\lesssim \|x\|\cdot\frac{\lambda(\tau_0)}{\lambda(\tau)}\left\langle\log\frac{\lambda(\tau)}{\lambda(\tau_0)}\right\rangle^{-1-\frac{\kappa}{2}},
\]
we immediately infer that
\begin{equation}\label{eq:cexplicit}\begin{split}
c(\tau) &= \nu\big(-(2+\nu^{-1})\tau^{-1-\nu^{-1}}\int_{\tau_0}^{\tau}\sigma^{1+\nu^{-1}}A(\sigma)\,d\sigma + (2+2\nu^{-1})\tau^{-1-2\nu^{-1}}\int_{\tau_0}^{\tau}\sigma^{1+2\nu^{-1}}A(\sigma)\,d\sigma\big)\\
& +  \nu\big(\tau^{-1-\nu^{-1}}\int_{\tau_0}^{\tau}\sigma^{2+\nu^{-1}}B(\sigma)\,d\sigma - \tau^{-1-2\nu^{-1}}\int_{\tau_0}^{\tau}\sigma^{2+2\nu^{-1}}B(\sigma)\,d\sigma\big)\\
\end{split}\end{equation}
satisfies the desired bound. Furthermore, we also note here the formula
\begin{equation}\label{eq:c'explicit}\begin{split}
c'(\tau) &= A(\tau) + \nu\big((2+\nu^{-1})(1+\nu^{-1})\tau^{-2-\nu^{-1}}\int_{\tau_0}^{\tau}\sigma^{1+\nu^{-1}}A(\sigma)\,d\sigma\\& - (2+2\nu^{-1})(1+2\nu^{-1})\tau^{-2-2\nu^{-1}}\int_{\tau_0}^{\tau}\sigma^{1+2\nu^{-1}}A(\sigma)\,d\sigma\big)\\
& +  \nu\big(-(1+\nu^{-1})\tau^{-2-\nu^{-1}}\int_{\tau_0}^{\tau}\sigma^{2+\nu^{-1}}B(\sigma)\,d\sigma + (1+2\nu^{-1})\tau^{-2-2\nu^{-1}}\int_{\tau_0}^{\tau}\sigma^{2+2\nu^{-1}}B(\sigma)\,d\sigma\big)\\
\end{split}\end{equation}
 
\end{proof}

{\it{Contribution of the principal source term $n(\tau)$.}} From \eqref{ctau ODE}, we recall that 
\[
n(\tau) = \lim_{R\rightarrow 0}R^{-1} \lambda^{-2}N(\veps)(\tau, R).
\]
Then in light of \eqref{eq:epsequation1} and Lemma~\ref{lem:epsbound}, we infer the following 
\begin{lemma}\label{lem:nbound} Assume that $\xb$ is admissible, define $\|\xb\|$ as in Lemma~\ref{lem:hbound}, and further let 
\[
\veps(\tau, R) = \phi(\mathcal{D}\veps) + c(\tau)\phi_0(R)
\]
as in \eqref{eqn:epsdecomp}. Then we have the bound 
\begin{align*}
\big|n(\tau)\big|&\lesssim\frac{\lambda(\tau_0)}{\lambda(\tau)}\left\langle\log\frac{\lambda(\tau)}{\lambda(\tau_0)}\right\rangle^{-1-\frac{\kappa}{2}}\tau^{-2}\|x\| + \gamma(\tau_0)\|c\|\tau^{-2}\frac{\lambda(\tau_0)}{\lambda(\tau)}\left\langle\log\frac{\lambda(\tau)}{\lambda(\tau_0)}\right\rangle^{-1-\frac{\kappa}{2}}\\& + D(\tau_0)\tau^{-2}\frac{\lambda(\tau_0)}{\lambda(\tau)}\left\langle\log\frac{\lambda(\tau)}{\lambda(\tau_0)}\right\rangle^{-1-\frac{\kappa}{2}}\left[\|x\|^2 + \|x\|^3 + \|c\|^2 + \|c\|^3\right]
\end{align*}
with $\lim_{\tau_0\rightarrow\infty}\gamma(\tau_0) = 0$, provided $\nu$ is sufficiently small. 
\end{lemma}
Note that we gain a power $\tau^{-2}$ in terms of additional decay here. 
\\

The proof is straightforward by observing the asymptotics of $u^{\nu} - Q(R)$ described in \cite{KST2,GaoKr}. We observe in particular that inserting this bound in the parametrix \eqref{eq:L_cParametrix} (with $n(\sigma)$ instead of $r(\sigma)$), one has a smallness gain for all contributions for $\|y\|$. We observe that the preceding bound is remarkable in that there is {\it{asymptotically as $\tau_0\rightarrow +\infty$ no contribution linear in $\|c\|$.}} The reason for this is that we in fact have 
\begin{align*}
\left|\lim_{R\rightarrow 0}R^{-1}\frac{\sin(u^\nu - Q(R))\sin(u^{\nu} + Q(R))}{R^2}c(\tau)\phi_0(R)\right|\leq \gamma(\tau_0)\|c\|\cdot \frac{\lambda(\tau_0)}{\lambda(\tau)}\left\langle\log\frac{\lambda(\tau)}{\lambda(\tau_0)}\right\rangle^{-1-\frac{\kappa}{2}},
\end{align*}
which comes from the fact that all the leading corrections $v_j$ used to construct $u^{\nu}$ in \cite{KST2,GaoKr} vanish to higher than first order at $R = 0$, and so only the final correction $\epsilon$ in \cite{KST2,GaoKr} contributes here, which, however, vanishes very rapidly toward $\tau = +\infty$. 
 
 \subsection{The effect of $c(\tau)$ on $c(\tau)$ after double iteration; smallness gain}
 
 In this subsection, we deal with the delicate technical issue of controlling and in fact gaining smallness for certain terms arising after twofold iteration. Specifically, recall from Proposition~\ref{prop:resonantlin} that the contribution to $\tilde{x}_{good}(\tau, \xi)$ by the term 
 \[
 \chi_{R\ll\tau}\mathcal{D}\big(N_2(\veps)\big),\,\veps(\tau, R) = c(\tau)\phi_0(R)
 \]
via the Duhamel parametrix displayed in  the statement of Proposition~\ref{prop:resonantlin} does not gain any smallness, i. e. the best one can assert is that 
\begin{align*}
\sup_{\tau\geq\tau_0}\frac{\lambda(\tau)}{\lambda(\tau_0)}\left\langle\log\frac{\lambda(\tau)}{\lambda(\tau_0)}\right\rangle^{1+\frac{\kappa}{2}}\big\|\tilde{x}_{good}(\tau, \xi)\big\|_{S_0}\leq D\left[\|c\| + \big|c(\tau_0)\big|\right]
\end{align*}
for some absolute constant $D$ only depending on $\nu$. If one then uses this coefficient $\tilde{x}_{good}$ as Fourier coefficient of $\mathcal{D}\veps(\tau, R)$ and in turn as a source term for the $c$ equation \eqref{ctau ODE} via 
\[
h(\tau) = \lim_{R\rightarrow 0}R^{-1}\mathcal{D}^*\mathcal{D}\veps, 
\]
then application of Lemma~\ref{lem:hbound} again does not result in a smallness gain, and so we face the possibility of divergence in the eventual iteration constructing the pair $\big(c(\tau), \xb(\tau, \xi)\big)$. Here we show that there actually is a smallness gain in this re-iteration step: 
\begin{lemma}\label{lem:ctocdelicate}
If $\veps(\tau, R) = c(\tau)\phi_0(R)$, and furthermore $c, c'$ are given by \eqref{eq:cexplicit}, \eqref{eq:c'explicit}, then we have the following bound if $c(\tau_0) = 0$:
\begin{align*}
&\left|\int_0^\infty\int_{\tau_0}^{\tau}\frac{\sin[\lambda(\tau)\xi^{\frac12}\int_{\sigma}^{\tau}\lambda^{-1}(u)\,du]}{\xi^{\frac12}}\mathcal{F}\big(\mathcal{D}\big(\chi_{R\ll\tau}N_2(\veps)\big)\left(\sigma, \frac{\lambda^2(\tau)}{\lambda^2(\sigma)}\xi\right)\,d\sigma \tilde{\rho}(\xi)d\xi\right|\\
&\leq \alpha\gamma(\tau_0)\frac{\lambda(\tau)}{\lambda(\tau_0)}\left\langle\log\frac{\lambda(\tau)}{\lambda(\tau_0)}\right\rangle^{-1-\frac{\kappa}{2}},
\end{align*}
where $\lim_{\tau_0\rightarrow\infty}\gamma(\tau_0) = 0$ and furthermore we set
\[
\alpha = \sup_{\tau\geq\tau_0}\frac{\lambda(\tau)}{\lambda(\tau_0)}\left\langle\log\frac{\lambda(\tau)}{\lambda(\tau_0)}\right\rangle^{1+\frac{\kappa}{2}}\left[\tau\big|A(\tau)\big| + \big|B(\tau)\big|\right].
\]
\end{lemma}
\begin{proof} This is largely analogous to the proof of Proposition~\ref{prop:resonantlin}. The idea is again to invoke integration by parts with respect to $\sigma$, but we have to be careful to avoid the boundary term at $\sigma = \tau$, which was responsible for the fact that we could not gain smallness in general in Proposition~\ref{prop:resonantlin}. In fact, this boundary term almost vanishes, as can be see seen from the following: assume that 
\[
x(\xi) = 
\xi^{-1}\mathcal{F}(\mathcal{D}g)(\xi). 
\]
This then means that 
\begin{align*}
x(\xi) &= \xi^{-1}\langle \mathcal{D}g, \xi^{-1}\mathcal{D}\big(R^{-\frac12}\phi_{KST}(R, \xi)\rangle_{L^2_{R\,dR}}\\
& = \xi^{-2}\langle g, \mathcal{L}\big(R^{-\frac12}\phi_{KST}(R, \xi)\rangle_{L^2_{R\,dR}}\\
& = \xi^{-1}\langle g, \big(R^{-\frac12}\phi_{KST}(R, \xi)\rangle_{L^2_{R\,dR}}\\
&=\xi^{-1}\mathcal{F}_{KST}\big(R^{\frac12}g\big)(\xi).
\end{align*}
It follows that 
\begin{equation}\label{eq:importantidentity}
\int_0^\infty x(\xi)\tilde{\rho}(\xi)\,d\xi = \int_0^\infty \mathcal{F}_{KST}\big(R^{\frac12}g\big)(\xi)\rho(\xi)\,d\xi = c\lim_{R\rightarrow 0}R^{-1}g(R).
\end{equation}
We shall use the consequence that
\[
\int_0^\infty\xi^{-1}\mathcal{F}\big(\mathcal{D}\big(\chi_{R\ll\tau}N_2(\veps)\big)(\tau, \xi)\tilde{\rho}(\xi)\,d\xi = \lim_{R\rightarrow 0}R^{-1}N_2(\veps)
\]
But we have already observed(after Lemma~\ref{lem:nbound}) that this is bounded in absolute value by 
\[
\leq \gamma(\tau_0)\|c\|\cdot \frac{\lambda(\tau_0)}{\lambda(\tau)}\left\langle\log\frac{\lambda(\tau)}{\lambda(\tau_0)}\right\rangle^{-1-\frac{\kappa}{2}}
\]
which is as desired since $\|c\|\lesssim \alpha$. 
\\

Now back to the double integral in the lemma, we first observe that letting $\chi\left(C^{-1}\frac{\xi^{\frac12}}{\frac{\lambda(\sigma)}{\lambda(\tau)}\sigma^{-1}}\right)$ be a cutoff localizing smoothly to $\xi^{\frac12}\lesssim C\frac{\lambda(\sigma)}{\lambda(\tau)}\sigma^{-1}$ for some large constant $C$, then we get 
\begin{align*}
&\left|\int_0^\infty\int_{\tau_0}^{\tau}\chi\left(C^{-1}\frac{\xi^{\frac12}}{\frac{\lambda(\sigma)}{\lambda(\tau)}\sigma^{-1}}\right)\frac{\sin[\lambda(\tau)\xi^{\frac12}\int_{\sigma}^{\tau}\lambda^{-1}(u)\,du]}{\xi^{\frac12}}\mathcal{F}\big(\mathcal{D}\big(\chi_{R\ll\tau}N_2(\veps)\big)\left(\sigma, \frac{\lambda^2(\tau)}{\lambda^2(\sigma)}\xi\right)\,d\sigma \tilde{\rho}(\xi)d\xi\right|\\
&\lesssim \alpha\int_{\tau_0}^{\tau} C\frac{\lambda(\sigma)}{\lambda(\tau)}\sigma^{-1}\left\langle\log \left(C^{-1}\frac{\lambda(\tau)}{\lambda(\sigma)}\sigma\right)\right\rangle^{-2}\cdot \frac{\lambda(\tau_0)}{\lambda(\sigma)}\left\langle\log\frac{\lambda(\sigma)}{\lambda(\tau_0)}\right\rangle^{-1-\frac{\kappa}{2}}\,d\sigma\\
& \leq \alpha \gamma(\tau_0)\frac{\lambda(\tau_0)}{\lambda(\tau)}\left\langle\log\frac{\lambda(\tau)}{\lambda(\tau_0)}\right\rangle^{-1-\frac{\kappa}{2}}.
\end{align*}
Replacing $\chi(\ldots)$ by $[1-\chi(\ldots)]$, we perform an integration by parts with respect to $\sigma$. We thereby encounter the boundary term 
\begin{align*}
&\int_0^\infty [1- \chi(C^{-1}\tau\xi^{\frac12})]\cdot\xi^{-1}\mathcal{F}\big(\mathcal{D}\big(\chi_{R\ll\tau}N_2(\veps)\big)(\tau, \xi)\tilde{\rho}(\xi)\,d\xi\\
& = -\int_0^\infty \chi(C^{-1}\tau\xi^{\frac12})\cdot\xi^{-1}\mathcal{F}\big(\mathcal{D}\big(\chi_{R\ll\tau}N_2(\veps)\big)(\tau, \xi)\tilde{\rho}(\xi)\,d\xi\\
& + O\left(\gamma(\tau_0)\alpha\cdot \frac{\lambda(\tau_0)}{\lambda(\tau)}\left\langle\log\frac{\lambda(\tau)}{\lambda(\tau_0)}\right\rangle^{-1-\frac{\kappa}{2}}\right),
\end{align*}
where we have taken advantage of the above relation \eqref{eq:importantidentity}. The principal term on the right can then be bounded by 
\begin{align*}
&\left|\int_0^\infty \chi(C^{-1}\tau\xi^{\frac12})\cdot\xi^{-1}\mathcal{F}\big(\mathcal{D}\big(\chi_{R\ll\tau}N_2(\veps)\big)(\tau, \xi)\tilde{\rho}(\xi)\,d\xi\right|\\
&\lesssim \left\langle\log\frac{\tau_0}{C}\right\rangle^{-1}\frac{\lambda(\tau_0)}{\lambda(\tau)}\left\langle\log\frac{\lambda(\tau)}{\lambda(\tau_0)}\right\rangle^{-1-\frac{\kappa}{2}}\cdot\alpha,
\end{align*}
which is also as desired setting $\gamma(\tau_0) = \langle\log\frac{\tau_0}{C}\rangle^{-1}$. 
In addition to the boundary term, we also encounter the double integral 
\begin{align*}
\int_0^\infty\int_{\tau_0}^{\tau}\frac{\cos[\lambda(\tau)\xi^{\frac12}\int_{\sigma}^{\tau}\lambda^{-1}(u)\,du]}{\xi}\frac{\partial}{\partial{\sigma}}\left([1-\chi(\ldots)]\frac{\lambda(\sigma)}{\lambda(\tau)}\mathcal{F}\big(\mathcal{D}\big(\chi_{R\ll\tau}N_2(\veps)\big)\left(\sigma, \frac{\lambda^2(\tau)}{\lambda^2(\sigma)}\xi\right)\right)\,d\sigma \tilde{\rho}(\xi)d\xi
\end{align*}
Here if the operator $\frac{\partial}{\partial{\sigma}}$ hits the factor $c(\sigma)$ and results in the term $A(\sigma)$, see \eqref{eq:c'explicit}, we cannot perform another integration by parts with respect to $\sigma$, but can take advantage of the faster decay of $A(\sigma)$. Taking advantage of the rapid decay of $\mathcal{F}\big(\ldots\big)$ for frequencies $>1$, we then bound the $\xi$-integral by 
\begin{align*}
\lesssim \left\langle\log\frac{\lambda(\tau)}{\lambda(\sigma)}\right\rangle^{-2}\log\sigma\cdot\alpha\cdot\sigma^{-2}\frac{\lambda(\sigma)}{\lambda(\tau)}\cdot\frac{\lambda(\tau_0)}{\lambda(\sigma)}\left\langle\log\frac{\lambda(\sigma)}{\lambda(\tau_0)}\right\rangle^{-1-\frac{\kappa}{2}}
\end{align*}
This in turn can be integrated between $\tau_0$ and $\tau$ to result in the desired upper bound 
\[
\leq\gamma(\tau_0)\alpha\cdot \frac{\lambda(\tau_0)}{\lambda(\tau)}\left\langle\log\frac{\lambda(\tau)}{\lambda(\tau_0)}\right\rangle^{-1-\frac{\kappa}{2}}
\]
with $\gamma(\tau_0) \lesssim \frac{\log\tau_0}{\tau_0}$. 
\\
Thus assume that if the operator $\frac{\partial}{\partial{\sigma}}$ hits the factor $c(\sigma)$, it does not result in $A(\sigma)$. Calling these remaining terms $\frac{\partial}{\partial{\sigma}}\big(\ldots\big)'$, we perform an integration by parts with respect to $\sigma$. Here the boundary terms vanish, and we reduce to bounding the double integral 
\begin{align*}
&\int_0^\infty\int_{\tau_0}^{\tau}\frac{\sin[\lambda(\tau)\xi^{\frac12}\int_{\sigma}^{\tau}\lambda^{-1}(u)\,du]}{\xi^{\frac32}}\\
\cdot&\frac{\partial}{\partial{\sigma}}\left[\frac{\lambda(\sigma)}{\lambda(\tau)}\frac{\partial}{\partial{\sigma}}\left([1-\chi(\ldots)]\frac{\lambda(\sigma)}{\lambda(\tau)}\mathcal{F}\big(\mathcal{D}\big(\chi_{R\ll\tau}N_2(\veps)\big)(\sigma, \frac{\lambda^2(\tau)}{\lambda^2(\sigma)}\xi)\right)'\right]\,d\sigma \tilde{\rho}(\xi)d\xi
\end{align*}
For fixed $\sigma$, the $\xi$-integral here is bounded in absolute value by 
\begin{align*}
&\lesssim \kappa(\sigma,\tau_0)\frac{\lambda^2(\sigma)}{\lambda^2(\tau)}\sigma^{-2}\left\|\xi^{-\frac32}\langle\log\xi\rangle^{-2}\left[1 - \chi\left(C^{-1}\frac{\xi^{\frac12}}{\frac{\lambda(\sigma)}{\lambda(\tau)}\sigma^{-1}}\right)\right]\right\|_{L^1_{d\xi}(\xi<1)}\\
& + \kappa(\sigma,\tau_0)\frac{\lambda^2(\sigma)}{\lambda^2(\tau)}\sigma^{-2}\left\|\xi^{-\frac32}\xi^2\left(\frac{\lambda^2(\tau)}{\lambda^2(\sigma)}\xi\right)^{-N}\left[1 - \chi\left(C^{-1}\frac{\xi^{\frac12}}{\frac{\lambda(\sigma)}{\lambda(\tau)}\sigma^{-1}}\right)\right]\right\|_{L^1_{d\xi}(\xi\geq 1)}\\
\end{align*}
where we use the notation 
\[
\kappa(\sigma,\tau_0) = \alpha\frac{\lambda(\tau_0)}{\lambda(\sigma)}\left\langle\log\frac{\lambda(\sigma)}{\lambda(\tau_0)}\right\rangle^{-1-\frac{\kappa}{2}}
\]
Then the first term on the right above is bounded by 
\begin{align*}
&\lesssim C^{-1}\frac{\lambda(\sigma)}{\lambda(\tau)}\sigma^{-1}\cdot \left\langle\log\left(C^{-1}\frac{\lambda(\tau)}{\lambda(\sigma)}\sigma\right)\right\rangle^{-2}\cdot  \alpha\frac{\lambda(\tau_0)}{\lambda(\sigma)}\left\langle\log\frac{\lambda(\sigma)}{\lambda(\tau_0)}\right\rangle^{-1-\frac{\kappa}{2}}\\
&\lesssim \alpha\langle\log(C^{-1}\tau_0)\rangle^{-1+\frac{\kappa}{2}}\frac{\lambda(\tau_0)}{\lambda(\tau)}\cdot\sigma^{-1}\left\langle\log\frac{\lambda(\tau)}{\lambda(\sigma)}\right\rangle^{-1-\frac{\kappa}{2}}\left\langle\log\frac{\lambda(\sigma)}{\lambda(\tau_0)}\right\rangle^{-1-\frac{\kappa}{2}}\\
 \end{align*}
provided $\tau_0>C$, integrating this over $\sigma$ furnishes the desired upper bound with $\gamma(\tau_0) = \langle\log(C^{-1}\tau_0)\rangle^{-1+\frac{\kappa}{2}}$. 
The second term above, comprising the $L^1$-norm over large frequencies $\xi\geq 1$, is much less delicate, and can be bounded by 
\[
\lesssim  \alpha\frac{\lambda^N(\sigma)}{\lambda^N(\tau)}\sigma^{-2}\left\langle\log\frac{\lambda(\tau)}{\lambda(\sigma)}\right\rangle^{-1-\frac{\kappa}{2}}\left\langle\log\frac{\lambda(\sigma)}{\lambda(\tau_0)}\right\rangle^{-1-\frac{\kappa}{2}}
\]
and integration over $\sigma$ furnishes the desired bound with $\gamma(\tau_0)\sim \tau_0^{-1}$. 
 This completes the proof of the lemma. 
\end{proof}

In a similar vein, we have 
\begin{lemma}\label{lem:ctocdelicate1}
If $\veps(\tau, R) = c(\tau)\phi_0(R)$, and furthermore $c, c'$ are given by \eqref{eq:cexplicit}, \eqref{eq:c'explicit}, then we have the following bound if $c(\tau_0) = 0$:
\begin{align*}
&\left|\int_0^\infty\int_{\tau_0}^{\tau}\frac{\sin[\lambda(\tau)\xi^{\frac12}\int_{\sigma}^{\tau}\lambda^{-1}(u)\,du]}{\xi^{\frac12}}\mathcal{F}\big(\chi_{R\ll\sigma}\tilde{N}_2(\veps)\big)\left(\sigma, \frac{\lambda^2(\tau)}{\lambda^2(\sigma)}\xi\right)\,d\sigma \tilde{\rho}(\xi)d\xi\right|\\
&\leq \alpha\gamma(\tau_0)\frac{\lambda(\tau)}{\lambda(\tau_0)}\left\langle\log\frac{\lambda(\tau)}{\lambda(\tau_0)}\right\rangle^{-1-\frac{\kappa}{2}},
\end{align*}
where $\lim_{\tau_0\rightarrow\infty}\gamma(\tau_0) = 0$ and furthermore we set
\[
\alpha = \sup_{\tau\geq\tau_0}\frac{\lambda(\tau)}{\lambda(\tau_0)}\left\langle\log\frac{\lambda(\tau)}{\lambda(\tau_0)}\right\rangle^{1+\frac{\kappa}{2}}[\tau\big|A(\tau)\big| + \big|B(\tau)\big|]
\]
Recall that $\tilde{N}_2$ stands for the terms listed in Proposition~\ref{prop:anotherresonantlin}. 
\end{lemma}
\begin{proof} We only deal with the delicate term 
\[
\tilde{N}_2(\veps) = \beta_{\nu}(\tau)c'(\tau)\frac{4R}{(1+R^2)^2}\phi_0(R),\,\beta_{\nu}(\tau) = \frac{\lambda'}{\lambda},
\]
the others being much simpler to deal with, on account of the better decay of $\left(\frac{\lambda'}{\lambda}\right)^2,\,\left(\frac{\lambda'}{\lambda}\right)'$. Note that this coincides inside the light cone $R\lesssim \tau$ with 
\[
\beta_{\nu}(\tau)c'(\tau)\mathcal{D}\left(\chi_{R\lesssim \tau}\phi_0(R)\int_0^R\chi_{s\ll\tau}\frac{4s}{(1+s^2)^2}\,ds\right).
\]
Then
\begin{equation}\label{eq:vanishing21}
\lim_{R\rightarrow 0}R^{-1}\left(\chi_{R\lesssim \tau}\phi_0(R)\int_0^R\chi_{s\ll\tau}\frac{4s}{(1+s^2)^2}\,ds\right) = 0.
\end{equation}
Introduce a cutoff $\chi\left(\xi^{\frac12}\frac{\lambda(\tau)}{\lambda(\sigma)}\sigma\right)$, smoothly localizing to $\xi^{\frac12}\lesssim \left(\frac{\lambda(\tau)}{\lambda(\sigma)}\sigma\right)^{-1}$. Then we get 
\begin{align*}
&\left|\int_0^\infty\int_{\tau_0}^{\tau}\chi\left(\xi^{\frac12}\frac{\lambda(\tau)}{\lambda(\sigma)}\sigma\right)\frac{\sin[\lambda(\tau)\xi^{\frac12}\int_{\sigma}^{\tau}\lambda^{-1}(u)\,du]}{\xi^{\frac12}}\mathcal{F}\big(\chi_{R\ll\tau}\tilde{N}_2(\veps)\big)\left(\sigma, \frac{\lambda^2(\tau)}{\lambda^2(\sigma)}\xi\right)\,d\sigma \tilde{\rho}(\xi)d\xi\right|\\
&\lesssim \alpha\int_{\tau_0}^{\tau}\left(\frac{\lambda(\tau)}{\lambda(\sigma)}\sigma\right)^{-1}\left\langle\log\left(\frac{\lambda(\tau)}{\lambda(\sigma)}\sigma\right)\right\rangle^{-2}\frac{\lambda(\tau_0)}{\lambda(\sigma)}\left\langle\log\frac{\lambda(\sigma)}{\lambda(\tau_0)}\right\rangle^{-1-\frac{\kappa}{2}}\,d\sigma\\
&\lesssim \alpha\langle \log\tau_0\rangle^{-1+\frac{\kappa}{2}}\frac{\lambda(\tau_0)}{\lambda(\tau)}\left\langle\log\frac{\lambda(\tau)}{\lambda(\tau_0)}\right\rangle^{-1-\frac{\kappa}{2}},
\end{align*}
which is as desired with $\gamma(\tau_0) = \langle \log\tau_0\rangle^{-1+\frac{\kappa}{2}}$. Then include the cutoff $\left[1 - \chi\left(\xi^{\frac12}\frac{\lambda(\tau)}{\lambda(\sigma)}\sigma\right)\right]$, and perform integration by parts with respect to $\sigma$. Arguing as in the preceding lemma, and using \eqref{eq:vanishing21}, one sees that the boundary term at $\sigma = \tau$ is negligible, and reduces to bounding the term 
\begin{align*}
&\left|\int_0^\infty\int_{\tau_0}^{\tau}\left[1 - \chi\left(\xi^{\frac12}\frac{\lambda(\tau)}{\lambda(\sigma)}\sigma\right)\right]\frac{\cos\left[\lambda(\tau)\xi^{\frac12}\int_{\sigma}^{\tau}\lambda^{-1}(u)\,du\right]}{\frac{\lambda(\tau)}{\lambda(\sigma)}\xi}\frac{\partial}{\partial{\sigma}}\mathcal{F}\big(\chi_{R\ll\tau}\tilde{N}_2(\veps)\big)\left(\sigma, \frac{\lambda^2(\tau)}{\lambda^2(\sigma)}\xi\right)\,d\sigma \tilde{\rho}(\xi)d\xi\right|,\\
\end{align*}
where we may use the representation formula 
\[
\frac{\partial}{\partial{\sigma}}\mathcal{F}\big(\chi_{R\ll\tau}\tilde{N}_2(\veps)\big)\left(\sigma, \frac{\lambda^2(\tau)}{\lambda^2(\sigma)}\xi\right) = \beta_{\nu}(\sigma)B(\sigma)\phi\left(\frac{\lambda^2(\tau)}{\lambda^2(\sigma)}\xi\right)
\]
with a rapidly decaying smooth function $\phi$ to leading order. Note that if we restrict the $\sigma$-integral to the region $|\tau-\sigma|\leq \frac{1}{C}\tau$ for some large $\tau$, then we can bound the preceding integral by 
\begin{align*}
\lesssim \int_{\tau(1-\frac1C)}^{\tau}\beta_{\nu}(\sigma)\big|B(\sigma)\big|\,d\sigma\lesssim \frac{1}{C}\cdot\alpha\frac{\lambda(\tau_0)}{\lambda(\tau)}\left\langle\log\frac{\lambda(\tau)}{\lambda(\tau_0)}\right\rangle^{-1-\frac{\kappa}{2}},
\end{align*}
which is again as desired. Restrict then the $\sigma$-integral to the region $\sigma<\left(1-\frac{1}{C}\right)\tau$, and perform integration by parts with respect to $\xi^{\frac12}$. This leads to a double integral of essentially the form (multiplied by a constant depending on $C$)
\begin{align*}
&\gamma_C\int_0^\infty\int_{\tau_0}^{\tau\cdot\left(1-\frac{1}{C}\right)}\left[1 - \tilde{\chi}\left(\xi^{\frac12}\frac{\lambda(\tau)}{\lambda(\sigma)}\sigma\right)\right]\frac{\sin\left[\lambda(\tau)\xi^{\frac12}\int_{\sigma}^{\tau}\lambda^{-1}(u)\,du\right]}{\left(\frac{\lambda(\tau)}{\lambda(\sigma)}\right)^2\sigma\cdot\xi^{\frac32}}\beta_{\nu}(\sigma)B(\sigma)\tilde{\phi}\left(\frac{\lambda^2(\tau)}{\lambda^2(\sigma)}\xi\right)\,d\sigma \tilde{\rho}(\xi)d\xi,
\end{align*}
and thanks to the weight $\tilde{\rho}$, we can bound its absolute value by 
\begin{align*}
&\lesssim \alpha\gamma_C\int_{\tau_0}^{\tau\cdot\left(1-\frac{1}{C}\right)}\frac{\lambda(\sigma)}{\lambda(\tau)}\left\langle\log\left(\frac{\lambda(\tau)}{\lambda(\sigma)}\sigma\right)\right\rangle^{-2}\frac{\lambda(\tau_0)}{\lambda(\sigma)}\left\langle\log\frac{\lambda(\sigma)}{\lambda(\tau_0)}\right\rangle^{-1-\frac{\kappa}{2}}\beta_{\nu}(\sigma)\,d\sigma\\
&\lesssim \alpha \gamma_C\langle\log\tau_0\rangle^{-1+\frac{\kappa}{2}}\cdot\frac{\lambda(\tau_0)}{\lambda(\tau)}\left\langle\log\frac{\lambda(\tau)}{\lambda(\tau_0)}\right\rangle^{-1-\frac{\kappa}{2}}.
\end{align*}
This is also as desired with $\gamma(\tau_0) \sim  \gamma_C\langle\log\tau_0\rangle^{-1+\frac{\kappa}{2}}$. Note that we can let $C$ grow to $+\infty$ with $\tau_0$. 

\end{proof}

\subsection{The effect of $c(\tau)$ on $\calD\veps$ and $c(\tau)$; another smallness gain after double iteration}

Recall that Proposition~\ref{prop:anotherresonantlin} did not furnish a bound for the terms arising upon substituting the Fourier transform of 
\[
\tilde{N}_2(\veps)(\tau, R) = \frac{\lambda'}{\lambda}\left(\partial_{\tau} + \frac{\lambda'}{\lambda}R\partial_R\right)\left(\frac{4R}{(1+R^2)^2}\veps\right),
\]
in the Duhamel parametrix, with $\veps(\tau, R) = c(\tau)\phi_0(R)$. 
Define now $\tilde{x}_{good}(\tau, \xi)$ via
\begin{align*}
\tilde{x}_{good}(\tau, \xi) = \xi^{-\frac{1}{2}}\int_{\tau_{0}}^{\tau}\sin\left(\lambda(\tau)\xi^{\frac{1}{2}}\int_{\sigma}^{\tau}\lambda(u)^{-1}du\right)\mathcal{F}\left(\chi_{R\ll\sigma}\tilde{N}_2(\veps)\right)\left(\sigma,\frac{\lambda(\tau)^{2}}{\lambda(\sigma)^{2}}\xi\right)d\sigma,
\end{align*}
with $\veps(\tau, R) = c(\tau)\phi_0(R)$ and $\tilde{N}_2(\veps)$ as defined above.  Here we show that if we go one step deeper into the iteration and take advantage of the fine structure of $c(\tau)$, a smallness gain can be obtained, at least for $\calD_{\tau}\tilde{x}_{good}$: 
\begin{lemma}\label{lem:reiterate2} Assume that $c(\tau)$ solves \eqref{ctau ODE} with trivial data and with $n = 0$ and $h$ given by 
\[
h(\tau) = \int_0^\infty \left(\int_{\tau_0}^{\tau}\frac{\sin[\lambda(\tau)\eta^{\frac12}\int_{\sigma}^{\tau}\lambda^{-1}(u)\,du]}{\eta^{\frac12}}\Psi\left(\sigma, \frac{\lambda^2(\tau)}{\lambda^2(\sigma)}\eta\right)\,d\sigma\right)\tilde{\rho}(\eta)\,d\eta,
\]
and $\Psi(\sigma, \xi) = \beta_{\nu}(\sigma)\mathcal{K}_0\mathcal{D}_{\sigma}x_{good}(\sigma,\cdot)$. Then defining $\tilde{x}_{good}$ as above, we have 
\begin{align*}
&\sup_{\tau\geq \tau_0}\frac{\lambda(\tau)}{\lambda(\tau_0)}\left\langle\log\left(\frac{\lambda(\tau)}{\lambda(\tau_0)}\right)\right\rangle^{1+\frac{\kappa}{2}}\big\|\mathcal{D}_{\tau}\tilde{x}_{good}(\tau, \cdot)\big\|_{S_1}\\
&\leq \gamma(\tau_0)\sup_{\tau\geq \tau_0}\frac{\lambda(\tau)}{\lambda(\tau_0)}\left\langle\log\left(\frac{\lambda(\tau)}{\lambda(\tau_0)}\right)\right\rangle^{1+\frac{\kappa}{2}}\big\|\mathcal{D}_{\tau}x_{good}(\tau, \cdot)\big\|_{S_1},\\
\end{align*}
where $\lim_{\tau_0\rightarrow\infty}\gamma(\tau_0) = 0$. A similar bound but without $\gamma(\tau_0)$ and $S_0$ instead of $S_1$ is obtained for $\tilde{x}_{good}(\tau, \cdot)$ instead of $\calD_{\tau}\tilde{x}_{good}(\tau, \cdot)$. Similar bounds are obtained when substituting for $\Psi$ any other source term on the right hand side of \eqref{eq on Fourier side final}, as well as when using $n$ as given after \eqref{ctau ODE} in the equation for $c(\tau)$ (then one replaces $\big\|\mathcal{D}_{\tau}x_{good}(\tau, \cdot)\big\|_{S_1}$ by $\big\|x_{good}\big\|_{S_0} + \|c\|$).  \end{lemma} 
\begin{proof} We prove the stated bound for $\big\|\mathcal{D}_{\tau}\tilde{x}_{good}(\tau, \cdot)\big\|_{S_1}$, the (weaker) bound for $\big\|\tilde{x}_{good}(\tau, \cdot)\big\|_{S_0}$ being obtained by similar arguments. We can write 
\begin{align*}
\mathcal{F}\left(\chi_{R\ll\sigma}\tilde{N}_2(\veps)\right)\left(\sigma,\frac{\lambda(\tau)^{2}}{\lambda(\sigma)^{2}}\xi\right) = \beta_{\nu}(\sigma)c'(\sigma)\phi\left(\frac{\lambda(\tau)^{2}}{\lambda(\sigma)^{2}}\xi\right) + \beta_{\nu}^2(\sigma)c(\sigma)\tilde{\phi}\left(\frac{\lambda(\tau)^{2}}{\lambda(\sigma)^{2}}\xi\right)
\end{align*}
The functions $\phi, \tilde{\phi}$ here are $C^\infty$ and decay rapidly toward $+\infty$. 
The contribution of the second term on the right is easily seen to be amenable to an argument like the one used for the proof of Proposition~\ref{prop:resonantlin}, so we focus on the contribution by the first term on the right. Applying $\mathcal{D}_{\tau}$, we have to bound 
\begin{align*}
\mathcal{D}_{\tau}\tilde{x}_{good}(\tau, \xi) = \int_{\tau_{0}}^{\tau}\cos\left(\lambda(\tau)\xi^{\frac{1}{2}}\int_{\sigma}^{\tau}\lambda(u)^{-1}du\right)\mathcal{F}\left(\chi_{R\ll\sigma}\tilde{N}_2(\veps)\right)\left(\sigma,\frac{\lambda(\tau)^{2}}{\lambda(\sigma)^{2}}\xi\right)d\sigma
\end{align*}
Arguing as in the proof of Lemma~\ref{lem:ctocdelicate}, we may include a cutoff of the form $\left[1 - \chi\left(C^{-1}\sigma\frac{\lambda(\tau)}{\lambda(\sigma)}\xi^{\frac{1}{2}}\right)\right]$. 
Integration by parts leads to leading order (up to similar terms) to the term 
\begin{align*}
\int_{\tau_{0}}^{\tau}\left[1 - \chi\left(C^{-1}\sigma\frac{\lambda(\tau)}{\lambda(\sigma)}\xi^{\frac{1}{2}}\right)\right]\frac{\sin\left(\lambda(\tau)\xi^{\frac{1}{2}}\int_{\sigma}^{\tau}\lambda(u)^{-1}du\right)}{\xi^{\frac12}\frac{\lambda(\tau)}{\lambda(\sigma)}}\beta_{\nu}(\sigma)c''(\sigma)\phi\left(\frac{\lambda(\tau)^{2}}{\lambda(\sigma)^{2}}\xi\right) d\sigma
\end{align*}
Using \eqref{eq:c'explicit} with $A = 0$, $B = h$, this reduces up to similar or better error terms to 
 \begin{align*}
\int_{\tau_{0}}^{\tau}\left[1 - \chi\left(C^{-1}\sigma\frac{\lambda(\tau)}{\lambda(\sigma)}\xi^{\frac{1}{2}}\right)\right]\frac{\sin\left(\lambda(\tau)\xi^{\frac{1}{2}}\int_{\sigma}^{\tau}\lambda(u)^{-1}du\right)}{\xi^{\frac12}\frac{\lambda(\tau)}{\lambda(\sigma)}}\beta_{\nu}(\sigma)h(\sigma)\phi\left(\frac{\lambda(\tau)^{2}}{\lambda(\sigma)^{2}}\xi\right) d\sigma,
\end{align*}
where we have 
\[
h(\sigma) = \int_0^\infty \left(\int_{\tau_0}^{\sigma}\frac{\sin[\lambda(\sigma)\eta^{\frac12}\int_{\sigma_1}^{\sigma}\lambda^{-1}(u)\,du]}{\eta^{\frac12}}\Psi\left(\sigma_1, \frac{\lambda^2(\sigma)}{\lambda^2(\sigma_1)}\eta\right)\,d\sigma_1\right)\tilde{\rho}(\eta)\,d\eta
\]
Alternatively, for the last part of the proposition, one substitutes $B = n$, obtaining 
 \begin{align*}
\int_{\tau_{0}}^{\tau}\left[1 - \chi\left(C^{-1}\sigma\frac{\lambda(\tau)}{\lambda(\sigma)}\xi^{\frac{1}{2}}\right)\right]\frac{\sin\left(\lambda(\tau)\xi^{\frac{1}{2}}\int_{\sigma}^{\tau}\lambda(u)^{-1}du\right)}{\xi^{\frac12}\frac{\lambda(\tau)}{\lambda(\sigma)}}\beta_{\nu}(\sigma)n(\sigma)\phi\left(\frac{\lambda(\tau)^{2}}{\lambda(\sigma)^{2}}\xi\right) d\sigma,
\end{align*}
Let us deal with this easier case first. Invoking Lemma~\ref{lem:nbound}, and exploiting the decay of $\phi$, we get 
\begin{align*}
&\left\|\xi^{2+\kappa}\langle\log\xi\rangle^{-1-\kappa}\int_{\tau_{0}}^{\tau}[\ldots]\frac{\sin\left(\lambda(\tau)\xi^{\frac{1}{2}}\int_{\sigma}^{\tau}\lambda(u)^{-1}du\right)}{\xi^{\frac12}\frac{\lambda(\tau)}{\lambda(\sigma)}}\beta_{\nu}(\sigma)n(\sigma)\phi\left(\frac{\lambda(\tau)^{2}}{\lambda(\sigma)^{2}}\xi\right) d\sigma\right\|_{L^2_{d\xi}}\\
&\lesssim \gamma(\tau_0)[\|x\| + \|c\|]\int_{\tau_{0}}^{\tau}\frac{\lambda(\sigma)}{\lambda(\tau)}\left\langle\log\frac{\lambda(\tau)}{\lambda(\sigma)}\right\rangle^{-1-\kappa}\cdot\frac{(\log\sigma)^{\frac12}}{\sigma^2}\cdot\beta_{\nu}(\sigma)\cdot\frac{\lambda(\tau_0)}{\lambda(\sigma)}\left\langle\log\frac{\lambda(\sigma)}{\lambda(\tau_0)}\right\rangle^{-1-\kappa}\,d\sigma\\
&\lesssim \gamma(\tau_0)[\|x\| + \|c\|]\frac{\lambda(\tau_0)}{\lambda(\tau)}\left\langle\log\frac{\lambda(\tau)}{\lambda(\tau_0)}\right\rangle^{-1-\kappa},
\end{align*}
as desired. Note that the smallness gain $\gamma(\tau_0)$ here is inherited from the corresponding $n$-bound. Also, the factor $(\log\sigma)^{\frac12}$ comes from the bound 
\begin{align*}
\left\|\left[1 - \chi\left(C^{-1}\sigma\frac{\lambda(\tau)}{\lambda(\sigma)}\xi^{\frac{1}{2}}\right)\right]\xi^{-\frac12}\phi\left(\frac{\lambda(\tau)^{2}}{\lambda(\sigma)^{2}}\xi\right)\right\|_{L^2_{d\xi}}\lesssim (\log\sigma)^{\frac12}. 
\end{align*}

We have to work harder when $B = h$, given by the expression displayed before. 
Passing to the variable $\tilde{\eta} = \frac{\lambda^2(\sigma)}{\lambda^2(\sigma_1)}\eta$, and combining the oscillatory terms, we get the phases 
\[
\frac{\lambda(\tau)}{\lambda(\sigma)}\xi^{\frac{1}{2}}\pm \frac{\lambda(\sigma_1)}{\lambda(\sigma)}\tilde{\eta}^{\frac12}.
\]
Restricting to the non-resonant region 
\[
\left|\frac{\lambda(\tau)}{\lambda(\sigma)}\xi^{\frac{1}{2}}\pm \frac{\lambda(\sigma_1)}{\lambda(\sigma)}\tilde{\eta}^{\frac12}\right|\gtrsim \frac{\lambda(\tau)}{\lambda(\sigma)}\xi^{\frac{1}{2}}
\]
and calling the suitably modified $h$ now $\tilde{h}(\sigma)$, we perform integration by parts with respect to $\sigma$. Use that on account of Proposition~\ref{prop:nonlocallinear1} and its proof we have 
\begin{align*}
&\sup_{\sigma\geq \tau_0}\frac{\lambda(\sigma)}{\lambda(\tau_0)}\left\langle\log\left(\frac{\lambda(\sigma)}{\lambda(\tau_0)}\right)\right\rangle^{1+\frac{\kappa}{2}}\left|\int_0^\infty \left(\int_{\tau_0}^{\sigma}\frac{\sin[\lambda(\sigma)\eta^{\frac12}\int_{\sigma_1}^{\sigma}\lambda^{-1}(u)\,du]}{\eta^{\frac12}}\Psi\left(\sigma_1, \frac{\lambda^2(\sigma)}{\lambda^2(\sigma_1)}\eta\right)\,d\sigma_1\right)\tilde{\rho}(\eta)\,d\eta\right|\\
& +\sup_{\sigma\geq \tau_0}\frac{\lambda(\sigma)}{\lambda(\tau_0)}\left\langle\log\left(\frac{\lambda(\sigma)}{\lambda(\tau_0)}\right)\right\rangle^{1+\frac{\kappa}{2}}\left|\int_0^\infty \left(\int_{\tau_0}^{\sigma}\frac{\cos[\lambda(\sigma)\eta^{\frac12}\int_{\sigma_1}^{\sigma}\lambda^{-1}(u)\,du]}{\eta^{\frac12}}\Psi\left(\sigma_1, \frac{\lambda^2(\sigma)}{\lambda^2(\sigma_1)}\eta\right)\,d\sigma_1\right)\tilde{\rho}(\eta)\,d\eta\right|\\
&\lesssim \sup_{\tau\geq \tau_0}\frac{\lambda(\tau)}{\lambda(\tau_0)}\left\langle\log\left(\frac{\lambda(\tau)}{\lambda(\tau_0)}\right)\right\rangle^{1+\frac{\kappa}{2}}\big\|\mathcal{D}_{\tau}x_{good}(\tau, \cdot)\big\|_{S_1}.
\end{align*}
We then find that 
\begin{align*}
&\left\|\int_{\tau_{0}}^{\tau}\left[1 - \chi\left(C^{-1}\sigma\frac{\lambda(\tau)}{\lambda(\sigma)}\xi^{\frac{1}{2}}\right)\right]\frac{\sin\left(\lambda(\tau)\xi^{\frac{1}{2}}\int_{\sigma}^{\tau}\lambda(u)^{-1}du\right)}{\xi^{\frac12}\frac{\lambda(\tau)}{\lambda(\sigma)}}\beta_{\nu}(\sigma)\tilde{h}(\sigma)\phi\left(\frac{\lambda(\tau)^{2}}{\lambda(\sigma)^{2}}\xi\right) d\sigma\right\|_{S_1}\\
&\lesssim \int_{\tau_{0}}^{\tau}\left\langle\log\left(C^{-1}\sigma\frac{\lambda(\tau)}{\lambda(\sigma)}\right)\right\rangle^{-1-\kappa}\frac{\lambda(\tau_0)}{\lambda(\sigma)}\left\langle\log\left(\frac{\lambda(\sigma)}{\lambda(\tau_0)}\right)\right\rangle^{-1-\frac{\kappa}{2}}\,d\sigma\\
&\hspace{4cm}\cdot \sup_{\tau\geq \tau_0}\frac{\lambda(\tau)}{\lambda(\tau_0)}\left\langle\log\left(\frac{\lambda(\tau)}{\lambda(\tau_0)}\right)\right\rangle^{1+\frac{\kappa}{2}}\big\|\mathcal{D}_{\tau}x_{good}(\tau, \cdot)\big\|_{S_1},
\end{align*}
and this gives the desired bound with $\gamma(\tau_0) = \langle\log\tau_0\rangle^{-\frac{\kappa}{2}}$. 
\\
It suffices henceforth to restrict to the situation 
\[
\left|\frac{\lambda(\tau)}{\lambda(\sigma)}\xi^{\frac{1}{2}}\pm \frac{\lambda(\sigma_1)}{\lambda(\sigma)}\tilde{\eta}^{\frac12}\right|\lesssim\frac{\lambda(\tau)}{\lambda(\sigma)}\xi^{\frac{1}{2}}
\]
Denoting by $I$ dyadic intervals, it suffices to bound 
\begin{align*}
\left\|\sum_{I}\chi_I\left(\xi^{\frac12}\frac{\lambda(\tau)}{\lambda(\sigma)}\right)\int_{\tau_{0}}^{\tau}\left[1 - \chi\left(C^{-1}\sigma\frac{\lambda(\tau)}{\lambda(\sigma)}\xi^{\frac{1}{2}}\right)\right]\frac{\sin\left(\lambda(\tau)\xi^{\frac{1}{2}}\int_{\sigma}^{\tau}\lambda(u)^{-1}du\right)}{\xi^{\frac12}\frac{\lambda(\tau)}{\lambda(\sigma)}}\beta_{\nu}(\sigma)h^I(\sigma)\phi\left(\frac{\lambda(\tau)^{2}}{\lambda(\sigma)^{2}}\xi\right) d\sigma\right\|_{S_1}, 
\end{align*}
where we let 
\begin{align*}
h^{I}(\sigma)=\int_0^\infty \chi_I(\eta)\left(\int_{\tau_0}^{\sigma}\frac{\sin[\lambda(\sigma)\eta^{\frac12}\int_{\sigma_1}^{\sigma}\lambda^{-1}(u)\,du]}{\eta^{\frac12}}\Psi\left(\sigma_1, \frac{\lambda^2(\sigma)}{\lambda^2(\sigma_1)}\eta\right)\,d\sigma_1\right)\tilde{\rho}(\eta)\,d\eta.
\end{align*}
Proposition~\ref{prop:nonlocallinear1} implies that 
\[
\big|h^I(\sigma)\big|\lesssim \langle \log I\rangle^{-1+\kappa}\frac{\lambda(\tau_0)}{\lambda(\sigma)}\left\langle\log\left(\frac{\lambda(\sigma)}{\lambda(\tau_0)}\right)\right\rangle^{-1-\frac{\kappa}{2}}\cdot \sup_{\tau\geq \tau_0}\frac{\lambda(\tau)}{\lambda(\tau_0)}\left\langle\log\left(\frac{\lambda(\tau)}{\lambda(\tau_0)}\right)\right\rangle^{1+\frac{\kappa}{2}}\big\|\mathcal{D}_{\tau}x_{good}(\tau, \cdot)\big\|_{S_1}
\]
By orthogonality and Minkowski's inequality, we then infer that the $\|\cdot\|_{S_1}$-norm above is bounded by 
\begin{align*}
&\Big\|\sum_{I}\ldots\Big\|_{S_1}\lesssim \int_{\tau_{0}}^{\tau}\frac{\lambda(\sigma)}{\lambda(\tau)}\left\langle\log\frac{\lambda(\tau)}{\lambda(\sigma)}\right\rangle^{-1-\kappa}\left(\sum_I\langle \log I\rangle^{-2+2\kappa}\right)^{\frac12}\frac{\lambda(\tau_0)}{\lambda(\sigma)}\left\langle\log\left(\frac{\lambda(\sigma)}{\lambda(\tau_0)}\right)\right\rangle^{-1-\frac{\kappa}{2}}\cdot \sup_{\tau\geq \tau_0}\ldots
\end{align*}
The fact that $\kappa<\frac12$ implies that the inner sum $\sum_I\langle \log I\rangle^{-2+2\kappa}$ converges, and so one recovers the desired bound, except that there is no smallness gain yet. However, such a gain is easily obtained by restricting the frequency $\xi$ to the ranges $\xi<\tau_0^{-\delta_1}, \xi>\tau_0^{\delta_1}$ for some $\delta_1>0$, say. 
\\

Thus it suffices to obtain an improved bound for 
\begin{align*}
&\frac{\lambda(\tau)}{\lambda(\tau_0)}\left\langle\log\frac{\lambda(\tau)}{\lambda(\tau_0)}\right\rangle^{1+\frac{\kappa}{2}}\left\|\int_{\tau_{0}}^{\tau}\frac{\sin\left(\lambda(\tau)\xi^{\frac{1}{2}}\int_{\sigma}^{\tau}\lambda(u)^{-1}du\right)}{\xi^{\frac12}\frac{\lambda(\tau)}{\lambda(\sigma)}}\beta_{\nu}(\sigma)h(\sigma)\phi\left(\frac{\lambda(\tau)^{2}}{\lambda(\sigma)^{2}}\xi\right) d\sigma\right\|_{S_1\left(\tau_0^{-\delta_1}<\cdot<\tau_0^{\delta_1}\right)}\\
\end{align*}
Again we exploit the oscillatory character of the integrand with respect to $\sigma$ to force frequency alignment. Specifically, changing the variable in the $\eta$-integral to $\tilde{\eta} = \frac{\lambda^2(\sigma)}{\lambda^2(\sigma_1)}\eta$, we arrive at the oscillatory kernel 
\[
\sin\left(\lambda(\tau)\xi^{\frac{1}{2}}\int_{\sigma}^{\tau}\lambda(u)^{-1}du\right)\cdot \sin\left[\lambda(\sigma_1)\tilde{\eta}^{\frac12}\int_{\sigma_1}^{\sigma}\lambda^{-1}(u)\,du\right],
\]
and so as long as 
\[
\left|\frac{\lambda(\tau)}{\lambda(\sigma)}\xi^{\frac{1}{2}} - \frac{\lambda(\sigma_1)}{\lambda(\sigma)}\tilde{\eta}^{\frac12}\right|\gtrsim \tau_0^{-\gamma}
\]
for some small $\gamma>0$, we gain a negative power of $\tau_0$ by performing integration by parts with respect to $\sigma$. Thus we can reduce to the situation 
\[
\left|\frac{\lambda(\tau)}{\lambda(\sigma)}\xi^{\frac{1}{2}} - \eta^{\frac12}\right|\lesssim \tau_0^{-\gamma},\quad \xi\in \{\tau_0^{-\delta_1}<\cdot<\tau_0^{\delta_1}\},
\]
where we may assume $\delta_1\ll\gamma$. Now cover $\frac{\lambda^2(\tau)}{\lambda^2(\sigma)}\cdot \{\tau_0^{-\delta_1}<\cdot<\tau_0^{\delta_1}\}$ with intervals $I$ of length $\sim \tau_0^{-\gamma}$, and consider 
\begin{align*}
\left\|\sum_{I}\int_{\tau_{0}}^{\tau}\chi_{I}\left(\frac{\lambda^2(\tau)}{\lambda^2(\sigma)}\xi\right)\frac{\sin\left(\lambda(\tau)\xi^{\frac{1}{2}}\int_{\sigma}^{\tau}\lambda(u)^{-1}du\right)}{\xi^{\frac12}\frac{\lambda(\tau)}{\lambda(\sigma)}}\beta_{\nu}(\sigma)h_I(\sigma)\phi\left(\frac{\lambda(\tau)^{2}}{\lambda(\sigma)^{2}}\xi\right) d\sigma\right\|_{S_1},
\end{align*}
where (with a suitable dilate $CI$ of $I$)
\begin{align*}
h_{I}(\sigma) = \int_0^\infty\chi_{CI}(\eta) \left(\int_{\tau_0}^{\sigma}\frac{\sin[\lambda(\sigma)\eta^{\frac12}\int_{\sigma_1}^{\sigma}\lambda^{-1}(u)\,du]}{\eta^{\frac12}}\Psi\left(\sigma_1, \frac{\lambda^2(\sigma)}{\lambda^2(\sigma_1)}\eta\right)\,d\sigma_1\right)\tilde{\rho}(\eta)\,d\eta
\end{align*}
Using Holder's inequality as well as Proposition~\ref{prop:nonlocallinear1} we infer 
\begin{align*}
\frac{\lambda(\sigma)}{\lambda(\tau_0)}\left\langle\log\frac{\lambda(\sigma)}{\lambda(\tau_0)}\right\rangle^{1+\frac{\kappa}{2}}\big|h_{I}(\sigma)\big|\lesssim \big|I\big|^{\frac12}\tau_0^{\frac{\delta_1}{2}}\sup_{\tau\geq \tau_0}\frac{\lambda(\tau)}{\lambda(\tau_0)}\left\langle\log\left(\frac{\lambda(\tau)}{\lambda(\tau_0)}\right)\right\rangle^{1+\frac{\kappa}{2}}\big\|\mathcal{D}_{\tau}x_{good}(\tau, \cdot)\big\|_{S_1}
\end{align*}
and again from Holder's inequality we infer 
\begin{align*}
&\left\|\int_{\tau_{0}}^{\tau}\chi_{I}\left(\frac{\lambda^2(\tau)}{\lambda^2(\sigma)}\xi\right)\frac{\sin\left(\lambda(\tau)\xi^{\frac{1}{2}}\int_{\sigma}^{\tau}\lambda(u)^{-1}du\right)}{\xi^{\frac12}\frac{\lambda(\tau)}{\lambda(\sigma)}}\beta_{\nu}(\sigma)h_I(\sigma)\phi\left(\frac{\lambda(\tau)^{2}}{\lambda(\sigma)^{2}}\xi\right) d\sigma\right\|_{S_1}\\
&\lesssim \tau_0^{C\delta_1}\big|I\big|^{\frac12}\cdot \big|I\big|^{\frac12}\tau_0^{\frac{\delta_1}{2}}\sup_{\tau\geq \tau_0}\frac{\lambda(\tau)}{\lambda(\tau_0)}\left\langle\log\left(\frac{\lambda(\tau)}{\lambda(\tau_0)}\right)\right\rangle^{1+\frac{\kappa}{2}}\big\|\mathcal{D}_{\tau}x_{good}(\tau, \cdot)\big\|_{S_1}
\end{align*}
Then we conclude by observing that 
\begin{align*}
&\left\|\sum_{I}\int_{\tau_{0}}^{\tau}\chi_{I}\left(\frac{\lambda^2(\tau)}{\lambda^2(\sigma)}\xi\right)\frac{\sin\left(\lambda(\tau)\xi^{\frac{1}{2}}\int_{\sigma}^{\tau}\lambda(u)^{-1}du\right)}{\xi^{\frac12}\frac{\lambda(\tau)}{\lambda(\sigma)}}\beta_{\nu}(\sigma)h_I(\sigma)\phi\left(\frac{\lambda(\tau)^{2}}{\lambda(\sigma)^{2}}\xi\right) d\sigma\right\|_{S_1}\\
&\leq \left(\sum_{I}\left\|\int_{\tau_{0}}^{\tau}\chi_{I}\left(\frac{\lambda^2(\tau)}{\lambda^2(\sigma)}\xi\right)\frac{\sin\left(\lambda(\tau)\xi^{\frac{1}{2}}\int_{\sigma}^{\tau}\lambda(u)^{-1}du\right)}{\xi^{\frac12}\frac{\lambda(\tau)}{\lambda(\sigma)}}\beta_{\nu}(\sigma)h_I(\sigma)\phi\left(\frac{\lambda(\tau)^{2}}{\lambda(\sigma)^{2}}\xi\right) d\sigma\right\|_{S_1}^2\right)^{\frac12}\\
&\lesssim  \big|I\big|^{\frac12}\tau_0^{C_1\delta_1}\sup_{\tau\geq \tau_0}\frac{\lambda(\tau)}{\lambda(\tau_0)}\left\langle\log\left(\frac{\lambda(\tau)}{\lambda(\tau_0)}\right)\right\rangle^{1+\frac{\kappa}{2}}\big\|\mathcal{D}_{\tau}x_{good}(\tau, \cdot)\big\|_{S_1}
\end{align*}
which gives the first assertion of the lemma in light of $\big|I\big|\sim \tau_0^{-\gamma}$, $\gamma\gg \delta_1>0$. 
\end{proof}
\section{Forcing smallness for the non-local linear terms via re-iteration}

At this stage, we have essentially all ingredients to set up an iterative scheme yielding a solution $\big(c(\tau), \xb(\tau, \xi)\big)$ for the system consisting of \eqref{ctau ODE}, \eqref{eq on Fourier side final}, except that there is still one delicate situation where an iterative step will not yield smallness, specifically due to the terms involving $\mathcal{D}_{\tau}$ in Proposition~\ref{prop:nonlocallinear2}, except upon inclusion of additional frequency cutoffs $\chi_{\xi<\epsilon}$, $\chi_{\xi>\epsilon^{-1}}$. In order to deal with this, we implement essentially the method from \cite{CondBlow} (in fact, a simplified method), which shows that iterating application of the parametrix to the `bad' nonlocal linear terms sufficiently many times does result in smallness, provided $\tau_0$ is chosen sufficiently large. The key step here is a reduction to a `diagonal' situation in a two-fold application of the iterative step, as seen in the following. 
\\

Introduce the operator 
\begin{equation}\label{eq:Phidef}
\Phi(f)(\tau, \xi): = \int_{\tau_0}^{\tau}\cos\left[\lambda(\tau)\xi^{\frac12}\int_{\sigma}^{\tau}\lambda^{-1}(u)\,du\right]\beta_{\nu}(\sigma)\mathcal{K}_0f\left(\sigma, \frac{\lambda^2(\tau)}{\lambda^2(\sigma)}\xi\right)\,d\sigma,\,\beta_{\nu}(\sigma) = \frac{\lambda'(\sigma)}{\lambda(\sigma)}.
\end{equation}
We shall consider iterates of this operator. To control them, a splitting of the kernel of $\mathcal{K}_0$ into a 'diagonal' and a 'non-diagonal' part are essential. Labelling this kernel $K_0(\xi, \eta)$, we write 
\begin{align*}
K_0(\xi, \eta) &= \chi_{\left|\frac{\xi}{\eta} - 1\right|<\frac1n}K_0(\xi, \eta) + \chi_{\left|\frac{\xi}{\eta} - 1\right|\geq \frac1n}K_0(\xi, \eta)\\
&=:K_0^d(\xi, \eta) + K_0^{nd}(\xi, \eta),
\end{align*}
where the cutoffs are smooth and localize to dilates of the indicated regions. Call the corresponding operators $\mathcal{K}_0^d, \mathcal{K}_0^{nd}$. Then we have the important 
\begin{lemma}\label{lem:ndsmall} We have the bound
\[
\sup_{\tau\geq \tau_0}\frac{\lambda(\tau)}{\lambda(\tau_0)}\left\langle\log\frac{\lambda(\tau)}{\lambda(\tau_0)}\right\rangle^{1+\frac{\kappa}{2}}\big\|\Phi\big(\beta_{\nu}(\tau)\mathcal{K}_0^{nd}\Phi(f)\big\|_{S_1}
\lesssim \tau_0^{-\gamma}\sup_{\tau\geq \tau_0}\frac{\lambda(\tau)}{\lambda(\tau_0)}\left\langle\log\frac{\lambda(\tau)}{\lambda(\tau_0)}\right\rangle^{1+\frac{\kappa}{2}}\big\|f(\tau, \cdot)\big\|_{S_1}
\]
for a suitable universal $\gamma>0$.
\end{lemma}
This is proved exactly as in Lemma 12.0.32 in \cite{CondBlow}, using integration by parts in the `intermediate' time variable. 
\\

The preceding lemma then reduces things to controlling the contribution of the diagonal operator $\mathcal{K}_0^d$, and it is here where we take decisive advantage of a large number of re-iterations: 
As in \cite{CondBlow}, introduce the operator 
\[
\mathcal{D}_{\tau}Uh(\tau, \xi): = \int_{\tau_0}^{\tau}\cos\left[\lambda(\tau)\xi^{\frac12}\int_{\tau}^{\sigma}\lambda^{-1}(u)\,du\right]h\left(\sigma, \frac{\lambda^2(\tau)}{\lambda^2(\sigma)}\xi\right)\,d\sigma. 
\]
Also, introduce the diagonal part of $\Phi$: 
\[
\Phi_1(f)(\tau, \xi):= \mathcal{D}_{\tau}U\big(\beta_{\nu}(\sigma)\mathcal{K}_0^df\big)(\tau, \xi).
\]
Then we have the key 
\begin{proposition}\label{prop:Kditeratedsmallness} For $\epsilon>0$ small enough, we have for any$ k\geq 1$ the bound 
\begin{align*}
\sup_{\tau\geq \tau_0}\frac{\lambda(\tau)}{\lambda(\tau_0)}\left\langle\log\frac{\lambda(\tau)}{\lambda(\tau_0)}\right\rangle^{1+\frac{\kappa}{2}}\big\|\Phi_1^k f(\tau, \cdot)\big\|_{S_1}\leq \frac{1}{|\log\epsilon|^{\gamma(\kappa)}}e^{\epsilon^{-2}}\sup_{\tau\geq \tau_0}\frac{\lambda(\tau)}{\lambda(\tau_0)}\left\langle\log\frac{\lambda(\tau)}{\lambda(\tau_0)}\right\rangle^{1+\frac{\kappa}{2}}\big\|f(\tau, \cdot)\big\|_{S_1}
\end{align*}
for suitable $\gamma(\kappa)>0$. 
\end{proposition}
\begin{proof} This is also essentially identical to arguments given in \cite{CondBlow}, except that the norms need to be adjusted. To begin with, we decompose the operator $\mathcal{K}_0^d$ into two parts restricting to very small or very large frequencies, as well as a part restricting to intermediate frequencies. Thus write 
\[
\mathcal{K}_0^d = \mathcal{K}_{0,1}^{d,\epsilon} +  \mathcal{K}_{0,2}^{d,\epsilon} +  \mathcal{K}_{0,3}^{d,\epsilon}. 
\]
Here the operators on the right are given in terms of their respective kernels as follows: 
\[
K_{0,1}^{d,\epsilon} = \chi_{\xi<\epsilon}K_0^d,\,K_{0,3}^{d,\epsilon} = \chi_{\xi>\epsilon^{-1}}K_0^d,\,K_{0,2}^{d,\epsilon} = K_0^d - K_{0,1}^{d,\epsilon}  - K_{0,3}^{d,\epsilon}. 
\]
Then in analogy to Proposition~\ref{prop:nonlocallinear2}, we infer 
\begin{align*}
&\sup_{\tau\geq \tau_0}\frac{\lambda(\tau)}{\lambda(\tau_0)}\left\langle \log\frac{\lambda(\tau)}{\lambda(\tau_0)}\right\rangle^{1+\frac{\kappa}{2}}\big\|\mathcal{D}_{\tau}U\beta_{\nu}(\sigma)\mathcal{K}_{1,3}^{d, \epsilon}f\big\|_{S_1}\\&\lesssim \frac{1}{|\log\epsilon|^{\gamma(\kappa)}}\sup_{\tau\geq \tau_0}\frac{\lambda(\tau)}{\lambda(\tau_0)}\left\langle \log\frac{\lambda(\tau)}{\lambda(\tau_0)}\right\rangle^{1+\frac{\kappa}{2}}\big\|f(\tau, \cdot)\big\|_{S_1}
\end{align*}
which gives a  smallness gain for these contributions for sufficiently small $\epsilon>0$. 
Moreover, in light of the definition of $\mathcal{D}_{\tau}U$, we have the following key vanishing relations:
\begin{equation}\label{eq:van1}
\mathcal{K}_{0,3}^{d,\epsilon}\mathcal{D}_{\tau}U\mathcal{K}_{0,2}^{d, (1+\frac{1}{n})\epsilon} = 0
\end{equation}
\begin{equation}\label{eq:van2}
\mathcal{K}_{0,2}^{d,(1+\frac{1}{n})\epsilon}\mathcal{D}_{\tau}U\mathcal{K}_{0,1}^{d, \epsilon} = 0
\end{equation}
This implies that operators of type $\mathcal{K}_{0,2}^{d,\epsilon}$, $\mathcal{K}_{0,3}^{d,\epsilon}$ can only be followed by a more restrictive class of operators, and in particular, we can 'lock in' a certain amount of gain in the presence of 'off diagonal' operators. However, iterating a large number of diagonal operators will result in smallness thanks to the fact that one essentially integrates over a simplex in high dimension. 
Specifically, we expand 
\begin{align*}
&\big(\beta_{\nu}(\tau)\mathcal{K}_0^{d}\mathcal{D}_{\tau}U\big)^n\\
& = \big(\beta_{\nu}(\tau)\mathcal{K}_{0,2}^{d, \epsilon}\mathcal{D}_{\tau}U\big)^n\\
& + \sum_{k=1}^{n-1} \big(\beta_{\nu}(\tau)\mathcal{K}_{0,2}^{d, \epsilon}\mathcal{D}_{\tau}U\big)^k \big(\beta_{\nu}(\tau)\mathcal{K}_{0,1}^{d, \epsilon}\mathcal{D}_{\tau}U\big)\big(\beta_{\nu}(\tau)\mathcal{K}_0^{d}\mathcal{D}_{\tau}U\big)^{n-k-1}\\
& + \sum_{k=1}^{n-1} \big(\beta_{\nu}(\tau)\mathcal{K}_{0,2}^{d, \epsilon}\mathcal{D}_{\tau}U\big)^k \big(\beta_{\nu}(\tau)\mathcal{K}_{0,3}^{d, \epsilon}\mathcal{D}_{\tau}U\big)\big(\beta_{\nu}(\tau)\mathcal{K}_0^{d}\mathcal{D}_{\tau}U\big)^{n-k-1}\\
& + \big(\beta_{\nu}(\tau)\mathcal{K}_{0,1}^{d, \epsilon}\mathcal{D}_{\tau}U\big)\big(\beta_{\nu}(\tau)\mathcal{K}_0^{d}\mathcal{D}_{\tau}U\big)^{n-1}\\
&+\big(\beta_{\nu}(\tau)\mathcal{K}_{0,3}^{d, \epsilon}\mathcal{D}_{\tau}U\big)\big(\beta_{\nu}(\tau)\mathcal{K}_0^{d}\mathcal{D}_{\tau}U\big)^{n-1}\\
&=: A+B+C+D+E.
\end{align*}
For the term $C$, observe that we have 
\begin{align*}
&\big(\beta_{\nu}(\tau)\mathcal{K}_{0,3}^{d, \epsilon}\mathcal{D}_{\tau}U\big)\big(\beta_{\nu}(\tau)\mathcal{K}_0^{d}\mathcal{D}_{\tau}U\big)^{n-k-1}\\& = \big(\beta_{\nu}(\tau)\mathcal{K}_{0,3}^{d, \epsilon}\mathcal{D}_{\tau}U\big)\big(\beta_{\nu}(\tau)\mathcal{K}_{0,3}^{d, 4\epsilon}\mathcal{D}_{\tau}U\big)^{n-k-1},
\end{align*}
and so all terms here are trapped in a high-frequency regime. Thus we get 
\begin{equation}\label{eq:Cequation}\begin{split}
& C = \sum_{k=1}^{n-1} \big(\beta_{\nu}(\tau)\mathcal{K}_{0,2}^{d, \epsilon}\mathcal{D}_{\tau}U\big)^k \big(\beta_{\nu}(\tau)\mathcal{K}_{0,3}^{d, \epsilon}\mathcal{D}_{\tau}U\big)\big(\beta_{\nu}(\tau)\mathcal{K}_0^{d}\mathcal{D}_{\tau}U\big)^{n-k-1}\\
& = \sum_{k=1}^{n-1} \big(\beta_{\nu}(\tau)\mathcal{K}_{0,2}^{d, \epsilon}\mathcal{D}_{\tau}U\big)^k\big(\beta_{\nu}(\tau)\mathcal{K}_{0,3}^{d, \epsilon}\mathcal{D}_{\tau}U\big)\big(\beta_{\nu}(\tau)\mathcal{K}_{0,3}^{d, 4\epsilon}\mathcal{D}_{\tau}U\big)^{n-k-1}
\end{split}\end{equation}
Moreover, for the term $B$, we have 
\begin{equation}\label{eq:Bequation}\begin{split}
&B = \sum_{k=1}^{n-1} \big(\beta_{\nu}(\tau)\mathcal{K}_{0,2}^{d, \epsilon}\mathcal{D}_{\tau}U\big)^k \big(\beta_{\nu}(\tau)\mathcal{K}_{0,1}^{d, \epsilon}\mathcal{D}_{\tau}U\big)\big(\beta_{\nu}(\tau)\mathcal{K}_0^{d}\mathcal{D}_{\tau}U\big)^{n-k-1}\\
& = \sum_{k=1}^{n-1} \big(\beta_{\nu}(\tau)\mathcal{K}_{0,2}^{d, \epsilon}\mathcal{D}_{\tau}U\big)^k \big(\beta_{\nu}(\tau)\mathcal{K}_{0,1}^{d, \epsilon}\mathcal{D}_{\tau}U\big)\big(\beta_{\nu}(\tau)\mathcal{K}_{0,2}^{d, \frac{\epsilon}{4}}\mathcal{D}_{\tau}U\big)^{n-k-1}\\
&+\sum_{\substack{1\leq k\leq n-1\\ j\leq n-k-2}} \big(\beta_{\nu}(\tau)\mathcal{K}_{0,2}^{d, \epsilon}\mathcal{D}_{\tau}U\big)^k \big(\beta_{\nu}(\tau)\mathcal{K}_{0,1}^{d, \epsilon}\mathcal{D}_{\tau}U\big)\big(\beta_{\nu}(\tau)\mathcal{K}_{0,2}^{d, \frac{\epsilon}{4}}\mathcal{D}_{\tau}U\big)^{j}\\
&\hspace{4cm}\cdot\big(\beta_{\nu}(\tau)\mathcal{K}_{0,3}^{d, \frac{\epsilon}{4}}\mathcal{D}_{\tau}U\big)\big(\beta_{\nu}(\tau)\mathcal{K}_{0,3}^{d, \epsilon}\mathcal{D}_{\tau}U\big)^{n-j-k-2}
\end{split}\end{equation}
Here we use that the operators $ \big(\beta_{\nu}(\tau)\mathcal{K}_{0,2}^{d, \epsilon}\mathcal{D}_{\tau}U\big)$ on the left force large frequencies at the end of the expression, and if only one operator $\big(\beta_{\nu}(\tau)\mathcal{K}_{0,3}^{d, \frac{\epsilon}{4}}$ occurs it will force very large frequencies after it. 
\\
For term $E$ we proceed just as for term $C$. Thus write 
\begin{align*}
\big(\beta_{\nu}(\tau)\mathcal{K}_{0,3}^{d, \epsilon}\mathcal{D}_{\tau}U\big)\big(\beta_{\nu}(\tau)\mathcal{K}_0^{d}\mathcal{D}_{\tau}U\big)^{n-1} = \big(\beta_{\nu}(\tau)\mathcal{K}_{0,3}^{d, \epsilon}\mathcal{D}_{\tau}U\big)\big(\beta_{\nu}(\tau)\mathcal{K}_{0,3}^{d, 4\epsilon}\mathcal{D}_{\tau}U\big)^{n-1}
\end{align*}
The conclusion is that for terms $A, B, C$ and $E$ we can write them in terms of a few consecutive strings of operators of type $\big(\beta_{\nu}(\tau)\mathcal{K}_{0,2}^{d, \epsilon}\mathcal{D}_{\tau}U\big)$, $\big(\beta_{\nu}(\tau)\mathcal{K}_{0,3}^{d, \epsilon}\mathcal{D}_{\tau}U\big)$, and for the latter we already have observed a smallness gain. 
Finally, for the remaining term $D$, we also write it in terms of a small number of consecutive strings, by writing 
\begin{align*}
&D =  \big(\beta_{\nu}(\tau)\mathcal{K}_{0,1}^{d, \epsilon}\mathcal{D}_{\tau}U\big)\big(\beta_{\nu}(\tau)\mathcal{K}_0^{d}\mathcal{D}_{\tau}U\big)^{n-1}\\
&=\sum_{j=1}^n\big(\beta_{\nu}(\tau)\mathcal{K}_{0,1}^{d, \epsilon}\mathcal{D}_{\tau}U\big)^j[A^{(n-j)}+B^{(n-j)} + C^{(n-j)} + E^{(n-j)}]
\end{align*}
where the superscript indicates that these terms are defined just as in $A$, $B$, $C$ and $E$ but with $n$ replaced by $n-j$. 
\\

At this point, we have essentially reduced the problem of bounding $\big(\beta_{\nu}(\tau)\mathcal{K}_0^{d}\mathcal{D}_{\tau}U\big)^n$ to the problem of bounding $\big(\beta_{\nu}(\tau)\mathcal{K}_{0,2}^{d, \epsilon}\mathcal{D}_{\tau}U\big)^n$, and so this is what we now turn to: 
\begin{lemma}\label{lem:simplex} Using the preceding notation and assuming $j\leq n$ (the latter as in the definition of $\mathcal{K}_0^{d, nd}$), we have the bound 
\begin{align*}
&\sup_{\tau\geq\tau_0}\frac{\lambda(\tau)}{\lambda(\tau_0)}\left\langle \log\frac{\lambda(\tau)}{\lambda(\tau_0)}\right\rangle^{1+\frac{\kappa}{2}}\big\|\mathcal{D}_{\tau}U\big(\beta_{\nu}(\tau)\mathcal{K}_{0,2}^{d, \epsilon}\mathcal{D}_{\tau}U\big)^j \beta_{\nu}(\tau)f\big\|_{S_1}\\&\lesssim \frac{\epsilon^{-j}}{j!} \sup_{\tau\geq\tau_0}\frac{\lambda(\tau)}{\lambda(\tau_0)}\left\langle \log\frac{\lambda(\tau)}{\lambda(\tau_0)}\right\rangle^{1+\frac{\kappa}{2}}\big\|f(\tau, \cdot)\big\|_{S_1}
\end{align*}

\end{lemma}
\begin{proof} Write 
\begin{align*}
&\big(\beta_{\nu}(\tau)\mathcal{K}_{0,2}^{d, \epsilon}\mathcal{D}_{\tau}U\big)^j(\beta_{\nu}(\tau)\mathcal{K}_{0,2}^{d, \epsilon}f)\\& = \beta_{\nu}(\tau)\mathcal{K}_{0,2}^{d, \epsilon}\int_{\tau_0}^{\tau}\int_0^\infty d\sigma_1d\eta_1\beta_{\nu}(\sigma_1)\cos\left(\lambda(\tau)\xi^{\frac{1}{2}}\int_{\tau}^{\sigma_1}\lambda^{-1}(u)\,du\right)\\
&\cdot K_{0,2}^{d, \epsilon}\left(\frac{\lambda^{2}(\tau)}{\lambda^{2}(\sigma_1)}\xi, \eta_1\right)\int_{\tau_0}^{\sigma_1}\int_0^\infty d\sigma_2d\eta_2\beta_{\nu}(\sigma_2)\cos\left(\lambda(\sigma_1)\eta_1^{\frac{1}{2}}\int_{\sigma_1}^{\sigma_2}\lambda^{-1}(u)\,du\right)\\
&\ldots\\
&\cdot K_{0,2}^{d, \epsilon}\left(\frac{\lambda^{2}(\sigma_{j-2})}{\lambda^{2}(\sigma_{j-1})}\eta_{j-2}, \eta_{j-1}\right)\int_{\tau_0}^{\sigma_{j-1}}d\sigma_j\beta_{\nu}(\sigma_j)\\&\hspace{2cm}\cdot\cos\left(\lambda(\sigma_{j-1})\eta_{j-1}^{\frac{1}{2}}\int_{\sigma_{j-1}}^{\sigma_j}\lambda^{-1}(u)\,du\right)((\beta_{\nu}(\cdot)\mathcal{K}_{0,2}^{d, \epsilon}f)\left(\sigma_j, \frac{\lambda^2(\sigma_{j-1})}{\lambda^2(\sigma_j)}\eta_{j-1}\right)\\
\end{align*}
Then we carefully recall that by choice of $\mathcal{K}_0^d$ we have 
\[
\left|\frac{\lambda^2(\sigma_k)\eta_k}{\lambda^2(\sigma_{k+1})\eta_{k+1}} - 1\right|<\frac{1}{n}, \,1\leq k\leq j-1<n. 
\]
and so 
\[
\lambda^2(\sigma_k)\eta_k\geq \left(1-\frac{1}{n}\right)^k\lambda^2(\tau)\xi,\,1\leq k\leq j-1. 
\]
Since we further have the restrictions $\xi\gtrsim\epsilon, \eta_k<\epsilon^{-1}$ on the support of the full expression, we get 
\[
\sigma_k>\tau\cdot\epsilon,\,1\leq k\leq j. 
\]
for $\nu$ and the  $\epsilon$ small enough, uniformly in $n$.  In particular, we get 
\[
\beta_{\nu}(\sigma_k)\lesssim \epsilon^{-1}\tau^{-1}, 1\leq k\leq j.
\]
Finally, we infer that for fixed $\tau\geq \tau_0$ 
\begin{align*}
&\big\|\big(\beta_{\nu}(\tau)\mathcal{K}_{0,2}^{d, \epsilon}\mathcal{D}_{\tau}U\big)^j(\beta_{\nu}(\tau)\mathcal{K}_{0,2}^{d, \epsilon}f)\big\|_{S_1}\\
&\lesssim \beta_{\nu}(\tau)\int_{\epsilon\tau}^{\tau}\beta_{\nu}(\sigma_1)\int_{\epsilon\tau}^{\sigma_1}\beta_{\nu}(\sigma_2)\ldots\int_{\epsilon\tau}^{\sigma_{j-1}}\beta_{\nu}(\sigma_j)\big\|f(\sigma_j, \cdot)\big\|_{S_1}\,d\sigma_j\ldots d\sigma_1\\
&\lesssim \beta_{\nu}(\tau)\frac{\epsilon^{-j}}{j!}\sup_{\sigma\sim \tau}\big\|f(\sigma, \cdot)\big\|_{S_1}
\end{align*}
We finally get the desired conclusion of the lemma 
\begin{align*}
&\sup_{\tau\geq \tau_0}\frac{\lambda(\tau)}{\lambda(\tau_0)}\left\langle\log\frac{\lambda(\tau)}{\lambda(\tau_0)}\right\rangle^{1+\frac{\kappa}{2}}\big\|\mathcal{D}_{\tau}U\big(\beta_{\nu}(\tau)\mathcal{K}_{0,2}^{d, \epsilon}\mathcal{D}_{\tau}U\big)^j(\beta_{\nu}(\tau)\mathcal{K}_{0,2}^{d, \epsilon}f)\big\|_{S_1}\\&\lesssim_\epsilon  \frac{\epsilon^{-j}}{j!}\sup_{\tau\geq \tau_0}\frac{\lambda(\tau)}{\lambda(\tau_0)}\left\langle\log\frac{\lambda(\tau)}{\lambda(\tau_0)}\right\rangle^{1+\frac{\kappa}{2}}\big\|f(\tau, \cdot)\big\|_{S_1}
\end{align*}
\end{proof}

We can now conclude the bound for $\Phi_1^k$ by bounding the terms $A$ - $E$ from before: 
\\

{\it{Bound for $A$.}} From preceding lemma, we have 
\begin{align*}
\sup_{\tau\geq \tau_0}\frac{\lambda(\tau)}{\lambda(\tau_0)}\left\langle\log\frac{\lambda(\tau)}{\lambda(\tau_0)}\right\rangle^{1+\frac{\kappa}{2}}\big\|(\mathcal{D}_{\tau}U Af)(\tau, \cdot)\big\|_{S_1}&\lesssim \frac{\epsilon^{-n}}{n!} \sup_{\tau\geq \tau_0}\frac{\lambda(\tau)}{\lambda(\tau_0)}\left\langle\log\frac{\lambda(\tau)}{\lambda(\tau_0)}\right\rangle^{1+\frac{\kappa}{2}}\big\|f(\tau, \cdot)\big\|_{S_1}\\
&\ll \epsilon^{n}e^{\epsilon^{-2}}\sup_{\tau\geq \tau_0}\frac{\lambda(\tau)}{\lambda(\tau_0)}\left\langle\log\frac{\lambda(\tau)}{\lambda(\tau_0)}\right\rangle^{1+\frac{\kappa}{2}}\big\|f(\tau, \cdot)\big\|_{S_1}
\end{align*}
provided $n$ is sufficiently large in relation to a fixed chosen $\epsilon$. 
\\

{\it{Bound for $B$.}} In light of identity \eqref{eq:Bequation}, we find 
\begin{align*}
&\sup_{\tau\geq \tau_0}\frac{\lambda(\tau)}{\lambda(\tau_0)}\left\langle\log\frac{\lambda(\tau)}{\lambda(\tau_0)}\right\rangle^{1+\frac{\kappa}{2}}\big\|(\mathcal{D}_{\tau}UBf)(\tau, \cdot)\big\|_{S_1}\\&\lesssim \big[\sum_{k=1}^{n-1}\frac{\epsilon^{-k}}{k!}\frac{1}{|\log\epsilon|^{\gamma(\kappa)}}\frac{\epsilon^{-(n-k-1)}}{(n-k-1)!}\\
&\hspace{1cm} + \sum_{k=1}^{n-1}\sum_{0\leq j\leq n-k-2}\frac{\epsilon^{-k}}{k!}\frac{\epsilon^{-j}}{j!}\frac{1}{|\log\epsilon|^{\gamma(\kappa)\cdot(n-k-j)}}\big]\sup_{\tau\geq \tau_0}\frac{\lambda(\tau)}{\lambda(\tau_0)}\left\langle\log\frac{\lambda(\tau)}{\lambda(\tau_0)}\right\rangle^{1+\frac{\kappa}{2}}\big\|f(\tau, \cdot)\big\|_{S_1}\\
&\ll \frac{1}{|\log\epsilon|^{\gamma(\kappa)\cdot n}}e^{\epsilon^{-2}}\sup_{\tau\geq \tau_0}\frac{\lambda(\tau)}{\lambda(\tau_0)}\left\langle\log\frac{\lambda(\tau)}{\lambda(\tau_0)}\right\rangle^{1+\frac{\kappa}{2}}\big\|f(\tau, \cdot)\big\|_{S_1}\\
\end{align*}
provided $\epsilon$ is sufficiently small and $n$ large enough. 
\\

{\it{Bound for $C, D, E$.}}The term $C$ is similar in light of relation \eqref{eq:Cequation}, as is term $E$. Finally, for the term $D$, we can bound it by 
\begin{align*}
&\sup_{\tau\geq \tau_0}\frac{\lambda(\tau)}{\lambda(\tau_0)}\left\langle\log\frac{\lambda(\tau)}{\lambda(\tau_0)}\right\rangle^{1+\frac{\kappa}{2}}\big\|Df(\tau, \cdot)\big\|_{S_1}\\&\lesssim\left[\sum_{j=1}^n \frac{1}{|\log\epsilon|^{\gamma(\kappa)\cdot j}}\frac{1}{|\log\epsilon|^{\gamma(\kappa)\cdot(n-j)}}e^{\epsilon^{-2}}\right]\sup_{\tau\geq \tau_0}\frac{\lambda(\tau)}{\lambda(\tau_0)}\left\langle\log\frac{\lambda(\tau)}{\lambda(\tau_0)}\right\rangle^{1+\frac{\kappa}{2}}\big\|f(\tau, \cdot)\big\|_{S_1}\\&\leq \frac{1}{|\log\epsilon|^{\gamma(\kappa)\cdot n}}e^{\epsilon^{-2}}\sup_{\tau\geq \tau_0}\frac{\lambda(\tau)}{\lambda(\tau_0)}\left\langle\log\frac{\lambda(\tau)}{\lambda(\tau_0)}\right\rangle^{1+\frac{\kappa}{2}}\big\|f(\tau, \cdot)\big\|_{S_1}
\end{align*}
for any $\gamma_2<\gamma_1$, provided $\epsilon<1$ and $n$ is sufficiently large. 
\\

The proposition is an immediate consequence of the preceding bounds.  
\end{proof}

Combining the preceding proposition and Lemma~\ref{lem:ndsmall}, and also keeping Proposition~\ref{prop:nonlocallinear1} in mind, we obtain upon using the splitting 
\[
\mathcal{K}_0 = \mathcal{K}_0^d + \mathcal{K}_0^{nd}. 
\]
the following key 
\begin{corollary}\label{cor:iteratedKbound} For $\epsilon>0$ small enough, and $\tau_0 = \tau_0(\epsilon)$ large enough, the operator $\Phi$ defined in \eqref{eq:Phidef}, satisfies the bound 
\begin{align*}
\sup_{\tau\geq \tau_0}\frac{\lambda(\tau)}{\lambda(\tau_0)}\left\langle\log\frac{\lambda(\tau)}{\lambda(\tau_0)}\right\rangle^{1+\frac{\kappa}{2}}\big\|(\Phi^kf)(\tau, \cdot)\big\|_{S_1}\leq \frac{1}{|\log\epsilon|^{\gamma(\kappa)\cdot n}}e^{\epsilon^{-2}}\sup_{\tau\geq \tau_0}\frac{\lambda(\tau)}{\lambda(\tau_0)}\left\langle\log\frac{\lambda(\tau)}{\lambda(\tau_0)}\right\rangle^{1+\frac{\kappa}{2}}\big\|f(\tau, \cdot)\big\|_{S_1}\\
\end{align*}
\end{corollary}

\section{Convergence of the iterative scheme and the proof of Theorem~\ref{thm:Main}}

Let $(\epsilon_0, \epsilon_1)$ be functions of the variable $r\geq 0$ and with $(\epsilon_0, \epsilon_1)\in H^4_{\R^2}\times H^3_{\R^2}$ and sufficiently small, i. e. 
\[
\big\|(\epsilon_0, \epsilon_1)\big\|_{H^4_{\R^2}\times H^3_{\R^3}}<\delta_0. 
\]
We interpret these functions as data at time $\tau = \tau_0$, corresponding in the old coordinates to $t = t_0$. We shall pick $\tau_0$ sufficiently large(independently of $\delta_0$, of course), which means picking $t_0$ sufficiently small. Write 
\[
\epsilon_0 = \phi(\mathcal{D}\epsilon_0) + c_0\phi_0(R),\,\epsilon_1 = \phi(\mathcal{D}\epsilon_1) + c_1\phi_0(R),
\]
where we think of $\epsilon_1 = \partial_{\tau}\epsilon|_{\tau = \tau_0}$. Then set 
\[
\xb_0(\xi): = \langle\mathcal{D}\epsilon_0, \phi(R,\xi)\rangle_{L^2_{R\,dR}},\,\xb_1(\xi): = \langle \mathcal{D}\epsilon_1, \phi(R,\xi)\rangle_{L^2_{R\,dR}} - 2\beta_{\nu}(\tau_0)(\xi\partial_{\xi})\xb_0(\xi).
\]
One easily checks(see Lemma 9.1 in \cite{KST2}) that having fixed.a large $\tau_0$, we then have 
\[
\big\|\big(\xb_0, \xb_1\big)\big\|_{S_0\times S_1}\lesssim_{\tau_0}\delta_0,
\]
and hence can be made arbitrarily small. Similarly, we get 
\[
\big|c_0\big| + \big|c_1\big|\lesssim\delta_0. 
\]
We shall now prove the following technical result, which immediately yields Theorem~\ref{thm:Main}:
\begin{proposition}\label{prop:iteration} The system consisting of \eqref{ctau ODE}, \eqref{eq on Fourier side final} admits a solution pair 
\[
\big(\xb(\tau, \xi), c(\tau)\big),
\]
satisfying the initial conditions $\big(\xb(\tau_0, \xi),\,\mathcal{D}_{\tau}\xb(\tau_0,\xi)\big) = \big(\xb_0, \xb_1\big)$, $(c(\tau_0), c'(\tau_0)) = (c_0, c_1)$, 
and such that $\xb(\tau, \xi)$ is admissible, i. e. allows a representation \eqref{eq:admissiblexb} satisfying the bounds
\begin{align*}
&\sum_{\pm}\tau\big|a_{kj}^{\pm}(\tau)\big| + \sum_{\pm}\tau\big|(a_{kj}^{\pm})'(\tau)\big| + \tau\big\|(|b| + |\xi^{\frac12}c| + |\xi^{\frac12}d|)(\tau, \xi)\big\|_{L_{\xi}^{\infty}}\\&\hspace{6cm} + \tau\big\|\langle\xi\rangle^{-\frac12}\mathcal{D}_{\tau}b(\tau, \xi)\big\|_{L_{\xi}^{\infty}} + \big\|x_{good}(\tau, \xi)\big\|_{S_0} + \big\|\mathcal{D}_{\tau}x_{good}(\tau, \xi)\big\|_{S_1}\\
&\lesssim_{\tau_0} \delta_0\frac{\lambda(\tau_0)}{\lambda(\tau)}\left\langle\log\frac{\lambda(\tau)}{\lambda(\tau_0)}\right\rangle^{-1-\frac{\kappa}{2}},
\end{align*}
while the resonant part satisfies the bound
\begin{align*}
\big|c(\tau)\big| + \tau\big|c'(\tau)\big|\lesssim_{\tau_0} \delta_0\tau^2\frac{\lambda(\tau_0)}{\lambda(\tau)}\left\langle\log\frac{\lambda(\tau)}{\lambda(\tau_0)}\right\rangle^{-1-\frac{\kappa}{2}}.
\end{align*}
Here as usual the actual perturbation $\veps(\tau, R)$ is then constructed via 
\[
\veps(\tau, R) = \phi(\mathcal{D}\veps) + c(\tau)\phi_0(R), \,\mathcal{D}\veps(\tau, R) = \int_0^\infty \xb(\tau, \xi)\phi(R, \xi)\tilde{\rho}(\xi)\,d\xi.
\]
\end{proposition}
\begin{proof} We shall obtain the solution $\big(\xb(\tau, \xi), c(\tau)\big)$ as the limit of a sequence of iterates $\big(\xb_j(\tau, \xi), c_j(\tau)\big)$, $j\geq 1$. Call the corresponding function $\veps_j$. Specifically, we set
\[
\big(\xb_{j}(\tau, \xi), c_{j}(\tau)\big) = \big(\xb_{j-1}(\tau, \xi), c_{j-1}(\tau)\big) + \big(\triangle\xb_{j}(\tau, \xi), \triangle c_{j}(\tau)\big),\,j\geq1,
\]
The zeroth iterate $\big(\xb_{0}(\tau, \xi), c_{0}(\tau)\big)$ is given by the solutions of the linear homogeneous problems 
\begin{align*}
 -\left(\calD^{2}_{\tau}+\frac{\lambda'}{\lambda}\calD_{\tau}+\xi\right)\xb_0(\tau, \cdot) = 0,\,\big(\xb_0(\tau_0,0), \mathcal{D}_{\tau}\xb_0(\tau_0,\cdot)\big) = \big(\xb_0, \xb_1\big),
\end{align*}
\begin{align*}
 \left(\partial_{\tau}+\frac{\lambda'}{\lambda}\right)^{2}c(\tau)+\frac{\lambda'}{\lambda}\left(\partial_{\tau}+\frac{\lambda'}{\lambda}\right)c_0(\tau) = 0,\,(c_0(\tau_0), c_0'(\tau_0)) = (c_0, c_1)
\end{align*}
Then in light of Proposition~\ref{prop:PropagatorEnergy}, $\xb_0$ is admissible and in fact of the form $x_{good}(\tau, \xi)$ (referring to the decomposition \eqref{eq:admissiblexb}), and also satisfies the bound in the statement of Proposition~\ref{prop:iteration}. Similarly, the required bound for $c_0(\tau)$ follows by taking advantage of the section after Proposition~\ref{prop:PropagatorEnergy}.
\\

To conclude the proof, we now have to consider the higher order corrections $\big(\triangle\xb_{j}(\tau, \xi), \triangle c_{j}(\tau)\big)$, and infer a priori bounds and in fact exponential decay for these in the relevant norms. This of course requires passing to the difference equations in 
\eqref{ctau ODE}, \eqref{eq on Fourier side final}. 
\\

Start with the $j$-th correction $\big(\triangle\xb_{j}(\tau, \xi), \triangle c_{j}(\tau)\big)$. Then $\triangle c_{j}$ contributes to $\triangle c_{j+1}$ via $\lambda^{-2}(\tau)\triangle n_j$ in the difference equation associated to  \eqref{ctau ODE}, and a smallness gain $\gamma(\tau_0)$ follows from Lemma~\ref{lem:nbound}. Similarly, considering the contribution of $\triangle c_{j}$ to $\triangle \xb_{j+1}$ and recalling our splitting of the $N(\veps)$ in \eqref{eq:epsequation1} into a term $N_2(\veps)$ linear in $\veps$ as well as a non-linear term $N_1(\veps)$, we easily get smallness for the difference term associated to $N_1(\veps)$, provided one of the factors is $\triangle c_j(\tau)\phi_0(R)$, as follows from Proposition~\ref{prop:resonantnonlin}, Proposition~\ref{prop:mulitlin1}. However, in light of Proposition~\ref{prop:resonantlin} we do not gain smallness when substituting $\triangle c_j(\tau)\phi_0(R)$ for $\veps$ in the difference term associated to $N_2(\veps)$, and instead we only recover an a priori bound in terms of $\big(\triangle\xb_{j-1}(\tau, \xi), \triangle c_{j-1}(\tau)\big)$. Further, $\triangle c_{j}$ contributes to $\triangle \xb_{j+1}$ by substituting $\triangle c_j(\tau)\phi_0(R)$ for $\veps$ in the terms $\tilde{N}_2(\veps)$ displayed in Proposition~\ref{prop:anotherresonantlin}, where we get bounds with smallness gains except for one term, which as far as $\calD{\tau}\triangle \xb_{j+1}$ is concerned gains smallness after re-iteration on account of Lemma~\ref{lem:reiterate2}. 
\\

Next, $\triangle\xb_{j}$ contributes to $\triangle c_{j+1}(\tau)$ via the (difference)source term $h(\tau)$ in \eqref{ctau ODE}, which results in a representation formula \eqref{eq:cexplicit} and a bound without a smallness gain, see Lemma~\ref{lem:hbound}. On the other hand, in light of Proposition~\ref{prop:mulitlin1}, Proposition~\ref{prop:nonlocallinear1}, all contributions of $\triangle\xb_{j}$ to $\triangle\xb_{j+1}$ gain smallness {\it{except}} those involving the nonlocal linear terms using $\mathcal{D}_{\tau}\triangle\xb_j$ as input. 
\\

We now consider the effect of {\it{going one step deeper into the iteration}}, i. e. passing to $\big(\triangle\xb_{j+2}(\tau, \xi), \triangle c_{j+2}(\tau)\big)$. More precisely, we only need to consider those contributions to the next iterate which {\it{in and of themselves}}, i. e. without taking into account the fact that they are defined in terms of a previous iterate, do not gain any smallness, see the figure. 
\\

As far as the contributions of $\triangle\xb_{j+1}$ are concerned, this involves the contribution to $\triangle c_{j+2}$ via $h$ in \eqref{ctau ODE}, and the contribution to $\triangle x_{j+2}$ via the nonlocal linear terms involving $\mathcal{D}_{\tau}\triangle x_{j+1}$. Now if in turn $\triangle\xb_{j+1}$ arose from $\triangle c_{j}$ without gaining smallness via $N_2(\veps)$, or by one of the terms $\tilde{N}_2(\veps)$ from Proposition~\ref{prop:anotherresonantlin} without possibly any bound, as discussed above, then the {\it{combination}} of the two iterative steps leading to $\triangle c_{j+2}$ {\it{does gain}} smallness, on account of Lemma~\ref{lem:ctocdelicate}, Lemma~\ref{lem:ctocdelicate1}.
\\
Similarly, the {\it{combination}} of the two iterative steps leading to $\triangle x_{j+2}$ also leads to smallness as far as $\calD_{\tau}\triangle x_{j+2}$ is concerned, on account of the last part of Proposition~\ref{prop:resonantlin}, Lemma~\ref{lem:reiterate2}.
\\
Still considering  the contributions of $\triangle\xb_{j+1}$, assume now that $\triangle\xb_{j+1}$ arose from $\triangle\xb_j$ via a source term not gaining smallness, which means a nonlocal term linear term involving $\mathcal{D}_{\tau}\triangle\xb_j$, see Proposition~\ref{prop:nonlocallinear1}. Then the contribution to $\triangle \xb_{j+2}$ which does not gain smallness is the one involving the non-local operator and $\mathcal{D}_{\tau}\triangle\xb_{j+1}$, and so this later case corresponds to twofold application of the operator $\Phi$ defined in \eqref{eq:Phidef}.
\\

The upshot of the preceding discussion is that threefold application of the iterative step yields  a smallness gain (in fact, a negative power of $\tau_0$) except for those terms involving threefold application of $\Phi$. However, for the latter, we can use Proposition~\ref{prop:Kditeratedsmallness}, which gives smallness after sufficiently many iterations. We can now infer from the preceding discussion the following final 
\begin{proposition}\label{prop:iteratesmallness} Given $\delta>0$, there is $k = k(\delta)\in \N$ as well as $\tau_0 = \tau_0(\delta)>0$ sufficiently large and $\delta_* = \delta_*(\delta,\tau_0)>0$ sufficiently small, such that restricting $\tau\geq\tau_0$, $\delta_0<\delta_*$ we have the following: denoting by $\Vert\xb\Vert_{S_3}$ the infimum of the sum of norms displayed after \eqref{eq:admissiblexb} over all decompositions as in \eqref{eq:admissiblexb} for an admissible $\xb$, we have 
\begin{align*}
&\sup_{\tau\geq \tau_0}\frac{\lambda(\tau)}{\lambda(\tau_0)}\left\langle\log\frac{\lambda(\tau)}{\lambda(\tau_0)}\right\rangle^{1+\frac{\kappa}{2}}\left[\Vert\triangle \xb_{j+k}(\tau, \cdot)\Vert_{S_3} + \tau^{-2}\big|\triangle c_{j+k}(\tau)\big| + \tau^{-1}\big|\triangle c_{j+k}'(\tau)\big|\right]\\
&\leq \delta \sup_{\tau\geq \tau_0}\frac{\lambda(\tau)}{\lambda(\tau_0)}\left\langle\log\frac{\lambda(\tau)}{\lambda(\tau_0)}\right\rangle^{1+\frac{\kappa}{2}}\left[\Vert\triangle \xb_{j}(\tau, \cdot)\Vert_{S_3}+ \tau^{-2}\big|\triangle c_j(\tau)\big| + \tau^{-1}\big|\triangle c_j'(\tau)\big|\right]\\
\end{align*}
\end{proposition}
In turn Proposition~\ref{prop:iteration} is an immediate consequence of the preceding proposition.

\end{proof}

In the figure below, numbers refer to lemmas or propositions. Also, the letter 's' means that smallness is gained. 
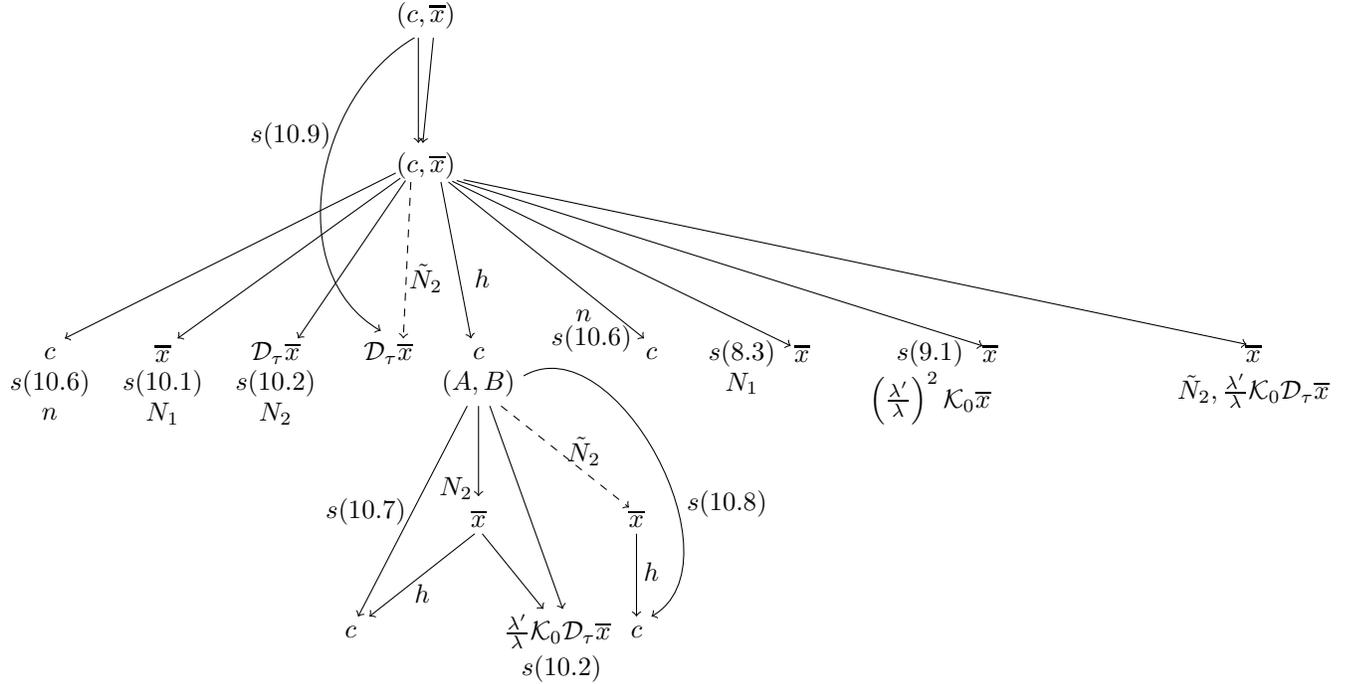
\begin{center}
\begin{tikzpicture}
 \draw (0,0) node {$(c,\xb)$};
 \draw (0,2) node {$(c,\xb)$};
 \draw (-1.8,0.4) node {$s(10.9)$};
 \draw [<-] (-0.1,0.3) to (-0.1,1.7);
 \draw [<-] (-0.04,0.3) to (0.1,1.7);
 \draw (-5,-2.5) node {$c$};
 \draw (-5,-2.9) node {$s(10.6)$};
 \draw (-5,-3.3) node {$n$};
 \draw [->] (-0.4,-0.1) to (-4.8,-2.3);  
 \draw (-3.5,-2.5) node {$\xb$};
 \draw (-3.5,-2.9) node {$s(10.1)$};
 \draw (-3.5,-3.3) node {$N_{1}$};
 \draw [->] (-0.34,-0.17) to (-3.3, -2.3);
 \draw (-2, -2.5) node {$\calD_{\tau}\xb$};
 \draw (-2,-2.9) node {$s(10.2)$};
 \draw (-2,-3.3) node {$N_{2}$};
 \draw [->] (-0.27, -0.2) to (-1.7, -2.3);
 \draw (-0.5,-2.5) node {$\calD_{\tau}\xb$};
 \draw (0,-1.5) node {$\tilde{N}_{2}$};
 \draw [->] (-0.15,1.7) to [out=210,in=150] (-0.6,-2.2); 
 \draw [dashed,->] (-0.2, -0.22) to (-0.3, -2.3);
 \draw (0.7,-2.5) node {$c$};
 \draw (0.75,-1.5) node {$h$};
 \draw [->] (0.2,-0.22) to (0.6,-2.3); 
 \draw (3,-2.5) node {$c$};
 \draw (2.1,-2) node {$n$};
 \draw (2.2,-2.3) node {$s(10.6)$};
 \draw [->] (0.3, -0.22) to (2.9,-2.3);
 \draw (5,-2.5) node {$\xb$};
 \draw (4.2,-2.5) node {$s(8.3)$};
 \draw (4.2,-2.9) node {$N_{1}$};
 \draw [->] (0.35,-0.21) to (4.85,-2.4);
 \draw (7.5,-2.5) node {$\xb$};
 \draw (6.7,-2.5) node {$s(9.1)$};
 \draw (6.7,-3.1) node {$\left(\frac{\lambda'}{\lambda}\right)^{2}\calK_{0}\xb$};
 \draw [->] (0.41,-0.2) to (7.4,-2.4);
 \draw (11,-2.5) node {$\xb$};
 \draw (11,-3) node {$\tilde{N}_{2},\frac{\lambda'}{\lambda}\calK_{0}\calD_{\tau}\xb$};
 \draw [->] (0.5,-0.19) to (10.91,-2.38);
 \draw (0.7,-2.9) node {$(A,B)$};
 \draw (0.7,-4.7) node {$\xb$};
 \draw (0.4,-4.3) node {$N_{2}$};
 \draw [->] (0.7,-3.2) to (0.7,-4.4);
 \draw (2.8,-4.7) node {$\xb$};
 \draw (2.1,-3.8) node {$\tilde{N}_{2}$};
 \draw [dashed,->] (1,-3.2) to (2.7,-4.55);
 \draw (2.8,-6.2) node {$c$};
 \draw (4,-4.5) node {$s(10.8)$};
 \draw (3,-5.4) node {$h$};
 \draw [->] (1.3,-2.8) to [out=30,in=30] (3,-6);
 \draw [->] (2.8,-4.9) to (2.8,-6);
 \draw (1.75,-6.2) node {$\frac{\lambda'}{\lambda}\calK_{0}\calD_{\tau}\xb$};
 \draw (1.8,-6.7) node {$s(10.2)$};
 \draw [->] (0.75,-4.9) to (1.56,-5.9);
 \draw [->] (0.85,-3.2) to (1.82,-5.9);
 \draw (-1,-6.2) node {$c$};
 \draw (-0.8,-4.6) node {$s(10.7)$};
 \draw (-0.05,-5.7) node {$h$};
 \draw [->] (0.65,-4.9) to (-0.75,-6);
 \draw [->] (0.55,-3.2) to (-0.9,-6);
\end{tikzpicture}
\end{center}

\bibliographystyle{plain}
\bibliography{bibWM}

\begin{thebibliography}{10}

\bibitem{BeKrTa}
Ioan Bejenaru, Joachim Krieger, and Daniel Tataru.
\newblock A codimension-two stable manifold of near soliton equivariant wave
  maps.
\newblock {\em Anal. PDE}, 6(4):829--857, 2013.

\bibitem{BizWM}
Piotr Bizo\'n, Tadeusz Chmaj, and Zbis\l~aw Tabor.
\newblock Formation of singularities for equivariant {$(2+1)$}-dimensional wave
  maps into the 2-sphere.
\newblock {\em Nonlinearity}, 14(5):1041--1053, 2001.

\bibitem{BuKr}
Stefano. {Burzio} and Joachim. {Krieger}.
\newblock {Type II blow up solutions with optimal stability properties for the
  critical focussing nonlinear wave equation on $\bbR^{3+1}$}.
\newblock {\em ArXiv e-prints}, September 2017.

\bibitem{ChShadi}
Demetrios Christodoulou and A.~Shadi Tahvildar-Zadeh.
\newblock On the regularity of spherically symmetric wave maps.
\newblock {\em Comm. Pure Appl. Math.}, 46(7):1041--1091, 1993.

\bibitem{Cote}
Rapha\"el C\^ote.
\newblock On the soliton resolution for equivariant wave maps to the sphere.
\newblock {\em Comm. Pure Appl. Math.}, 68(11):1946--2004, 2015.

\bibitem{DoMHKS}
Roland Donninger, Min Huang, Joachim Krieger, and Wilhelm Schlag.
\newblock Exotic blowup solutions for the {$u^5$} focusing wave equation in
  {$\Bbb{R}^3$}.
\newblock {\em Michigan Math. J.}, 63(3):451--501, 2014.

\bibitem{GaoKr}
Can Gao and Joachim Krieger.
\newblock Optimal polynomial blow up range for critical wave maps.
\newblock {\em Commun. Pure Appl. Anal.}, 14(5):1705--1741, 2015.

\bibitem{IsenLieb}
James Isenberg and Steven~L. Liebling.
\newblock Singularity formation in {$2+1$} wave maps.
\newblock {\em J. Math. Phys.}, 43(1):678--683, 2002.

\bibitem{KlMa2}
Sergiu Klainerman and Matei Machedon.
\newblock Smoothing estimates for null forms and applications.
\newblock {\em Duke Math. J.}, 81(1):99--133 (1996), 1995.
\newblock A celebration of John F. Nash, Jr.

\bibitem{KlMa1}
Sergiu Klainerman and Matei Machedon.
\newblock On the regularity properties of a model problem related to wave maps.
\newblock {\em Duke Math. J.}, 87(3):553--589, 1997.

\bibitem{Krieger04}
Joachim Krieger.
\newblock Global regularity of wave maps from {$\bold R^{2+1}$} to {$H^2$}.
  {S}mall energy.
\newblock {\em Comm. Math. Phys.}, 250(3):507--580, 2004.

\bibitem{CondBlow}
Joachim. {Krieger}.
\newblock {On stability of type II blow up for the critical NLW on
  $\mathbb{R}^{3+1}$}.
\newblock {\em to appear on Memoirs of the AMS}, May 2017.

\bibitem{KS0}
Joachim Krieger and Wilhelm Schlag.
\newblock {\em Concentration compactness for critical wave maps}.
\newblock EMS Monographs in Mathematics. European Mathematical Society (EMS),
  Z\"urich, 2012.

\bibitem{KS1}
Joachim Krieger and Wilhelm Schlag.
\newblock Full range of blow up exponents for the quintic wave equation in
  three dimensions.
\newblock {\em J. Math. Pures Appl. (9)}, 101(6):873--900, 2014.

\bibitem{KST2}
Joachim Krieger, Wilhelm Schlag, and Daniel Tataru.
\newblock Renormalization and blow up for charge one equivariant critical wave
  maps.
\newblock {\em Invent. Math.}, 171(3):543--615, 2008.

\bibitem{KST1}
Joachim Krieger, Wilhelm Schlag, and Daniel Tataru.
\newblock Renormalization and blow up for the critical {Y}ang-{M}ills problem.
\newblock {\em Adv. Math.}, 221(5):1445--1521, 2009.

\bibitem{KST}
Joachim Krieger, Wilhelm Schlag, and Daniel Tataru.
\newblock Slow blow-up solutions for the {$H^1(\Bbb R^3)$} critical focusing
  semilinear wave equation.
\newblock {\em Duke Math. J.}, 147(1):1--53, 2009.

\bibitem{MeRaRo2}
Frank Merle, Pierre Rapha\"el, and Igor Rodnianski.
\newblock Blow up dynamics for smooth equivariant solutions to the energy
  critical {S}chr\"odinger map.
\newblock {\em C. R. Math. Acad. Sci. Paris}, 349(5-6):279--283, 2011.

\bibitem{MeRaRo1}
Frank Merle, Pierre Rapha\"el, and Igor Rodnianski.
\newblock Blowup dynamics for smooth data equivariant solutions to the critical
  {S}chr\"odinger map problem.
\newblock {\em Invent. Math.}, 193(2):249--365, 2013.

\bibitem{Perelman}
Galina Perelman.
\newblock Blow up dynamics for equivariant critical {S}chr\"odinger maps.
\newblock {\em Comm. Math. Phys.}, 330(1):69--105, 2014.

\bibitem{RaRod}
Pierre Rapha\"el and Igor Rodnianski.
\newblock Stable blow up dynamics for the critical co-rotational wave maps and
  equivariant {Y}ang-{M}ills problems.
\newblock {\em Publ. Math. Inst. Hautes \'Etudes Sci.}, 115:1--122, 2012.

\bibitem{RodSter}
Igor Rodnianski and Jacob Sterbenz.
\newblock On the formation of singularities in the critical {${\rm O}(3)$}
  {$\sigma$}-model.
\newblock {\em Ann. of Math. (2)}, 172(1):187--242, 2010.

\bibitem{StTat}
Jacob Sterbenz and Daniel Tataru.
\newblock Regularity of wave-maps in dimension {$2+1$}.
\newblock {\em Comm. Math. Phys.}, 298(1):231--264, 2010.

\bibitem{Struwe1}
Michael Struwe.
\newblock Equivariant wave maps in two space dimensions.
\newblock {\em Comm. Pure Appl. Math.}, 56(7):815--823, 2003.
\newblock Dedicated to the memory of J\"urgen K. Moser.

\bibitem{Struwe2}
Michael Struwe.
\newblock Radially symmetric wave maps from {$(1+2)$}-dimensional {M}inkowski
  space to general targets.
\newblock {\em Calc. Var. Partial Differential Equations}, 16(4):431--437,
  2003.

\bibitem{Tao1}
Terence Tao.
\newblock Global regularity of wave maps. {I}. {S}mall critical {S}obolev norm
  in high dimension.
\newblock {\em Internat. Math. Res. Notices}, (6):299--328, 2001.

\bibitem{Tao2}
Terence Tao.
\newblock Global regularity of wave maps. {II}. {S}mall energy in two
  dimensions.
\newblock {\em Comm. Math. Phys.}, 224(2):443--544, 2001.

\bibitem{Tao3}
Terence {Tao}.
\newblock {Global regularity of wave maps III. Large energy from $\R^{1+2}$ to
  hyperbolic spaces}.
\newblock {\em ArXiv e-prints}, May 2008.

\bibitem{Tao4}
Terence {Tao}.
\newblock {Global regularity of wave maps IV. Absence of stationary or
  self-similar solutions in the energy class}.
\newblock {\em ArXiv e-prints}, June 2008.

\bibitem{Tao5}
Terence {Tao}.
\newblock {Global regularity of wave maps V. Large data local wellposedness and
  perturbation theory in the energy class}.
\newblock {\em ArXiv e-prints}, August 2008.

\bibitem{Tao6}
Terence {Tao}.
\newblock {Global regularity of wave maps VI. Abstract theory of minimal-energy
  blowup solutions}.
\newblock {\em ArXiv e-prints}, June 2009.

\bibitem{Tao7}
Terence {Tao}.
\newblock {Global regularity of wave maps VII. Control of delocalised or
  dispersed solutions}.
\newblock {\em ArXiv e-prints}, August 2009.

\bibitem{Tataru1}
Daniel Tataru.
\newblock On global existence and scattering for the wave maps equation.
\newblock {\em Amer. J. Math.}, 123(1):37--77, 2001.

\bibitem{Tataru2}
Daniel Tataru.
\newblock The wave maps equation.
\newblock {\em Bull. Amer. Math. Soc. (N.S.)}, 41(2):185--204, 2004.

\bibitem{Tataru3}
Daniel Tataru.
\newblock Rough solutions for the wave maps equation.
\newblock {\em Amer. J. Math.}, 127(2):293--377, 2005.

\end{thebibliography}

\bigskip

\centerline{\scshape Joachim Krieger }
\medskip
{\footnotesize
 \centerline{B\^{a}timent des Math\'ematiques, EPFL}
\centerline{Station 8, 
CH-1015 Lausanne, 
  Switzerland}
  \centerline{\email{joachim.krieger@epfl.ch}}
} 

\medskip
\centerline{\scshape Shuang Miao}
\medskip
{\footnotesize
 \centerline{B\^{a}timent des Math\'ematiques, EPFL}
\centerline{Station 8, 
CH-1015 Lausanne, 
  Switzerland}
  \centerline{\email{shuang.miao@epfl.ch}}
} 

\end{document}